\documentclass[10pt, oneside]{amsart}

\usepackage[utf8]{inputenc}
\usepackage[margin=1in]{geometry}
\usepackage[mathscr]{euscript}
\usepackage{mathtools,bm,bbm,framed,hyperref,tikz, xcolor,graphicx,subcaption,dsfont,booktabs,pst-node,tikz-cd,amsmath,amssymb, amsthm,cleveref,verbatim}
\newtheorem{theorem}{Theorem}[section]
\newtheorem{lemma}[theorem]{Lemma}
\newtheorem{prop}[theorem]{Proposition}
\newtheorem{cor}[theorem]{Corollary}
\newtheorem{conj}[theorem]{Conjecture}

\newtheorem{definition}[theorem]{Definition}

\newtheorem{example}[theorem]{Example}
\newtheorem{remark}[theorem]{Remark}
\numberwithin{equation}{section}

\usetikzlibrary{positioning,arrows}
\usetikzlibrary{decorations.markings}
\tikzstyle{circledvertex} = [draw, black, shape=circle, minimum size=8pt, inner sep=1pt]
\tikzstyle{invisivertex} = [black, shape=rectangle, minimum size=0pt, inner sep=2pt]
\tikzstyle{point}=[draw, black, fill,shape=circle, minimum size=4pt, inner sep=0pt]
\tikzstyle{label} = [draw, violet, shape=circle, minimum size=8pt, inner sep=1pt]
\tikzstyle{glabel} = [draw, teal, shape=circle, minimum size=4pt, inner sep=1pt]
\tikzstyle{graylabel} = [draw, gray, shape=circle, minimum size=4pt, inner sep=1pt]
\tikzstyle{BVertex}  = [draw, black, fill, shape=circle, minimum size=3pt, inner sep=0pt]

\newcommand{\symm}[1]{\mathfrak{S}_{#1}}
\newcommand{\symmB}[1]{\mathfrak{S}^{\pm}_{#1}}

\newcommand{\Conf}{\mathrm{Conf}}

\newcommand{\TypeAConf}[1]{{X}_{#1}}
\newcommand{\TypeACohomology}[1]{{\mathcal X}_{#1}}
\newcommand{\TypeAEulerianIdempotent}[2]{{E^{\mathfrak{S}_{#1}}_{#2}}}

\newcommand{\TypeBConf}[1]{{Y}_{#1}}
\newcommand{\TypeBCohomology}[1]{{\mathcal Y}_{#1}}
\newcommand{\TypeBEulerianIdempotent}[2]{{E^{\mathfrak{S}^{\pm}_{#1}}_{#2}}}

\newcommand{\PeakConf}[1]{{Z}_{#1}}

\newcommand{\PeakIdempotent}[2]{{E^{\mathcal{P}_{#1}}_{#2}}}

\newcommand{\quadgens}{\mathcal{Q}}
\newcommand{\tvbasis}{\vbasis \cap \prod \quadgens}

\newcommand{\Sol}[1]{\mathrm{Sol}(#1)}
\newcommand{\PeakAlgebra}[1]{{\mathcal P}_{#1}}

\newcommand{\Lie}[1]{{\mathrm{Lie}_{#1}}}

\DeclareMathOperator{\peak}{Peak}

\newcommand{\origin}{{\mathbf{0}}}

\DeclareMathOperator{\Z}{{\mathbb{Z}}}
\DeclareMathOperator{\R}{{\mathbb{R}}}
\DeclareMathOperator{\C}{{\mathbb{C}}}

\DeclareMathOperator{\PP}{{\mathbb{P}}}
\DeclareMathOperator{\SSS}{{\mathbb{S}}}
\DeclareMathOperator{\RP}{{\mathbb{RP}}}
\DeclareMathOperator{\rad}{rad}

\DeclareMathOperator{\kk}{{\mathbf{k}}}
\DeclareMathOperator{\xx}{\mathbf{x}}

\DeclareMathOperator{\bihilb}{{\mathsf{H}}}
\DeclareMathOperator{\symmbihilb}{{\mathcal H}}

\DeclareMathOperator{\B}{\mathcal{B}}

\DeclareMathOperator{\lie}{Lie}

\DeclareMathOperator{\init}{in}
\DeclareMathOperator{\spn}{span}
\DeclareMathOperator{\fieldchar}{char}

\DeclareMathOperator{\A}{\mathcal{A}}

\DeclareMathOperator{\N}{\mathcal{N}}
\DeclareMathOperator{\PPP}{\mathcal{P}}

\DeclareMathOperator{\lat}{\mathcal{L}}

\DeclareMathOperator{\I}{\mathcal{I}}

\DeclareMathOperator{\T}{\mathcal{T}}

\DeclareMathOperator{\U}{\mathcal{U}}
\DeclareMathOperator{\uu}{\mathfrak{u}}

\DeclareMathOperator{\odd}{odd}
\DeclareMathOperator{\oddparts}{\mathrm{Odd}}
\DeclareMathOperator{\even}{even}
\DeclareMathOperator{\evenparts}{\mathrm{Even}}
\DeclareMathOperator{\cyc}{cyc}

\DeclareMathOperator{\VG}{\mathcal{VG}}

\newcommand{\J}{\mathcal{J}}
\newcommand{\gr}{\mathfrak{gr}}

\newcommand{\cU}{\mathcal{U}}
\newcommand{\cLfree}{\mathfrak{L}}

\DeclareMathOperator{\V}{\mathcal{V}}

\DeclareMathOperator{\sol}{Sol}
\DeclareMathOperator{\Des}{Des}
\DeclareMathOperator{\des}{des}
\DeclareMathOperator{\Sym}{Sym}

\newcommand{\slambda}{(\lambda^{+}, \lambda^{-})}

\newcommand{\basismap}{\mathcal{B}}

\DeclareMathOperator{\GGG}{\mathcal{G}}

\DeclareMathOperator{\ubasis}{\mathcal{U}}
\DeclareMathOperator{\vbasis}{\mathcal{V}}
\DeclareMathOperator{\typeabasis}{\mathcal{T}}

\DeclareMathOperator{\fchar}{ch}
\DeclareMathOperator{\ind}{\uparrow}
\DeclareMathOperator{\res}{\downarrow}

\newcommand{\sh}[1]{[#1]}



\title[Configuration spaces and peak representations]{Configuration spaces and peak representations}

\author{Marcelo Aguiar}
\email{maguiar@math.cornell.edu}
\address{Department of Mathematics, Cornell University, Ithaca NY 14853}
\author{Sarah Brauner}
\email{sarahbrauner@gmail.com}
\address{Brown University, Providence, RI, 02912}
\author{Victor Reiner}
\email{reiner@umn.edu}
\address{School of Mathematics, University of Minnesota, Minneapolis MN 55455}

\begin{document}

\begin{abstract}
Within the group algebras of the symmetric and hyperoctahedral groups, one has their descent algebras and families of 
{\it Eulerian} idempotents.  These idempotents are known to generate
group representations with topological interpretations, as the cohomology of configuration spaces of types $A$ and $B$. We provide an analogous cohomological interpretation for the representations generated by idempotents in the {\it peak algebra}, called the {\it peak representations}. We describe the peak representations as sums of {\it Thrall's higher Lie characters}, give Hilbert series and branching rule recursions for them, and discuss a connection to Jordan brackets.
\end{abstract}

\maketitle

\section{Introduction}
\label{intro-section}
This paper develops connections between combinatorial algebras arising in the theory of reflection groups and configuration spaces, an important class of topological space. Previous work of the second author relates Types $A$ and $B$ configuration spaces (defined in equation \eqref{typeAB-configuration-space-definitions} below) to \emph{Solomon's descent algebra} and the famed \emph{Eulerian idempotents} therein \cite{brauner2022eulerian}. Here, we expand the scope of such connections by (1) linking the \emph{peak algebra} of the symmetric group to a configuration space $\PeakConf{n}:= \Conf_n(\R \PP^2 \times \R)$ and (2) describing the cohomology ring of $H^*\PeakConf{n}$ equivariantly in terms of the well-known \emph{higher Lie representations} of Thrall.  

\subsection*{Overview}
Our work studies and relates the cohomology $H^* X=H^*(X,\kk)$ with coefficients in a field $\kk$ for three different topological configuration spaces $X=X_n, Y_n, Z_n$ having large symmetry groups $W$.  For each, the (ungraded) cohomology carries the regular representation of $W$, that is,
$
H^* X \cong \kk W.
$

At the heart of the story is the fact that, for $\kk$ of characteristic zero, the decomposition into $H^i X$ turns out to match a combinatorial direct sum decomposition 
\begin{equation} \label{eq:generalcorrespondance}
H^* X =\bigoplus_i H^i X, \quad \quad \cong \quad \quad
\kk W =\displaystyle \bigoplus_i (\kk W) E_i
\end{equation}
for certain interesting complete families $\{ E_i \}$ 
of {\it orthogonal idempotents} in $\kk W$, meaning  $E_iE_j=\delta_{ij} E_i$ and $1=\sum_i E_i$. Precise statements of the correspondence for $X = \TypeAConf{n}, \TypeBConf{n}$ are given in equations \eqref{type-A-Eulerian-idempotents-give-cohomology-reps} and \eqref{type-B-Eulerian-idempotents-give-cohomology-reps} below; Theorem \ref{peak-idempotents-give-cohomology-reps} then gives our new correspondence for $\PeakConf{n}$.

The first two spaces $X_n, Y_n$ are well-studied: $X_n$ is the {\it ordered configuration space} of $n$ points in $\R^d$ 
with $d \geq 3$ and odd, while $Y_n$ is the {\it $\Z_2$-orbit configuration spaces} for the $\Z_2$-action via $\xx \mapsto -\xx$ on the same $\R^d$:
\begin{equation}
\label{typeAB-configuration-space-definitions}
\begin{aligned}
\TypeAConf{n}&:=\Conf_n\R^d=\{\xx \in (\R^d)^n: x_i \neq x_j \text{ for }1 \leq i < j \leq n\},\\
\TypeBConf{n}&:=\Conf^{\Z_2}_n\R^d=\{\xx \in (\R^d )^n: x_i \neq \pm x_j \text{ for }1 \leq i <j \leq n, \text{ and }x_i \neq 0 \text{ for }1 \leq i \leq n\}.
\end{aligned}
\end{equation}
These spaces have appeared in a wide range of mathematics; perhaps most notably, $\Conf_n(\R^2)$ is the classifying space of the pure Artin braid group \cite{arnold}. 
In the case that $d$ is odd, $H^*\TypeAConf{n}$ and $H^*\TypeBConf{n}$ are isomorphic to the \emph{associated graded Varchenko-Gelfand ring}, reviewed in Section \ref{sec:VGrings}.

Note that $X_n$ has an action of the {\it symmetric group} $W=\symm{n}$ that permutes the coordinates of $\xx=(x_1,\ldots,x_n)$, while $Y_n$ carries an action of the {\it hyperoctahedral group} $W=\symmB{n}$ that permutes and negates coordinates. This implies that their cohomology rings carry representations of $\symm{n}$ and $\symmB{n}$, which will be critical in obtaining the decomposition in equation \eqref{eq:generalcorrespondance}.

One can view $X_n$ and $ Y_n$ as complements of
certain {\it subspace arrangements} derived from the {\it reflection hyperplane arrangements} $\A_W$ associated to the reflection groups $W=\symm{n}, \symmB{n}$: 
\begin{itemize}
    \item $W=\symm{n}$ acts on $V=\R^n$, with action generated by
    reflections through the hyperplanes 
$$
\A_{\symm{n}}=\{x_i = x_j\}_{1 \leq i < j \leq n}.
$$
\item $W=\symmB{n}$ acts on $V=\R^n$, with action generated by reflections
through the hyperplanes 
$$
\A_{\symmB{n}}=\{x_i=0\}_{1 \leq i \leq n} \sqcup  \{x_i = \pm x_j\}_{1 \leq i<j \leq n}.
$$
\end{itemize}
From this viewpoint, letting $V=\R^n$, one can re-express
\begin{equation}
\label{configuration-spaces-reexpressed}
\begin{aligned}
\TypeAConf{n} &=\left( V \otimes_{\R} \R^d \right) \setminus \left( \A_{\symm{n}} \otimes_{\R} \R^d \right),\\
\TypeBConf{n} &=\left( V \otimes_{\R} \R^d \right) \setminus \left( \A_{\symmB{n}} \otimes_{\R} \R^d \right).
\end{aligned}
\end{equation}
General results on {\it graded Varchenko-Gelfand rings}, reviewed in Section \ref{sec:VGrings}, explain why this implies that, for odd $d \geq 3$, the cohomology rings $H^* X_n, H^* Y_n$ are concentrated in cohomological degrees divisible by $d-1$, and as ungraded representations, carry the regular representations $\kk W$ for the associated reflection groups $W=\symm{n},\symmB{n}$ of types $A$ and $B$.  For simplicity and concreteness, we choose $d=3$ from here onward.

The idempotent decompositions of $\kk \symm{n}$ and $\kk \symmB{n}$ come from the well-known {\it type $A$ and $B$ Eulerian idempotents} 
\[ \{ \TypeAEulerianIdempotent{n}{k} \}_{k=0,1,\ldots,n-1} \textrm{ in } \kk \symm{n}, \quad \textrm{ and } \quad 
\{ \TypeBEulerianIdempotent{n}{k} \}_{k=0,1,\ldots,n} \textrm{ in } \kk \symmB{n},\]
defined in work of Reutenauer \cite{reutenauer}, Gerstenhaber--Schack \cite{gerstenhaberschack},  and  Bergeron--Bergeron \cite{bergeronbergeron}. 

The Eulerian idempotents lie within the subalgebra of the group algebra $\kk W$ known as {\it Solomon's descent algebra} $\Sol{W}$ \cite{solomon}, meaning that
when expressed as $\sum_{w \in W} c_w w$, their coefficients $c_w$ depend only upon the Coxeter group {\it descent set} of $w$. Solomon's descent algebra is a fundamental object in algebraic combinatorics, with important connections to 
the theory of the free Lie algebra, representation theory of Coxeter groups,  quassisymmetric functions, hyperplane arrangements, Markov chains, poset homology and Hopf algebras, to name a few. See \cite{aguiar2017topics} and the references therein for a summary.

Work of Hanlon \cite{hanlon}, Sundaram-Welker \cite{sundaramwelker} and Brauner \cite{brauner2022eulerian} gives a precise correspondence between these objects, as the following isomorphisms of group representations: 
\begin{align}
\label{type-A-Eulerian-idempotents-give-cohomology-reps}
H^{2k} X_n&\cong \left(\kk \symm{n} \right) \TypeAEulerianIdempotent{n}{n-1-k} \text{ for }k=0,1,\ldots,n-1,\\
\label{type-B-Eulerian-idempotents-give-cohomology-reps}
H^{2k} Y_n&\cong \left(\kk \symmB{n}\right) \TypeBEulerianIdempotent{n}{n-k} \,\,\,\, \text{ for }k=0,1,\ldots,n.
\end{align}

\subsection*{The peak story} In this paper, we use \eqref{type-A-Eulerian-idempotents-give-cohomology-reps} and \eqref{type-B-Eulerian-idempotents-give-cohomology-reps} as the starting point to give a new decomposition of $\kk \symm{n}$ of the form \eqref{eq:generalcorrespondance}, this time using the space $Z_n \cong \Conf_n(\R \PP^2 \times \R)$ and idempotents in the peak algebra, $\PeakAlgebra{n}$. The peak algebra
is another nested subalgebra 
$$
\PeakAlgebra{n} \subset \Sol{\symm{n}} \subset \kk\symm{n}
$$
whose elements $\sum_{w \in W} c_w w$ have coefficients $c_w$ depending only upon the {\it peak set} of $w=(w_1,\ldots,w_n)$
$$
\peak(w):= \{i : 1\leq i \leq n-1\text{ and } w_{i-1} < w_i > w_{i+1} \} ,
$$
with convention $w_0:=0$.  

The peak set of a permutation was first studied by Stembridge \cite{stembridge1997enriched}, who approached it from the perspective of enriched $P$-partitions. Peaks have since been studied in connection to diverse areas of mathematics such as the generalized Dehn-Sommerville equations 
\cite{aguiar2004peak, aguiar2006peak,bergeron2000noncommutative, billera2003peak} and the
Schubert calculus of isotropic flag manifolds \cite{billey1995schubert, garsiareutenauer, hsiao2002structure}. Algebraic and representation theoretic properties of the peak algebra, as well as connections to quasisymmetric functions, Hopf algebras and Hecke-Clifford algebras have been investigated by many authors, see for example \cite{aguiar2004peak, aguiar2006peak, bergeron2006coloured, li2016representation, novelli2010representation, nyman2003peak, Schocker-peaks}.

Our investigations began from an observation of Aguiar, Bergeron and Nyman \cite{aguiar2004peak} relating the descent algebras $\Sol{\symm{n}}, \Sol{\symmB{n}}$ to $\PeakAlgebra{n}$. Recall that one can express the hyperoctahedral group of all signed permutations as 
$
\symmB{n} = \symm{n} \ltimes \Z_2^n
$
where $\Z_2^n$ is the normal subgroup performing arbitrary sign changes in the coordinates.
The quotient map $\symmB{n} \twoheadrightarrow \symmB{n}/\Z_2^n \cong \symm{n}$
of groups, which forgets the signs in a signed permutation, gives a surjective $\kk$-algebra map 
$
\varphi: \kk \symmB{n} \twoheadrightarrow \kk \symm{n}.
$
In \cite{aguiar2004peak}, it was shown that the peak subalgebra $\PeakAlgebra{n}$ is the image under $\varphi$ of $\Sol{\symmB{n}}$. In other words, $\varphi$ restricts to an algebra surjection
$
\Sol{\symmB{n}} \overset{\varphi}{\twoheadrightarrow} \PeakAlgebra{n}.
$

As a consequence, one can define a family of {\it peak idempotents} $\{ \PeakIdempotent{n}{k}\}_{k=0,1,\ldots,n}$ inside $\PeakAlgebra{n} \subset \kk \symm{n}$ via
$$
\PeakIdempotent{n}{k}:=
\varphi(\TypeBEulerianIdempotent{n}{k}).
$$
The $\{\PeakIdempotent{n}{k}\}$ 
inherit from $\{\TypeBEulerianIdempotent{n}{k}\}$
the property of being a complete system of orthogonal idempotents in $\kk\symm{n}$.
Our first result, proven in Section \ref{sec:invariantring}, relates these peak idempotents to the cohomology of the quotient space $\TypeBConf{n}/\Z_2^n$, where $\Z_2^n$ is the normal subgroup of $\symmB{n}$ consisting of all
sign changes.  That is, we consider
$$
\PeakConf{n}:=\TypeBConf{n}/\Z_2^n = \Conf_n(\left( \R^3 \setminus \{\origin\} \right)/\Z_2)
\cong \Conf_n(\RP^2 \times \R),
$$
the configuration space 
of $n$ ordered points in the quotient of $\R^3 \setminus \{\origin\}$ under the $\Z_2$-action via $\xx \mapsto -\xx$:
$$
\left( \R^3 \setminus \{\origin\} \right)/\Z_2
\,\, \cong \,\, \RP^2 \times \R.
$$

\begin{theorem}
\label{peak-idempotents-give-cohomology-reps}
Let $H^* Z_n$ be the cohomology with coefficients
in a field $\kk$ of 
$$
Z_n=\TypeBConf{n}/\Z_2^n\cong\Conf_n(\RP^2 \times \R) .
$$
\begin{itemize}
\item[(i)] When $\fieldchar(\kk)>2$, the $\kk \symm{n}$-module on the total cohomology is the regular representation:
$$
H^* \PeakConf{n} \cong \kk \symm{n}.
$$
\item[(ii)] When $\fieldchar(\kk)>n$, for  $0 \leq k \leq n$ one has a $\kk \symm{n}$-module isomorphism 
$$
H^{2k} \PeakConf{n} \cong \left( \kk \symm{n} \right) \PeakIdempotent{n}{n-k}.
$$
\item[(iii)] If $\fieldchar(\kk)>2$ then  $H^i\PeakConf{n}=0$
unless $i \equiv 0 \bmod 4$.
If $\fieldchar(\kk)>n$ then $\PeakIdempotent{n}{k}=0$ 
unless $k \equiv n \bmod 2$. 
\end{itemize}
\end{theorem}
\noindent
In the case that $k=0$, Theorem \ref{peak-idempotents-give-cohomology-reps} (iii) recovers work of Aguiar--Nyman--Orellana \cite[Thm 6.2]{aguiar2006peak}, who give formulas for $\PeakIdempotent{n}{0}$ and show that $\varphi(\TypeBEulerianIdempotent{n}{0}) = 0$ when $n$ is odd.

Our proof employs the fact that for $\symmB{n}$-representations $U$, the $\Z_2^n$-fixed space $U^{\Z_2^n}$ is a representation of the quotient group $\symm{n} \cong \symmB{n}/\Z_2^n$. 
Much of 
Theorem~\ref{peak-idempotents-give-cohomology-reps} stems
from analyzing how this applies to the two sides in the $\symmB{n}$-representation isomorphism \eqref{type-B-Eulerian-idempotents-give-cohomology-reps} above.  
Starting from the $\symm{n}$-representation isomorphism
$$
H^* \PeakConf{n} \cong \left( H^* \TypeBConf{n} \right)^{\Z_2^n},
$$
we use a known presentation of the cohomology ring 
$H^* \TypeBConf{n}$ due to
Xicot{\'e}ncatl (explained in Section \ref{sec:VGrings} below), followed by 
a change-of-variables in Section \ref{eigengbasis-subsection} that diagonalizes
the $\Z_2^n$-action on this ring.  

\subsection*{Connection to Thrall's higher Lie characters} In Section \ref{sec:filtration-and-gradings}, we embark upon a finer study of $H^* \PeakConf{n}$ that takes advantage of the fact that generators for the cohomology $H^* \TypeBConf{n}$
segregate into two $\symmB{n}$-orbits:  those indexed by
reflecting hyperplanes of the form $x_i=0$ versus hyperplanes of the form $x_i = \pm x_j$.  This leads to a helpful {\it filtration} on the cohomology ring $H^* \TypeBConf{n}$, whose
associated graded ring is a {\it bigraded} $\symmB{n}$-representation. This filtration and bigrading persists when one takes $\Z_2$-fixed spaces, giving a bigraded $\symm{n}$-representation on an associated
graded ring for $H^* \PeakConf{n}$.  

One surprising outcome of our analysis is an (ungraded) isomorphism (Corollary~\ref{Sn-equivariant-surjection-and-iso}) of $\kk \symm{n}$-modules
\begin{equation}
\label{singly-graded-isomorphism}
H^* \PeakConf{n} \cong H^* \TypeAConf{n}.
\end{equation}
Using the bigrading on $H^* \PeakConf{n}$, we apply this
isomorphism \eqref{singly-graded-isomorphism}  to give a description of each $H^i \PeakConf{n}$ in terms of the {\it higher Lie representations} $\{ \Lie{\lambda}\}$ of Thrall \cite{thrall}. The $\lambda$ indexing the higher Lie characters are integer partitions of $n$, written $\lambda \vdash n$, meaning $|\lambda|:=\sum_{i} \lambda_i= n$. 

The $\Lie{\lambda}$ are representations of the symmetric group  introduced by Thrall in connection to the Poincar\'{e}-Birkhoff-Witt theorem; it is a long-standing open problem (known as \emph{Thrall's problem}) to decompose $\Lie{\lambda}$ into irreducible $\symm{n}$-representations. Nonetheless, the higher Lie characters have appeared in a variety of interesting contexts such as in relation to the free Lie algebra \cite{reutenauer}, the homology of the partition lattice \cite{Sundaram}, the Tsetlin library and derangement representation \cite{brauner2023invariant}, as well as the invariant theory of Tits's face semigroup algebra for the braid arrangement \cite{BHR,Bidigare-thesis,commins2024invariant,Uyemura-Reyes}.

Section \ref{sec:backgroundonidempotents} reviews classical results showing
how the $\symm{n}$-representations $\Lie{\lambda}$ decompose the regular representation,
and refine the cohomological decomposition in \eqref{type-A-Eulerian-idempotents-give-cohomology-reps} as follows, via the
{\it length} or {\it number of parts} $\ell(\lambda)$:
\begin{equation}
\label{type-A-higher-Lie-decomposition}
\begin{aligned}
\kk \symm{n} &\cong \,\,\,\,\, \bigoplus_{|\lambda| = n} \Lie{\lambda}\\
 \left(\kk \symm{n} \right) 
 \TypeAEulerianIdempotent{n}{n-1-k} & \cong
\bigoplus_{\substack{|\lambda|=n:\\ \ell(\lambda)=n-k}} \Lie{\lambda} \cong H^{2k} X_n.
\end{aligned}
\end{equation}
In particular, since $\dim_{\kk} \Lie{\lambda}$ is the number of permutations in $\symm{n}$ of cycle type $\lambda$, this means that the Betti number
$H^{2k} \TypeAConf{n}$ counts permutations in $\symm{n}$ with $n-k$ cycles, a {\it (signless) Stirling number of the 1st kind $c(n,n-k)$}.

Our next main result, proven in Section \ref{sec:proofofpeakreps}, shows the higher Lie representations decompose $H^* \PeakConf{n} \cong H^* \TypeAConf{n}$ more finely than \eqref{type-A-higher-Lie-decomposition}.  To state it, let $\odd(\lambda)$ be the number of odd parts in a partition $\lambda$.

\begin{theorem}
\label{decomposition-of-peak-idempotent-reps-theorem}
When $\fieldchar(\kk) >n$, for  $0 \leq k \leq n$ with $k$ even, one has a $\kk \symm{n}$-module  isomorphism 
$$
\left( \kk \symm{n} \right) \PeakIdempotent{n}{n-k} 
\cong \bigoplus_{\substack{|\lambda|=n:\\ \odd(\lambda)=n-k}} \Lie{\lambda}  \cong
H^{2k} \PeakConf{n}.
$$    
Hence the Betti number $\dim_{\kk} H^{2k} \PeakConf{n}$ counts permutations in $\symm{n}$ with $n-k$ odd cycles.

Furthermore, in the associated graded ring
for $H^* \PeakConf{n}$,  consider the bigraded component of $H^{2k} \PeakConf{n}$ that corresponds to filtration degree $\ell$ in the variables indexed by the $x_i=\pm x_j$ hyperplanes.  Then this component
is isomorphic to the following $\kk \symm{n}$-module:
$$
\bigoplus_{\substack{|\lambda|=n:\\ \ell(\lambda)=n-\ell\\ \odd(\lambda)=n-k }}
\Lie{\lambda}.
$$
\end{theorem}
In fact, we refine Theorem \ref{decomposition-of-peak-idempotent-reps-theorem} by describing the $\symm{n}$-representations generated by \emph{primitive} idempotents in the peak algebra $\PeakAlgebra{n}$, both in terms of a certain decomposition of  $H^*\PeakConf{n}$ and Thrall's higher Lie characters; see Corollary \ref{cor:peakidempotentrep_primitive}.

Theorem~\ref{decomposition-of-peak-idempotent-reps-theorem} has several consequences. First, it lets us
compute a Hilbert series (Corollary~\ref{cor:nonequivariant-bihilb-facts}) and even an $\symm{n}$-equivariant Hilbert series for $H^* \PeakConf{n}$, with recursions lifting to a $\kk\symm{n}$-module branching rule (Theorem~\ref{restriction-rep-recursion}).

Second, Theorem~\ref{decomposition-of-peak-idempotent-reps-theorem} suggests a connection, discussed in Section \ref{sec:jordanalgebra}, between
the filtration on $H^* \PeakConf{n}$, peak idempotents, and work of Robbins \cite{robbins1971jordan} and Calderbank, Hanlon and Sundaram \cite{calderbank1994representations} on the multilinear part of the {\it space of simple Jordan elements} within a free associative algebra on $n$ letters. They showed that the latter space carries the 
$\kk\symm{n}$-module
$$
\bigoplus_{\substack{|\lambda|=n:\\ \odd(\lambda)=\ell(\lambda) }} \Lie{\lambda}.
$$
As explained in Section \ref{sec:jordanalgebra}, this coincides with a certain piece of the associated graded ring of $H^* \PeakConf{n}$.

\medskip

Our analysis has several benefits: 
\begin{itemize}
    \item We offer a novel approach to obtaining structural statements about the peak algebra, avoiding computations in the algebra itself and instead developing and utilizing concrete combinatorial descriptions of the cohomology rings $H^*X_n$, $H^*Y_n$, and $H^*Z_n$. 
    
\item Our work extends to the peak algebra $\PeakAlgebra{n}$ the known connection between the descent algebra $\sol(\symm{n})$ and Thrall's higher Lie characters.
Particularly notable here are Corollary~\ref{cor:peakidempotentrep_primitive} and Conjecture~\ref{conj:peak-idempotent-as-sum-of-Lies}. 
\item Our work contributes toward understanding the cohomology of configuration spaces $\symm{n}$-equivariantly. In general, this is a difficult  problem. Our study of $H^*\PeakConf{n}$ gives an interesting example---notably, one \emph{not} arising directly from a hyperplane arrangement---of a solution to this problem that could offer insight into studying similar spaces. 

\end{itemize}

\subsection*{Organization} The paper is structured as follows:
\begin{itemize}
    \item Section \ref{sec:background} gives background on the    cohomology rings $H^* \TypeAConf{n}, H^* \TypeBConf{n}$, presenting them combinatorially as the graded Varchenko-Gelfand rings for
    the reflection hyperplane arrangements of types $A_{n-1}$ and $B_n$ (Section \ref{sec:VGrings}). To decompose these rings, we discuss the lattices of flats
    for the hyperplane arrangements (Section \ref{sec:flats}), the Eulerian idempotents, Thrall's higher Lie characters, and the connections between these objects (Section \ref{sec:backgroundonidempotents}). \medskip
    
    \item Section~\ref{sec:invariantring} uses topology of quotients by free actions of
    finite groups to prove Theorem~\ref{peak-idempotents-give-cohomology-reps}(i),(ii), identifying $H^* \PeakConf{n}$ as the $\Z_2^n$-fixed subalgebra $(H^* \TypeBConf{n})^{\Z_2^n}$.  \medskip
    
    \item Section~\ref{sec:presentations-decompositions} 
    makes a change-of-variable in the ring presentation of $H^* \TypeBConf{n}$, diagonalizing the $\Z_2^n$-action, to faciliate the analysis of  $H^* \PeakConf{n}=(H^* \TypeBConf{n})^{\Z_2^n}$.  \medskip

    \item Section~\ref{sec:ringZn} uses an interesting
    combinatorial Pairing Lemma~\ref{lem:pairing-bijection} to prove Theorem~\ref{peak-idempotents-give-cohomology-reps}(iii).  The same lemma reappears later, at a crucial step in the proof of
    Theorem~\ref{decomposition-of-peak-idempotent-reps-theorem}.  \medskip
    
    \item Section~\ref{sec:filtration-and-gradings} considers the $\uu$-adic filtration of $H^* \TypeBConf{n}$ where $\uu$ is the ideal generated by
    variables corresponding to coordinate hyperplanes.  The
    associated graded ring for this filtration leads to a further decomposition of $H^*\TypeBConf{n}$ as a $\kk$-vector space indexed by {\it double partitions} $\slambda$ of $n$ in Section \ref{sec:signedpartitionsdecomposition}, which is shown to be compatible with other decompositions.  Section \ref{sec:VSdescription} then gives a quotient $\kk$-vector space presentation of $H^*\TypeBConf{n}$ which will be important for Section \ref{sec:proofofpeakreps}.  \medskip
    
    \item Section~\ref{sec:proofofpeakreps} then studies the $\symm{n}$-representations on $H^* \PeakConf{n}$ via its associated graded ring
    for its inherited $\uu$-adic filtration. The key tool is an $\symm{n}$-equivariant surjection $H^* \TypeBConf{n} \rightarrow H^* \TypeAConf{n}$.  This turns out to restrict to a $\symm{n}$-equivariant vector space isomorphism $H^* \PeakConf{n} \cong H^* \TypeAConf{n}$, leading to the proof of Theorem \ref{decomposition-of-peak-idempotent-reps-theorem},
    and more precise vanishing statements.  \medskip
    
    \item Section \ref{sec:connectiontopeakreps} considers family of peak idempotents which are finer than
    the $\{\PeakIdempotent{n}{k}\}$.
    It shows that they are primitive idempotents for $\PeakAlgebra{n}$, and describes their associated representations in terms of $\lie_{\lambda}$.  \medskip
    
    \item Section \ref{sec:recursions} produces generating functions and recursions for the Hilbert series of $H^* \PeakConf{n}$, and strengthens these to $\symm{n}$-equivariant results. In particular,
    the recursion lifts to a branching rule relating the representations on $H^* \PeakConf{m}$ for $m=n,n-1,n-2$.  \medskip 
    
    \item Section \ref{sec:jordanalgebra} discusses
    the connection to the 
    simple Jordan elements in a free associative algebra.  \medskip
    
    \item Appendix \ref{appendix} proves the branching rules from Section \ref{sec:recursions} (Section  \ref{sec:branching-rule-proof}) and Section \ref{sec:jordanalgebra} (Section   \ref{sec:CHS_generalization-proof})  via symmetric function computations.

\end{itemize}

\section*{Acknowledgments}
The authors are very grateful to Sheila Sundaram for pointing out a  connection to Jordan brackets.  The second author is supported by NSF Mathematical Sciences Postdoctoral Fellowship DMS-2303060, and the third author is partially supported by NSF grant  DMS-2053288.

\section{Background}\label{sec:background}
This section reviews properties of real hyperplane arrangements.  This includes the reflection arrangements of types $A$ and $B$, along with the graded Varchenko-Gelfand ring presentation of the cohomology 
of the spaces in \eqref{configuration-spaces-reexpressed},
and connections to Eulerian idempotents and Thrall's higher Lie characters. 

\subsection{Hyperplane arrangements, reflection groups, and Types $A$ and $B$}
\label{sec:flats}

For background on hyperplane arrangements,
see Aguiar and Mahajan \cite{aguiar2017topics},
Orlik and Terao \cite{OrlikTerao},
and Stanley \cite{Stanley-hyperplanes}.

Let $\A=\{H_1,H_2,\ldots,H_t\}$
be a real (central) hyperplane arrangement
in $V=\R^n$; in other words,
each $H_i$ is a codimension one $\R$-linear subspace of $V$, containing the origin.
 A {\it flat} is an intersection subspace $X=H_{i_1} \cap \cdots \cap H_{i_k}$ of subsets of the hyperplanes in $\A$.
 The collection of all flats, ordered by inclusion, forms a poset, and actually a lattice, denoted $\lat(\A)$.  Note that in the literature, one often sees the opposite order on $\lat(\A)$, by {\it reverse} inclusion which makes it a {\it geometric lattice}, i.e. one which is atomic and upper semimodular.  With our choice, $\lat(\A)$ will be poset dual to a geometric lattice; in other words, it is a coatomic and lower semimodular lattice. This means that $\lat(\A)$ will be ranked, and the flat $X$ will have rank $\dim_{\R}(X)$.

 An important player in our story is the arrangement {\it complement} 
 \begin{equation}
 \label{arrangement-complement}
 V \setminus \A =\R^n \setminus \bigcup_{i=1}^t H_i
 \end{equation}
 which decomposes into connected components called {\it chambers}.

When a finite subgroup $W$ of $GL(V)$ stabilizes $\A$, meaning $w(H)$ lies in $\A$ for all $w$ in $W$ and all $H$ in $\A$, then $W$ acts by poset automorphisms on $\lat(\A)$. We will write $\sh{X}$ to denote the $W$-orbit of $X \in \lat(\A)$. The set of $W$-orbits $\lat(\A)$ will be denoted $\lat(\A)/W$.  Note that in this setting, $W$ also acts via a permutation representation on the set of
chambers in the arrangement complement \eqref{arrangement-complement}.

Of particular interest for us will be the special case where $W$ is a {\it finite reflection group}, that is,
a finite subgroup of $GL(V)$ for $V=\R^n$ which is generated by the reflections $s$ in $W$.  Each reflection $s$ is determined by its reflecting hyperplane $H=V^s=\{v \in V: s(v)=v\}$, since
$s$ fixes $H$ pointwise and acts as the scalar $-1$ on the orthogonal line $H^\perp$.  The {\it reflection arrangement} $\A_W$ is the set of all reflecting hyperplanes for all reflections in $W$.
An interesting feature here is that the permutation action of $W$ on the set of chambers
of the complement \eqref{arrangement-complement}
is {\it simply transitive}.  Hence once may identify this $W$-permutation representation with the {\it (left-)regular representation} $\kk W$, in which $W$ acts on itself via left-translation.

The next two subsections provide explicit descriptions of $\lat(\A_W)$ and $\lat(\A_W)/W$ when $W=\symm{n}, \symmB{n}$.

\subsubsection{The symmetric group $W=\symm{n}$ as a reflection group}
\label{subsec:braidarrangement}
See \cite[Ch. 6.3]{aguiar2017topics} for more
background here.
We consider $W=\symm{n}$ acting on $V=\R^n$ by permuting the coordinates $\{x_i\}_{i=1}^n$. Then $W$ is a reflection group, generated by the transpositions
$s_{ij}=(i,j)$ which are reflections in the hyperplanes $x_i=x_j$.  This gives rise to the type $A_{n-1}$ reflection arrangement in $V=\R^n$: 
$$
\A_{\symm{n}}=\{x_i = x_j\}_{1 \leq i < j \leq n}.
$$
Its lattice of flats $\lat( \A_{\symm{n}})$ is isomorphic to the lattice of set partitions of $[n]:=\{1,2,\ldots,n\}$. This isomorphism identifies a flat $X \in \lat(\A_{\symm{n}})$ with the set partition $\pi_X = \{ B_1, \cdots, B_k \}$ whose blocks $B_j = \{ i_1, \cdots, i_\ell \}$ correspond to subspaces 
$x_{i_1} = x_{i_2} = \cdots =x_{i_{\ell}}
$
whose intersection is $X$;  that is,
one has $x_i=x_j$ 
for all $x$ in $X$ if and only if $i,j$ lie in the same block of $\pi_X$. The order relation on flats corresponds to refinement: $X \leq Y$ if $\pi_Y$ is a refinement of $\pi_{X}$. The dimension of $X$ is $\dim(X)=\ell(\pi_X)$, the number of blocks of $\pi_X$.

The group $\symm{n}$ acts on $\lat(\A_{\symm{n}})$. The $\symm{n}$-orbit of a flat $X$
may be identified with the (integer) partition $\sh{X}=(\lambda_1 \geq \cdots \geq \lambda_k)$ with $|\lambda|=n$ whose parts are the sizes $|B_1|,\ldots,|B_k|$ of the blocks in the set partition $\pi_X$, but written in weakly decreasing order.
Thus $\lambda$ has length $\ell(\lambda)=k=\dim(X)$ parts.

\begin{example}\rm
When $W=\symm{8}$, the following flat in  $\lat( \A_{\symm{8}} )$
$$
X=\{ x\in V=\R^8:  
x_3=x_4 , \ \   x_1= x_2=x_5 , \ \  x_7 =x_8\}.
$$
corresponds to the set partition
$$
\pi_X=\{B_1,B_2,B_3,B_4\}=\{\{3,4\},\{6\},\{1,2,5\},\{7,8\}\},
$$
whose $W$-orbit $\sh{X}$ in $\lat(\A_{\symm{8}})/\symm{8}$ corresponds to the number partition $\sh{X}=(3,2,2,1) \vdash 8=n$.
\end{example}

\subsubsection{The hyperoctahedral group $W=\symmB{n}$ as a reflection group}
\label{subsec:typebarrangement}
See \cite[\S 6.7]{aguiar2017topics} for more background here.
We consider the hyperoctahedral group
$W=\symmB{n}$, acting on $V=\R^n$
via the $n \times n$ signed permutation matrices
within $GL(V)$. In other words, $\symmB{n}$ acts by permuting and negating
the coordinates in $V$.  Then $W$ is a reflection group, generated by 
\begin{itemize}
    \item the transpositions that reflect through the hyperplanes $x_i=x_j$, 
    \item the
``signed transpositions" that reflect through the hyperplanes $x_i=-x_j$, 
\item and the coordinate negations $\tau_i$ that reflect through the hyperplanes $x_i=0$. 
\end{itemize}
 This gives rise to the type $B_n$ (or $C_n$) reflection arrangement in $V=\R^n$ 
$$
\A_{\symmB{n}}=\{x_i=0\}_{1 \leq i \leq n} \sqcup  \{x_i = \pm x_j\}_{1 \leq i<j \leq n}. 
$$
In order to describe the lattice of flats $\lat(\A_{\symmB{n}})$, we first establish
certain conventions.. We consider the set $[n]^{\pm}:= \{ \overline{1}, \overline{2}, \cdots, \overline{n}, 1, 2, \cdots n \}$
with the negation involution $i \leftrightarrow \bar{i}$ (so that $\bar{\bar{i}}=i$),
and adopt the convention for the coordinates of $x=(x_1,\ldots,x_n)$ in $\R^n$ that $x_{\bar{i}}:=-x_i$.  With this convention, 
a flat $X \in \lat(\A_{\symmB{n}})$ can be identified with a set partition  \[
\pi_X=\{ \bar{C}_k,\bar{C}_{k-1},\ldots,\bar{C}_1,\bar{C}_0=C_0,C_1,C_2,\ldots,C_k\}
\]
of the set $[n]^{\pm}$ satisfying two properties:
\begin{itemize}
    \item $\pi_X$ is stable under $i \leftrightarrow \bar{i}$, meaning $C_i$ appears in $\pi_X$ if and only if
$\bar{C}_i:=\{\bar{j}\}_{j \in C_i}$ also appears, and
\item $\pi_X$ has at most one block $C_0=\bar{C}_0$, called the {\it zero block} $C_0$ (if present).
\end{itemize}
Points $x$ in the flat $X$ satisfy $x_i=x_j$
if and only if $i,j$ in $[n]^\pm$ lie in the same block of $\pi_X$, bearing in mind that $x_{\bar{i}}=-x_i$; this means that points $x$ in $X$ satisfy $x_i=0$ if
and only if $i$ lies in the zero block $C_0$.
Denote the cardinality of the zero block as $z(X):= |C_0|/2$.
One can check that the dimension of the flat $X$ is $k$, or half the number of nonzero blocks.

The hyperoctahedral group $\symmB{n}$ acts on flats in $\lat(\A_{\symmB{n}})$, and we again identify the $\symmB{n}$-orbit of $X$ with a partition $\sh{X}=\mu=(\mu_1 \geq \cdots \geq \mu_k)$ listing the sizes of (half of) the
nonzero blocks $|C_1|,\ldots,|C_k|$ in weakly decreasing order.  Thus $\mu$ has length $\ell(\mu)=k=\dim(X)$ parts and weight $|\mu|=\sum_i \mu_i=n-z(X) \leq n$.

\begin{example}\label{ex:Bflat_indexing} \rm
When $W=\symmB{8}$, the following flat in $\lat(\A_{\symmB{8}})$
$$
X=\{ x\in V=\R^8:  
x_2=-x_5=x_6, \ \ x_1=-x_8, \ \ x_3=x_7=0 \},
$$
corresponds to the unordered set partition 
$$
\pi_X =
\{ \bar{C}_3,\bar{C}_2,\bar{C}_1,C_0,C_1,C_2,C_3\}=
\{
\{2,\bar{5},6\},
\{4\},
\{\bar{1},8\},
\{3,7,\bar{3},\bar{7}\},
\{1,\bar{8}\},
\{\bar{4}\},
\{\bar{2},5,\bar{6}\}
\}.
$$
The $W$-orbit $\sh{X}$ in $\lat(\A_{\symmB{8}})/\symmB{8}$ 
corresponds to the number partition
$\sh{X} = (3,2,1) \vdash 6=8-z(X)$.
 \end{example}
 
\subsection{Graded Varchenko-Gelfand rings}\label{sec:VGrings}
The configuration space cohomology rings $H^* X_n$ and $H^* Y_n$ from the introduction have presentations which are special cases of a more general ring called the (associated graded) Varchenko-Gelfand ring of a real hyperplane arrangement.

\begin{definition} \rm
Let $\A=\{H_1,H_2,\ldots,H_t\}$
be a real hyperplane arrangement
in $V=\R^n$.  After arbitrarily choosing linear forms $\{ \alpha_i\}_{i=1,2,\ldots,t}$ in $V^*$ such that $\ker \alpha_i=H_i$, define the {\it associated graded Varchenko-Gelfand ring} $\VG(\A)$ as the commutative quotient  
\begin{equation}
\label{VG-presentation}
\VG(\A):=\Z[u_1,u_2,\ldots,u_t]/\J_{\A}.
\end{equation}
The ideal $\J_{\A}$ is
generated by the squares $u_1^2,u_2^2,\ldots,u_t^2$
together with one relation $\partial(C)$ for
each minimal $\R$-linear dependence 
$
\sum_{j=1}^s c_j \cdot \alpha_{i_j}=\bf{0},
$ 
called a {\it (matroid) circuit} $C=\{i_1,i_2,\ldots,i_s\}$
:
    \begin{equation}
    \label{VG-circuit-relation}
      \partial(C):=  \sum_{j=1}^{s} \mathrm{sgn}(c_j) \cdot u_{i_1} u_{i_2} 
\cdots \widehat{u_{i_j}} \cdots u_{i_{c-1}} u_{i_c}.
    \end{equation}
\end{definition}

\begin{remark} \rm
\label{VG-grobner-basis-remark}
    In fact, it was shown by Cordovil \cite[Cor.~2.8]{Cordovil} (see also Dorpalen-Barry \cite[Thm.~1]{Dorpalen-Barry}) that the generators $u_i^2$ and \eqref{VG-circuit-relation} form a {\it Gr\"obner basis} with respect to any monomial ordering $\prec$ on $\kk[u_1,\ldots,u_t]$.  Furthermore,
    if one re-indexes so that $u_1 \prec u_2 \prec \cdots \prec u_t$, then
    the $\prec$-initial term of the relation \eqref{VG-circuit-relation} is the {\it broken-circuit monomial} $u_{i_2} u_{i_3} \cdots u_{i_s}$
    where $i_1=\min\{i_1,i_2,\ldots,i_s\}$.  Consequently, $\VG(\A)$ has an {\it nbc-monomial basis} of
    $\prec$-standard monomials, that is, those
    squarefree monomials in the $\{u_i\}$ which are divisible by none of these broken-circuit monomials.
    In particular, this means that the {\it Hilbert series}
    of $\VG(\A)$ has the following expression:
    $$
    \mathrm{Hilb}(\VG(\A),t)
    :=\sum_{d \geq 0} t^d \cdot \dim_{\kk} \VG(\A)_d
    =\sum_{\text{nbc-monomials }m} t^{\deg(m)}.
    $$
\end{remark}

Work of de Longueville and Schultz \cite[Cor. 5.6]{DelonguevilleSchultz} and later 
Moseley \cite{moseley2017equivariant} identifies $\VG(\A)$ as a cohomology algebra, namely $
\VG(\A) \cong H^*(X_{\R^3},\Z)
$
where
\begin{equation}
\label{R3-thickening}
X_{\R^3}:=\left( V \otimes_{\R} \R^3 \right) \setminus \left( \A \otimes_{\R} \R^3 \right),
\end{equation}
is the ``$\R^3$-thickening" of the hyperplane arrangement complement $V \setminus \A$
from \eqref{arrangement-complement}.
Note that the unthickened hyperplane arrangement complement $V \setminus \A$ has its cohomology ring $H^*(V \setminus \A,\Z)=H^0(V \setminus \A,\Z)$
concentrated in dimension $0$, and 
is the ring of locally constant $\Z$-valued functions from the chambers (=connected components of $V\setminus \A$) under pointwise addition and multiplication.  The latter ring carries an interesting filtration studied  originally by Varchenko and Gelfand \cite{varchenkogelfand}, for which the associated graded ring is isomorphic to $\VG(\A)$.

\subsection{Symmetry}
We are interested in the situation where the
arrangement $\A$ carries the symmetry of a finite subgroup $W$ of $GL(V)$.  In this case, the ring $\VG(\A)$ carries 
a $W$-action induced by the following signed permutation action on the generators $u_i$ corresponding to the hyperplanes $H_i=\ker(\alpha_i)$ and
choice of linear forms $\{\alpha_i\}$: 
\begin{align}
 \notag
\text{ if }w\in W\text{ has }w(H_i)&=H_j,\\
\notag
\text{ so that }w(\alpha_i)&=\pm \alpha_j,\\
   \label{W-action-on-VG-variables}
   \text{ then }w(u_i) &:= \pm u_j.
\end{align}
We abuse notation slightly, allowing $\VG(\A)$ to be defined
as a quotient of $\kk[u_1,\ldots,u_t]$ for any coefficient ring $\kk$
rather than just $\Z$. Moseley's work then implies the following.

\begin{theorem}[Moseley, \cite{moseley2017equivariant}]
\label{thm:Moseley}
   Consider a real hyperplane arrangement $\A \subset V$ with a finite subgroup  $W \subset GL(V)$ of symmetries.
 Then as graded $\kk$-algebras and $\kk W$-modules, one has a grade-doubling isomorphism
 \begin{equation}
 \label{Moseley-cohomology-ring-isomorphism}
 \VG(\A) \cong H^*(X_{\R^3},\kk).
 \end{equation}
   Furthermore, for fields $\kk$ and $|W| \in \kk^\times$, so $\kk W$ is semisimple, one has
    (ungraded) $\kk W$-module isomorphisms
    $$
    \VG(\A) \cong H^*(X_{\R^3},\kk) \cong
    H^0(  V \setminus \A,\kk),
    $$
    where $H^0(X,\kk)$ is the $\kk W$-permutation module on the chambers of $\A$.
    
    In particular, for $\A_W$ the reflection arrangement 
    of a finite reflection group $W$ acting on $V$, the total (ungraded) cohomology $H^*(X_{\R^3},\kk)$ carries the (left-)regular $\kk W$-module:
    $$
    \VG(\A_W) \cong H^*(X_{\R^3},\kk) \cong H^0(V \setminus \A_W,\kk) \cong \kk W.
    $$
\end{theorem}

\subsection{Decompositions by flat and flat-orbits}

The natural grading of the ring $\VG(\A)$
corresponds to the cohomological grading on $H^*(X_{\R^3})$
under the isomorphism \eqref{Moseley-cohomology-ring-isomorphism} (after dividing the cohomological grading in half).  In \cite{brauner2022eulerian}, the second author proved that for any central real arrangement $\A$, the presentation of the associated graded Varchenko-Gefland ring $\VG(\A)\cong H^*(X_{\R^3})$ endows it with a finer
$\Z$-module direct sum decomposition 
indexed by flats 
\begin{equation}
\label{VG-flat-decomposiiton}
\VG(\A)=\bigoplus_{X \in \lat_{\A}} \VG(\A)_X,
\end{equation}
where here the $X$-summand is the $\Z$-span of all monomials 
$\{ u_{i_1} \cdots u_{i_\ell} \}$
for which $H_{i_1} \cap \cdots \cap H_{i_\ell}=X$.
When a group $W$ acts on $\A$, this coarsens
to a $\Z W$-submodule direct sum decomposition indexed by \emph{flat orbits}:
\begin{equation}
\label{flat-orbit-decomposition}
\VG(\A) 
=\bigoplus_{\sh{X} \in \lat(\A)/W} 
\underbrace{\left( \bigoplus_{Y \in \sh{X}} \VG(\A)_Y \right)}_{\text{ call this }
\VG(\A)_{\sh{X}}}
=\bigoplus_{\sh{X} \in \lat(\A)/W} 
\VG(\A)_{\sh{X}}.
\end{equation}

In fact, the work in \cite[Ex. 3.3, Thm. 5.9]{brauner2022eulerian} uses Theorem~\ref{Moseley-cohomology-ring-isomorphism}
to identify the flat-orbit summands $\VG(\A)_{\sh{X}}$ topologically,
corresponding under the
isomorphism  $\VG(\A) \cong H^*(X_{\R^3})$ to 
$\kk W$-summands
inherent in Sundaram and Welker's {\it equivariant Goresky-MacPherson formula} \cite[Thm. 2.5(ii)]{sundaramwelker} for $H^*(X_{\R^3})$.  The next subsection reviews how these
$\kk W$-submodules $\VG_{\sh{X}}$ also arise as
summands $(\kk W) E^W_{\sh{X}}$ of $\kk W$
generated by certain idempotents $\{ E^W_{\sh{X}} \}$ within Solomon's descent algebra for $W$.

\subsection{Solomon's Descent algebra, Eulerian idempotents, and Thrall's Higher Lie characters}\label{sec:backgroundonidempotents}
Here we review how the theory of descents in reflection groups led to certain primitive orthogonal idempotents in Solomon's descent
algebra, and to the isomorphisms
\eqref{type-A-Eulerian-idempotents-give-cohomology-reps}
and \eqref{type-B-Eulerian-idempotents-give-cohomology-reps} from the Introduction.

Throughout this subsection, $W$ is a finite real reflection group with reflection arrangement $\A_W \subset V=\R^n$.  Recall from Section~\ref{sec:flats}
that $W$ acts transitively on the chambers which are the
connected components of the complement $V \setminus \A_W$.
Picking one such chamber $C_0$ arbitrarily, and assuming that $W$ acts irreducibly on $V$, it turns
out that there will be exactly $n$ hyperplanes $H_1,\ldots,H_n$ {\it adjacent} to the chamber $C_0$. (``Adjacent" here means
$H_i$ is spanned by its intersection with the closed chamber $\bar{C}_0$).  Furthermore, the reflections $S=\{s_1,\ldots,s_n\}$ through the hyperplanes $H_1,\ldots,H_n$ are known to give
a {\it Coxeter presentation}
$$
W \cong \langle S: s_i^2=1=(s_i s_j)^{m_{ij}} \rangle.
$$   
From this one defines for $w$ in $W$ its
{\it Coxeter length function} $\ell(w)$, and descent set $\Des(w)$, {\it descent number} $\des(w)$:
\begin{align*}
\ell(w)&:=\min\{\ell: w = s_1 s_2 \cdots s_\ell\text{ for }s_i \in S\},\\
\Des(w)&:= \{ s \in S: \ell(ws) < \ell(w) \},\\
\des(w)&:= \# \Des(w).
\end{align*}

\begin{example} \rm
\label{descent-sets-in-type-A-example}
    When $W = \symm{n}$ permutes coordinates in $\R^n$, it acts irreducibly as a reflection group on the subspace $V \cong \R^{n-1}$ where $x_1+\cdots+x_n=0$. Choosing the chamber $C_0$ to be the open cone in $V$ where $x_1 < x_2 < \cdots < x_n$, then its adjacent hyperplanes are $\{ x_1=x_2, \  x_2=x_3,\ \ldots, \ x_{n-1}=x_n\}$, and the reflections $S=\{s_1,s_2,\ldots,s_{n-1}\}$ are the adjacent transpositions $s_i=(i,i+1)$.  A permutation $w=(w_1,w_2,\ldots,w_n)$ in $\symm{n}$ has descent set $\Des(w)=\{s_i: w_i > w_{i+1}\}$. We identify this set with the classical descent set 
$$
\Des(w):=\{i: w_i > w_{i+1}\} \subset \{1,2,\ldots,n-1\}.
$$
via the bijection $S  \leftrightarrow \{1,2,\ldots,n-1\}$ sending $s_i \mapsto i$.
The descent number of $w$ is then $\des(w)=\#\{i:w_i > w_{i+1}\}$.
\end{example}

As mentioned in the Introduction, Solomon \cite{solomon} discovered a surprising subalgebra of the group algebra $\kk W$, now called \emph{Solomon's descent algebra}, defined as follows: 
$$
\sol(W):=\left\{ \sum_{w \in W} c_w \cdot w:  c_w \in \kk \text{ and }c_w\text{ depends only upon }\Des(w)\right\} \subset \kk \symm{n}.
$$

When $W=\symm{n}$ as in Example~\ref{descent-sets-in-type-A-example}, the descent algebra $\sol(\symm{n})$ contains an even smaller subalgebra of $\kk\symm{n}$, called the \emph{peak algebra} $\PeakAlgebra{n}$, first investigated by Nyman \cite{nyman2003peak} and Aguiar, Bergeron and Nyman \cite{aguiar2004peak}. To define it, introduce the \emph{Peak set}
$$
\peak(\sigma):= \{i : \sigma_{i-1} < \sigma_i > \sigma_{i+1} \} \subset \{ 1, 2, \cdots, n-1 \},
$$
with the convention that $\sigma_0:= 0$. One defines the peak algebra as
$$
\PeakAlgebra{n}:=
\left\{ 
\sum_{w \in \symm{n}} c_w \cdot w:  c_w \in \kk \text{ and }c_w\text{ depends only upon }\peak(w)
\right\} \subset \sol(\symm{n}) \subset \kk \symm{n}.
$$
Importantly, the work of Aguiar--Bergeron--Nyman \cite{aguiar2004peak} realizes $\PeakAlgebra{n}$ as the image of $\sol(\symm{n})$ under the map on group algebras induced by the group quotient $\varphi:\symmB{n} \twoheadrightarrow \symm{n}$ which forgets the signs of a signed permutation:
\begin{align}
\label{sign-forgetting-map}
\varphi: \kk\symmB{n} &\twoheadrightarrow \kk\symm{n}\\
\text{restricts to }\varphi: \sol(\symmB{n}) &\twoheadrightarrow \sol(\symm{n}).
\end{align} 
\medskip

One can understand the connection between $\sol(W)$ and $\VG(\A_W)$ by studying certain complete families of orthogonal idempotents inside $\sol(W)$, called the \emph{Eulerian idempotents}, defined next in Section \ref{sec:eulerianidempotents}. One goal of this paper will be to show that the same relationship holds between $\PeakAlgebra{n}$ and $H^* \PeakConf{n}$. 

\subsubsection{The Eulerian idempotents}\label{sec:eulerianidempotents}
The Eulerian idempotents have been extensively studied, with their origins in Type $A$. 
In \cite{reutenauer1986theorem}, Reutenauer introduced a complete family of primitive orthogonal idempotents 
$\{\TypeAEulerianIdempotent{n}{\lambda}\}$ in $\sol(\symm{n})$, indexed by partitions $|\lambda|=n$, as part of his investigation of the Campbell-Baker-Hausdorff formula \cite{reutenauer1986theorem}. He and Garsia studied these further in \cite{garsiareutenauer}, introducing  the \emph{Eulerian idempotents} 
$\{ \TypeAEulerianIdempotent{n}{k}\}$ where
$$
\TypeAEulerianIdempotent{n}{k}=
\sum_{\substack{\lambda\vdash n:\\ \ell(\lambda)=k}}\ \TypeAEulerianIdempotent{n}{\lambda}.
$$
Garsia--Reutenauer showed the Eulerian idempotents have a generating formula expression: 
\begin{equation}\label{eq:typeaeul} \sum_{k=0}^{n-1} t^{k+1}\TypeAEulerianIdempotent{n}{k} = \sum_{\sigma \in \symm{n}} \binom{t-1+n-\des(\sigma)}{n}\sigma.\end{equation}
Independently, but around the same time, Gerstenhaber and Schack in \cite{gerstenhaberschack} introduced a family of idempotents in $\kk \symm{n}$ which gave a Hodge-type decomposition of Hochschild homology. Though not obvious, work of Loday in \cite{lodayoperations} shows that these idempotents are the same $\{\TypeAEulerianIdempotent{n}{k} \} $ as in \eqref{eq:typeaeul}.

So began a flurry of activity in studying and generalizing the Eulerian idempotents. Bergeron and Bergeron in \cite{bergeronbergeron} constructed a complete family of primitive orthogonal idempotents, 
$\{ \TypeBEulerianIdempotent{n}{\mu}\}_{|\mu| \leq n}$ for $\sol(\symmB{n})$
along with type $B_n$ Eulerian idempotents 
$\{\TypeBEulerianIdempotent{n}{k}\}$ having a generating function analogous to \eqref{eq:typeaeul}:
\begin{equation}\label{eq:typebeul}
    \sum_{k=0}^{n}t^{k}\TypeBEulerianIdempotent{n}{k} = \sum_{w \in \symmB{n}} \binom{\frac{t-1}{2} + n - \des(w)}{n} w.
\end{equation}

Bergeron, Bergeron, Howlett and Taylor \cite{BBHT} generalized this work, constructing complete families of primitive, orthogonal idempotents for the descent algebra $\sol(W)$ of any finite Coxeter group $W$, generalizing the constructions from \cite{bergeronbergeron,garsiareutenauer, reutenauer1986theorem}. 
These {\it BBHT idempotents} $\{E^W_{\sh{X}}\}_{\sh{X} \in \lat(\A_W)/W}$ are indexed by the $W$-orbits $\sh{X}$ of the flats $X$ in $\lat(\A_W)$ for the reflection arrangement $\A_W$.  We will not need formulas for $E^W_{\sh{X}}$ in general, but note that one can use them to define {\it Eulerian idempotents} for $W$, by summing them over $W$-orbits
 of flats with a fixed dimension:
\begin{equation}
\label{general-BBHT-to-Eulerian-refinement}
E^W_k:=\sum_{\substack{\sh{X} \in \lat(\A_W)/W:\\ \dim(X)=k}} E^W_{\sh{X}}.
\end{equation}
For $W=\symm{n}, \symmB{n}$, the $\{E^W_k\}$
recover the type $A, B$ Eulerian idempotents
$\{ \TypeAEulerianIdempotent{n}{k}\}, \{ \TypeBEulerianIdempotent{n}{k}\}$ from \eqref{eq:typeaeul}, \eqref{eq:typebeul}.

\subsubsection{The Eulerian representations}\label{sec:eulerianreps}
Work of the second author \cite[Thms. 1.1, 5.9]{brauner2022eulerian} related the $\kk W$-modules generated by BBHT idempotents\footnote{Actually, \cite{brauner2022eulerian} used slightly different primitive idempotents for $\sol(W)$, but any two such families in $\sol(W)$ are $W$-conjugate \cite[Corollary 1.7.4 ]{benson1998representations}, with isomorphic representations $(\kk W) E^W_{\sh{X}}$.  The second author thanks Franco Saliola for this observation. } $(\kk W) E^W_{\sh{X}}$ to the cohomology ring $H^*(X_{\R^3})$ and the flat-orbit summands $\VG(\A_W)_{\sh{X}}$ in 
\eqref{flat-orbit-decomposition}.

\begin{theorem}
\label{thm:eulerian-reps-as-cohomology}
For 
$W$ a finite reflection group of rank $r$, 
with $|W| \in \kk^\times$,
one has $\kk W$-module isomorphisms
$$
(\kk W) E^W_{\sh{X}} \cong 
\VG(\A_W) _{\sh{X}}.
$$
Summing this over $W$-orbits $\sh{X}$ where $\dim(X)=k$, one therefore also has
$$
(\kk W) E^W_k   \cong \VG(\A_W)_{r-k}  \cong 
     H^{2(r-k)}(X_{\R^3},\kk).
$$

In particular, taking $W=\symm{n}$ one has for $0 \leq k \leq n-1$ and $|\lambda|=n$ that there are $\kk \symm{n}$-module isomorphisms
\begin{align}
\label{eq:typeaiso}
H^{2(n-1-k)}\TypeAConf{n} 
\cong \VG(\A_{\symm{n}})_k
\cong (\kk \symm{n} )\TypeAEulerianIdempotent{n}{k},\\
\label{eq:typea-flatorbit-iso} 
\VG(\A_{\symm{n}})_\lambda
\cong (\kk \symm{n}) \TypeAEulerianIdempotent{n}{\lambda}. 
\end{align} 

Likewise, taking $W=\symmB{n}$, for $0 \leq k \leq n$ and partitions $\mu$ with $|\mu| \leq n$, one has $\kk \symmB{n}$-module isomorphisms
\begin{align}
    \label{eq:typebiso} H^{2(n-k)}\TypeBConf{n}
    \cong \VG(\A_{\symmB{n}})_k
    \cong (\kk \symmB{n} )\TypeBEulerianIdempotent{n}{k}\\
    \label{eq:typeb-flat-orbit-iso}
\VG(\A_{\symmB{n}})_\mu
    \cong (\kk \symmB{n} )\TypeBEulerianIdempotent{n}{\mu}.
\end{align}
\end{theorem}
Theorem~\ref{thm:eulerian-reps-as-cohomology} is the inspiration and theoretical backbone for the results in this paper. The $\kk \symm{n}$-module isomorphisms \eqref{eq:typeaiso},\eqref{eq:typea-flatorbit-iso} were known earlier via comparison of character computations of Sundaram and Welker \cite{sundaramwelker} and Hanlon \cite{hanlon};
its remaining assertions were proven
for reflection groups in \cite{brauner2022eulerian}.

\subsubsection{Thrall's higher Lie Characters}
\label{Thrall-section}
We conclude Section \ref{sec:background} by
giving a third incarnation, when $W=\symm{n}$,
for the representations appearing in \eqref{eq:typea-flatorbit-iso}, namely the
 \emph{higher Lie characters} $\lie_{\lambda}$ of
 Thrall \cite{thrall}. 

 Letting $V=\C^n$, one can consider the tensor powers $T^d(V):=\C^{\otimes d}$ and the {\it tensor algebra}
 $$
 T(V)=\bigoplus_{d=0}^\infty T^d(V)
 = \C \oplus V \oplus (V \otimes V) \oplus (V\otimes V \otimes V) \oplus \cdots
 $$
 as a {\it free associative $\C$-algebra} on the generators $x_1,\ldots,x_n$ which are a $\C$-basis for $V$.  It is also a
 Lie algebra under the usual bracket operation
 $[x,y]:=x \otimes y - y \otimes x$.  The Jacobi identity
 shows that the Lie subalgebra $\cLfree(V)$ generated by $V$ is the same as the $\C$-span of left-bracketings 
 of elements of $V$
 $$
 \cLfree(V) := \bigoplus_{d=1}^\infty \cLfree_d(V) =V \oplus [V,V] \oplus [[V,V],V] \oplus \cdots
 $$
 that is, the span of all elements such as 
 $$
 v_1, \,\, 
 [v_1,v_2], \,\, 
 [[[v_1,v_2],v_3], \,\, 
 [[[[v_1,v_2],v_3],v_4],\ldots
 $$
 with $v_i \in V$.
 This space $\cLfree(V)$ turns out to be isomorphic to the {\it free Lie algebra} on the $n$ generators $x_1,\ldots,x_n$,
 with $T(V)$ as its universal enveloping algebra
 $\cU(\cLfree(V))$; see Reutenauer \cite{reutenauer} for a thorough treatement.
 
 \begin{remark} \rm
 Section~\ref{sec:jordanalgebra} considers the $\C$-subspace of $T(V)$, called
 the space of {\it simple Jordan elements} analogous
to $\cLfree(V)$, replacing the left-bracketings above by their {\it Jordan bracket} 
variant $[x,y]:=x \otimes y + y \otimes x$.
\end{remark}

 Thrall examined this {\it Poincar\'e-Birkhoff-Witt 
 isomorphism} of representations of $GL(V) \cong GL_n(\C)$:
 \begin{align*}
 T(V) 
 = \cU(\cLfree(V)) 
 \cong \Sym(\cLfree(V))
&=\Sym\left(\bigoplus_{d=1}^\infty \cLfree_d(V) \right)\\
&=\bigoplus_{ \lambda=1^{m_1} 2^{m_2} \cdots \ell^{m_\ell}} 
 \underbrace{
 \Sym^{m_1}\cLfree(V)_1 \otimes 
  \Sym^{m_2}\cLfree(V)_2\otimes \cdots \otimes
   \Sym^{m_\ell}\cLfree(V)_\ell}_{\cLfree_\lambda(V):=} 
 \end{align*}
 where the index of summation in the last direct sum runs over all partitions $\lambda$, and where $\lambda=1^{m_1} 2^{m_2} \cdots$,
 means that $m_i$ the number of parts of size $i$ in $\lambda$.

 Thrall asked for the $GL(V)$-irreducible decomposition of each summand $\cLfree_\lambda(V)$,
 or equivalently via {\it Schur-Weyl duality}, the $\symm{n}$-irreducible decomposition on its {\it multilinear part}
 $\Lie{\lambda}$, now called a {\it higher Lie character}.  The work of Reutenauer \cite{Reutenauer-Stirling} and Garsia and Reutenauer \cite[\S 2]{garsiareutenauer} (see also Garsia \cite{Garsia-freeLie}) identifies these higher Lie characters with the $\kk \symm{n}$-modules generated by the idempotents $\TypeAEulerianIdempotent{n}{\lambda}$. That is,
 when $\kk$ has characteristic zero, there are
 $\kk \symm{n}$-module isomorphisms:
 \begin{equation}
 \label{Lie-lambda-isos}
 \lie_\lambda \cong (\kk \symm{n}) \TypeAEulerianIdempotent{n}{\lambda}
 \cong \VG(\A_{\symm{n}})_\lambda. 
 \end{equation}

Combining this with Theorem~\ref{thm:eulerian-reps-as-cohomology}, for $\kk$ of characteristic zero,
one has $\kk\symm{n}$-module isomorphims
\begin{equation}
\label{typeA-cohomology-as-higher-Lie}
H^{2(n-1-k)}\TypeAConf{n} \cong \bigoplus_{\substack{|\lambda|=n\\ \ell(\lambda) = k}} \lie_{\lambda}.
\end{equation}
This isomorphism plays in an essential role in our later analysis of the peak representations $(\kk \symm{n}) \PeakIdempotent{n}{k}$.

\subsection{Semisimplicity assumptions}
\label{sec:semisimplicity}

Note that for finite groups $G$
and nested $\kk G$-modules $U \subset V$,
generally $V \not\cong U \oplus V/U$ as $\kk G$-modules
and the $G$-invariant spaces $(V/U)^G \not\cong V^G/U^G$
as $\kk$-vector spaces.  However,
whenever $\kk G$ is semisimple, that is, $|G|$ in $\kk^\times$, then there will be a
$\kk G$-submodule complement $U'$, so that one has both
$$
\begin{array}{rcccll}
V &\cong& U \oplus U' &\cong& U \oplus V/U &\text{ as }\kk G\text{-modules, and }\\
(V/U)^G &\cong& (U')^G &\cong& V^G/U^G
&\text{ as }\kk\text{-vector spaces.}
\end{array}
$$
In particular, when one has a $\kk G$-module filtration
$$
(V :=) \,\, V_0 \supset V_1 \supset V_2 \supset \cdots \supset V_\ell \supset 0 \,\,(=:V_{\ell+1}),
$$
then whenever $\kk G$ is semisimple, one has these isomorphisms with the associated graded spaces:
$$
\begin{array}{rcll}
V &\cong&\displaystyle\bigoplus_{d=0}^\ell V_d/V_{d+1} &\text{ as }\kk G\text{-modules, and }\\
V^G&\cong&\displaystyle\bigoplus_{d=0}^\ell V_d^G/V^G_{d+1}
&\text{ as }\kk\text{-vector spaces.}
\end{array}
$$
In this paper, the groups acting are usually $G=\symmB{n},\symm{n},\Z_2^n$,
of cardinalities $2^n n!, n!, 2^n$.  Consequently, we will sometimes impose hypotheses on the field characteristic $\fieldchar(\kk)$ of the form $\fieldchar(\kk) >2$
so $\kk \Z_2^n$ is semisimple, or
$\fieldchar(\kk)>n\geq 2$ so both $\kk \symm{n}, \kk\symmB{n}$ are semisimple.  When writing $\fieldchar(\kk) > m$, we will always regard fields of characteristic zero as satisfying this condition.
Also, in discussing $\kk \symm{n}$-modules, 
we will always assume $n\ge 2$, even if it is not explicitly stated, since all of our results trivialize for $n=1$.

\section{Proof of Theorem~\ref{peak-idempotents-give-cohomology-reps}(i),(ii): Peak representations as cohomology}
\label{sec:invariantring}

We now use $H^*\TypeBConf{n}$  
to understand $H^{*}\PeakConf{n}$, as a way of proving Theorem \ref{peak-idempotents-give-cohomology-reps} parts (i) and (ii). 

Recall that one has a $\Z_2$-action on $\R^3 \setminus \{\origin\}$ via $\xx \mapsto -\xx$.
Under the homeomorphism $\R^3 \setminus \{\origin\} \rightarrow \SSS^2 \times \R$
that sends $\xx \mapsto (\mathbf{u},r)=\left(\frac{\xx}{\|\xx\|},\log \|\xx\|\right)$, this corresponds to the $\Z_2$-action $(\mathbf{u},r) \mapsto (-\mathbf{u},r)$, leading to these homeomorphisms:
$$
\left( \R^3 \setminus \{\origin\} \right)/\Z_2
\,\, \cong \,\, 
\left( \SSS^2 / \Z_2 \right) \times \R
\,\, \cong \,\, \RP^2 \times \R.
$$
Consequently, one can view the space $Z_n$
as the following configuration space:
$$
\PeakConf{n}:=\TypeBConf{n}/\Z_2^n = \Conf_n(\left( \R^3 \setminus \{\origin\} \right)/\Z_2)
\cong \Conf_n(\RP^2 \times \R).
$$
We recall the assertions (i),(ii) of the theorem here. 

\vskip.1in
\noindent
{\bf{Theorem~\ref{peak-idempotents-give-cohomology-reps}}}
{\it
Let $H^* Z_n$ be the cohomology with coefficients
in a field $\kk$ of 
$$
Z_n=\TypeBConf{n}/\Z_2^n=\Conf_n(\RP^2 \times \R) .
$$
\begin{itemize}
\item[(i)] When $\fieldchar(\kk)>2$, the $\kk \symm{n}$-module on the total cohomology is the regular representation:
$$
H^* \PeakConf{n} \cong \kk \symm{n}.
$$
\item[(ii)] When $\fieldchar(\kk)>n$, for  $0 \leq k \leq n$ one has a $\kk \symm{n}$-module isomorphism 
$$
H^{2k} \PeakConf{n} \cong \left( \kk \symm{n} \right) \PeakIdempotent{n}{n-k}.
$$
\end{itemize}
}
\vskip.1in

\begin{proof}
Take all cohomology with $\kk$ coefficients. The two assertions (i), (ii) then come from the following two sequences of $\kk\symm{n}$-module isomorphisms, which will be justfied below:
\begin{equation}
\label{isomorphism-strings}
\begin{array}{rccccccl}
\left( H^*\PeakConf{n} :=\right)
& H^{*}\left( \TypeBConf{n}/\Z_2^n\right) 
&\overset{(a)}{\cong}& \left( H^{*}\TypeBConf{n} \right)^{\Z_2^n}
&\overset{(b)}{\cong}& \left( \kk \symmB{n} \right)^{\Z_2^n}
&\overset{(c)}{\cong}& \kk \symm{n} \\
\left( H^{2k}\PeakConf{n} :=\right)
& H^{2k}\left( \TypeBConf{n}/\Z_2^n\right) 
&\overset{(d)}{\cong}& \left( H^{2k}\TypeBConf{n} \right)^{\Z_2^n}
&\overset{(e)}{\cong}& \left( \left(\kk \symmB{n} \right)  \TypeBEulerianIdempotent{n}{n-k} \right)^{\Z_2^n}
&\overset{(f)}{\cong}& \left( \kk \symm{n} \right) \PeakIdempotent{n}{n-k}
\end{array}
\end{equation}

\vskip.1in
\noindent
{\sf Isomorphism (a):} 
Apply to $\Gamma:=\Z_2^n$ and $Y:=Y_n$ a general fact (see Hatcher \cite[Prop.~3G.1, p.~321]{Hatcher}):
when a finite group $\Gamma$ acts freely on a space $Y$, the surjection $Y \twoheadrightarrow Y/\Gamma $ induces a $\kk$-algebra map $H^*(Y/\Gamma) \rightarrow (H^* Y)^\Gamma$,
which is an isomorphism (with inverse the {\it transfer map}, up to a scalar) as long $|\Gamma|$ lies in $\kk^\times$.

For the remainder of the proof, note that since $\Gamma=\Z_2^n$ is a {\it normal} subgroup of $G=\symmB{n}$, one still has a well-defined action of the quotient group $G/\Gamma=\symmB{n}/\Z_2^n \cong \symm{n}$, both on
 \begin{itemize}
 \item the fixed $\kk$-subspace $U^\Gamma \subset U$ for any $\kk G$-module $U$, defined as $(g\Gamma)(u):= g(u)$, and
 \item the (topological) quotient space $Y/G$ of $G$-orbits $\{\Gamma y\}$, defined as $(g\Gamma) (\Gamma y):=\Gamma g(y)$.
\end{itemize}

\vskip.1in
\noindent
{\sf Isomorphism (b):}
Apply $(-)^{\Z_2^n}$ to the $\kk\symmB{n}$-module isomorphism $H^*Y_n \cong \kk \symmB{n}$
in Theorem~\ref{thm:Moseley} for $W=\symmB{n}$.

\vskip.1in
\noindent
{\sf Isomorphism (c):}
For any subgroup $\Gamma$ of a finite group $G$, one has an isomorphism
of (left-)$\kk G$-modules
$$
\begin{array}{rcl}
\kk[G/\Gamma] &\overset{\alpha}{\longrightarrow}& (\kk G)^\Gamma,\\
g\Gamma &\longmapsto &\displaystyle \sum_{h \in g\Gamma} h.
\end{array}
$$
Whenever $\Gamma$ is a normal subgroup, since $\Gamma$ lies in the kernel of both actions, this induces an isomorphism of (left-)$\kk G/\Gamma$-modules, and the action on $\kk[G/\Gamma]$ is the left-regular action for the group $G/\Gamma$.

\vskip.1in
\noindent
{\sf Isomorphism (d):}  The isomorphism from (a) respects
the cohomological grading.

\vskip.1in
\noindent
{\sf Isomorphism (e):} 
Apply $(-)^{\Z_2^n}$ to the $\kk\symmB{n}$-module isomorphism
$
H^{2k} Y_n \cong \left(\kk \symmB{n}\right) \TypeBEulerianIdempotent{n}{n-k}
$
in \eqref{type-B-Eulerian-idempotents-give-cohomology-reps}.

\vskip.1in
\noindent
{\sf Isomorphism (f):} 
We prove something slightly more general, under the following hypotheses.  Assume $\Gamma$ is a normal subgroup of a finite group $G$.  Let 
$\kk$ be a field with $|\Gamma| \in \kk^\times$,
and let $\varphi: \kk G \twoheadrightarrow \kk[G/\Gamma]$ be the $\kk$-algebra surjection 
induced by the quotient map $G \twoheadrightarrow G/\Gamma$.  Let $e_1,e_2,\ldots,e_\ell$ in $\kk G$
be complete orthogonal idempotents ($e_i^2=e_i, e_i e_j =0$ for $i \neq j$, and $1=\sum_{i=1}^\ell e_i$), so that their images $\hat{e}_i:=\varphi(e_i)$ are orthogonal idempotents in $\kk[G/\Gamma]$.
Then applying the following lemma, with $G=\symmB{n}, \Gamma=\Z_2^n, G/\Gamma=\symm{n}$ and $\{e_i\}=\{\TypeBEulerianIdempotent{n}{k}\}$
$\{\hat{e}_i\}=\{\PeakIdempotent{n}{k}\}$, gives the desired isomorphism (f).
\begin{lemma} 
\label{normal-subgroup-idempotent-lemma}
In the above setting, 
\begin{enumerate}
\item[(i)]  restricting the domain of $\beta:=\frac{1}{|\Gamma|} \varphi: \kk G \rightarrow \kk[G/\Gamma]$
to the $\Gamma$-fixed space $(\kk G)^\Gamma$ gives the inverse isomorphism to the isomorphism $\alpha: \kk[G/\Gamma] \rightarrow (\kk G)^\Gamma$ in (c), and
\item[(ii)] the isomorphsim $\beta: (\kk G)^\Gamma \rightarrow\kk [G/\Gamma]$ restricts, for each $i=1,2,\ldots,\ell$,
to an isomorphism
$$
\left((\kk G)e_i\right)^\Gamma \rightarrow
\kk[G/\Gamma]\hat{e}_i.
$$
\end{enumerate} 
\end{lemma}
\begin{proof}[Proof of Lemma~\ref{normal-subgroup-idempotent-lemma}.]
For (i), one checks directly that $\beta \circ \alpha=1_{\kk[G/\Gamma}$, and
and that $\alpha \circ \beta=1_{(\kk G)^\Gamma}$.

For (ii), consider the $\kk[G/\Gamma]$-module direct sum decomposition coming from the $\{\hat{e}_i\}$:
$$
\kk[G/\Gamma]=\bigoplus_{i=1}^\ell
\kk[G/\Gamma] \hat{e}_i.
$$
To compare this with $(\kk G)^\Gamma$,
introduce the idempotent $e_\Gamma:=\frac{1}{|\Gamma|} \sum_{\gamma \in \Gamma}$ in $\kk G$
that projects onto $\Gamma$-fixed spaces.
Note that since $\Gamma$ was assumed to be normal,
 this idempotent $e_\Gamma$ lies in the center of $\kk G$, commuting with all $\{e_i\}$.  This means that $\{e_i e_\Gamma\}_{i=1}^\ell$ are
 again orthogonal idempotents, summing to $e_\Gamma$, and giving this $\kk[G/\Gamma]$-module direct sum decomposition:
$$
(\kk G)^\Gamma
=\left( \bigoplus_{i=1}^\ell \kk G e_i\right)^\Gamma
=\bigoplus_{i=1}^\ell \left( \kk G e_i\right)^\Gamma
=\bigoplus_{i=1}^\ell (\kk G) e_i e_\Gamma
=\bigoplus_{i=1}^\ell (\kk G) e_\Gamma e_i.
$$
It then suffices to check for each $i$ that the isomorphism $\beta$ maps
$(\kk G) e_\Gamma e_i$ into $\kk[G/\Gamma]\hat{e}_i$.  
Since $\beta=\frac{1}{|\Gamma|}$ is a scaling of the ring homomorphism 
$\varphi$, this calculation shows just that:
$
\varphi(e_\Gamma e_i)= 
\varphi(e_\Gamma) \cdot \varphi(e_i)=
\varphi(e_\Gamma) \cdot \hat{e}_i. \qedhere
$
\end{proof}
With Lemma~\ref{normal-subgroup-idempotent-lemma} proven, the proof
of Theorem~\ref{peak-idempotents-give-cohomology-reps} $(i),(ii)$ is also complete.
\end{proof}

\section{Presentations of $H^*\TypeAConf{n}$ and $H^*\TypeBConf{n}$}
\label{sec:presentations-decompositions}

Our goal in the next few sections is to give
explicit presentations of $H^*\TypeAConf{n}$ and $H^*\TypeBConf{n}$. We  start with their presentations as the Varchenko-Gelfand rings $\VG(\A_{\symm{n}}),\VG(\B_{\symm{n}})$,
where we describe known quadratic Gr\"obner basis presentations.
We then show that in the case of $H^*\TypeBConf{n}$ there is a change-of-basis, valid when $2$ lies in $\kk^\times$, that diagonalizes the action of $\Z_2^n$ inside $\symmB{n}$.

\subsection{Quadratic Gr\"obner basis presentations
for $H^*\TypeAConf{n}$ and $H^*\TypeBConf{n}$}

We first specialize the presentation of the graded Varchenko-Gelfand ring $\VG(\A)$ given in 
\eqref{VG-presentation} and \eqref{VG-circuit-relation} to the case of the reflection arrangements
$\A_{\symm{n}}, \A_{\symmB{n}}$
of types $A$ and $B$. Recall by our discussion in Section \ref{sec:VGrings} that this will give a presentation of the cohomology of two of the configuration
spaces $X_n, Y_n$ defined in \eqref{typeAB-configuration-space-definitions} of the introduction, taking
into account their reformulations in \eqref{configuration-spaces-reexpressed}. The quadratic presentations for
$H^* \TypeAConf{n}, H^* \TypeBConf{n}$ are originally due to Cohen \cite{cohen} 
and Xicot{\'e}ncatl  \cite{xico}, respectively.

We synthesize these known presentations below in Theorem \ref{thm:AB_VGpresentation}. The general presentation of $\VG(\A)$ simplifies significantly in this context, leading to quadratic Gr\"obner bases for the defining ideals of $\VG(\A_{\symm{n}}), \VG(\A_{\symmB{n}})$, and simple structure
for their standard/nbc-monomial $\kk$-bases and resulting Hilbert series.

\begin{theorem}
\label{thm:AB_VGpresentation}
Let $\kk$ be any field or $\Z$.
\begin{itemize}
\item[(i)]
One has graded $\kk$-algebra isomorphisms,
doubling the grading for the left isomorphism on each line,
\begin{align*} 
H^*(X_n,\kk)  &\cong \VG(\A_{\symm{n}}) \cong \kk [t_{ij}]/\J_{\symm{n}},\\
H^*(Y_n,\kk) &\cong \VG(\A_{\symmB{n}}) \cong\kk[ u_{ij}^+, u_{ij}^-, u_i]/\J_{\symmB{n}},
\end{align*}
where $\J_{\symm{n}}$ and $\J_{\symmB{n}}$ are generated by the quadratic relations listed in Table \ref{table:AB_VG}.
\item[(ii)]
For fields $\kk$, these relations form quadratic Gr\"obner bases, with initial terms shown underlined in Table \ref{table:AB_VG}, if one uses any monomial order $\prec$ that orders their variables so as to extend the partial orders shown below:
\begin{align*}
  &\underbrace{\{ t_{12} \}}_{T_1} 
  \prec  \underbrace{\{ t_{13}, \ t_{23}  \}}_{T_2} 
  \prec \underbrace{\{ t_{14}, \ t_{24} , \  t_{34} \}}_{T_3}
  \prec \cdots 
  \prec \underbrace{\{  t_{1n}, \ t_{2n}, \ \cdots, t_{(n-1)n} \}}_{T_{n-1}},\\
 &\underbrace{\{ u_1 \}}_{U_1}
 \prec \ \underbrace{\{ u_2, \  u_{12}^+, \ u_{12}^- \}}_{U_2} 
 \prec \cdots \prec \underbrace{\{ u_n, \ u_{1n}^+, \ u_{1n}^-, \cdots, u_{(n-1)n}^+, \ u_{(n-1)n}^- \}}_{U_n}.
 \end{align*}
\item[(iii)]
With respect to monomial orders as in (ii),
the standard (nbc-)monomial $\kk$-bases for $H^* \TypeAConf{n}, H^*\TypeBConf{n}$ are the
squarefree monomial products obtained by choosing
at most one variable from each set $T_i$ (respectively,
from each set $U_i$);  call these
the $\typeabasis$-basis of $H^* \TypeAConf{n}$ and $\ubasis$-basis of $H^* \TypeBConf{n}$. 
\item[(iv)] In particular, these $\kk$-bases lead to these Hilbert series expressions:
\begin{align*}
\mathrm{Hilb}(\VG(\A_{\symm{n}}),t)&=
\prod_{i=1}^{n-1} (1+|T_i|t)
= (1+t)(1+2t)(1+3t)\cdots (1+(n-1)t),\\
\mathrm{Hilb}(\VG(\A_{\symmB{n}}),t)&=\prod_{i=1}^{n} (1+|U_i|t) =(1+t)(1+3t)(1+5t) \cdots (1+(2n-1)t).
\end{align*}
\end{itemize}
\end{theorem}
\noindent
\begin{center}
\begin{table}[!h]
\centering
\setlength{\tabcolsep}{10pt} 
\renewcommand{\arraystretch}{1.5} 
\begin{tabular}{|l|l|}  \hline
\rm
$\J_{\symm{n}}$ generating relations & \rm $\J_{\symmB{n}}$  generating relations \\ \hline\hline 
$\underline{t_{ij}^2}$ &  $\underline{(u_{ij}^{+})^2}, \ \underline{(u_{ij}^{-})^2}, \ \underline{u_i^2} $ \\
$t_{ij}t_{jk} - t_{ik}t_{ij} - \underline{t_{ik}t_{jk}}$ & $u_i u_{ij}^+ - u_i u_{ij}^{-} - \underline{u_{ij}^+ u_{ij}^-}$ \\
 & $u_i u_{ij}^+ - u_i u_j - \underline{u_{ij}^+ u_j}$ \\ 
 &$u_i u_j - u_i u_{ij}^- - \underline{u_j u_{ij}^-}$ \\
 &$u_{ij}^{+}u_{jk}^{+} - u_{ij}^{+}u_{ik}^{+} - \underline{u_{ik}^{+}u_{jk}^{+}} $ \\
 &$u_{ij}^{-}u_{jk}^{+} - u_{ij}^{-}u_{ik}^{-} - \underline{u_{ik}^{-}u_{jk}^{+}}$\\
 &$-u_{ij}^{-}u_{jk}^{-} + u_{ij}^{-}u_{ik}^{+} - \underline{u_{ik}^{+}u_{jk}^{-}} $  \\
 &$-u_{ij}^{+}u_{jk}^{-} + u_{ij}^{+}u_{ik}^{-} - \underline{u_{ik}^{-}u_{jk}^{-}} $ \\
 &\\
 \hline 
\end{tabular}
\caption{Relations in the Type A and B graded Varchenko-Gelfand rings.  Here $i,j,k$ are always assumed to be ordered with $i<j<k$, and we have underlined the $\prec$-initial terms in each relation with respect to the
monomial orderings $\prec$ in Theorem~\ref{thm:AB_VGpresentation}.}
\label{table:AB_VG}
\end{table}
\end{center}

\begin{proof}
    For assertions (i), (ii), first pick a monomial ordering $\prec$ as in (ii). One already has Gr\"obner bases as in Remark~\ref{VG-grobner-basis-remark} given by all of the relations $\partial(C)$ from \eqref{VG-circuit-relation}
    indexed by circuits $C=\{i_1,i_2,\ldots,i_s\}$. Hence their initial terms
    $\init_\prec(\partial(C))$, which are the broken circuit monomials $u_{i_2} u_{i_3} \cdots u_{i_s}$ with $i_1=\min_{\prec} C$,
     will generate the monomial initial
    ideal $(\init_\prec(f): f \in \J)$.  However,
    for the arrangements $\A_{\symm{n}}, \A_{\symmB{n}}$, one can check that every circuit $C$ has its broken circuit divisible by
    a {\it quadratic} broken circuit, that is, coming from a circuit $C$ of size $3$.  In fact, this property (minimal broken circuits under divisiblity are all quadratic) characterizes the hyperplane arrangements which are {\it supersolvable} by a result of Bj\"orner and Ziegler \cite{BjornerZiegler}; see also Peeva \cite[\S 4]{Peeva} and Dorpalen-Barry \cite[\S 4]{Dorpalen-Barry}.  One can then easily check that the three element circuits $C$ for $\A_{\symm{n}}, \A_{\symmB{n}}$ give exactly the circuit relations $\partial(C)$ shown in Table~\ref{table:AB_VG}, with their $\prec$-initial terms as underlined there.

    Assertion (iii) follows by checking that the standard monomials, that is, those divisible by none of the initial terms underlined in Table~\ref{table:AB_VG}, must be squarefree and avoid having two variables from any $\T_i$ or $\U_i$.

    Assertion (iv) is then a consequence of (iii).
    \end{proof}

\begin{remark} \rm
\label{rmk:choice-of-roots}
To define the action of $W=\symm{n}, \symmB{n}$ on the rings
$\VG(\A_{\symm{n}}), \VG(\A_{\symmB{n}})$,
the rule \eqref{W-action-on-VG-variables}
dictates that we specify our choice of linear forms
$\alpha_H$ to define the hyperplanes $H=\ker \alpha_H$ that correspond to
each of the variables in $\{ t_{ij} \}$ in $\VG(\A_{\symm{n}})$, and the variables
in $\{ u_i,u^+_{ij},u^-_{ij} \}$ in $\VG(\A_{\symmB{n}})$ for $1 \leq i < j \leq n$:
$$
\begin{array}{rcll}
t_{ij} &\leftrightarrow& H=\ker(+x_j-x_i),\\
& \\
u_i &\leftrightarrow& H=\ker(+x_i), \\
u^+_{ij} &\leftrightarrow& H=\ker(+x_j-x_i),\\
u^-_{ij} &\leftrightarrow& H=\ker(+x_j+x_i).
\end{array}
$$
One can check this implies that $\sigma \in \symm{n}$ acts via 
$
\sigma(t_{ij})=t_{\sigma(i) \sigma(j)},
$
with convention $t_{ji}:=-t_{ij}$ for $i < j$. 

For $\symmB{n}$, the above choices of linear forms imply the following actions of the sign changes $\tau_i,\tau_j$ in $\symmB{n}$ for $1 \leq i < j \leq n$ on $u_i,u+_{ij},u^-_{ij}$ in $\VG(\A_{\symmB{n}})$:

\begin{center}
\centering
\setlength{\tabcolsep}{10pt} 
\renewcommand{\arraystretch}{1.25} 
\begin{tabular}{|l|l|l|l|}\hline
Variable  &Image       & Image    & Image under \\
generator &under $\tau_i$ & under $\tau_j$  &under $t_k$ for $ k \neq i,j$ \\\hline \hline
  $u_{i}$&$-u_{i}$&$u_{i}$& $u_{i}$ \\ \hline 
  $u_{ij}^+$ & $u_{ij}^-$ & $-u_{ij}^-$& $u_{ij}^+$ \\ \hline
 $u_{ij}^-$ & $u_{ij}^+$ & $-u_{ij}^+$&$u_{ij}^-$ \\ \hline
\end{tabular}
\end{center}
For other signed permutations $\sigma$ in $\symmB{n}$,
the rules are a bit more complicated
to write down in all cases.  However,
if one lets $\hat{\sigma}=\varphi(\sigma)$ denote the (unsigned) permutation in $\symm{n}$ which forgets the
signs in $\sigma$, as in \eqref{sign-forgetting-map}, then
any variable $u_i$ will have 
\begin{equation}
\label{signed-perms-on-u-variables}
\sigma(u_i) = \pm u_{\hat{\sigma}(i)}.
\end{equation}
Furthermore, the variables $u^+_{ij}, u^-_{ij}$ 
will have $\sigma(u^+_{ij}), \sigma(u^-_{ij})$
taking one of the forms
$\pm u^+_{\hat{\sigma}(i)\hat{\sigma}(j)}$
or $\pm u^-_{\hat{\sigma}(i)\hat{\sigma}(j)}$.

As an example, for the signed permutation 
$$
\sigma=(+4,+3,-5,-1,+2)=\left( 
\begin{smallmatrix} 1&2&3&4&5\\
+4,&+3,&-5,&-1,&+2 \end{smallmatrix} \right),
$$
one has $\sigma(x_4-x_2)=-x_3-x_1=-(x_3+x_1)$,
and thus $\sigma(u^+_{24})=-u^-_{13}$.
\end{remark}

From here on, we abbreviate the notation for
these rings as follows, to emphasize the above presentations:
\begin{align*}
    \TypeACohomology{n}&:= H^* \TypeAConf{n} \cong \VG(\A_{\symm{n}}),\\
    \TypeBCohomology{n}&:= H^* \TypeBConf{n} \cong VG(\A_{\symmB{n}}). 
\end{align*}
Recall the discussion from Section \ref{sec:flats}
showed that flat orbits $\lat(\A_W)/W$ for $W=\symm{n}$ are parametrzied by partitions $\lambda$ with $|\lambda|=n$, while those for
$W=\symmB{n}$ are parametrzied by partitions $\mu$ with $|\mu|\leq n$.  We will therefore similarly abbreviate the flat-orbit decompositions for $\TypeACohomology{n}, \TypeBCohomology{n}$ as follows:
\begin{equation}
\label{eq:typeABflatorbitdecomposition} 
\begin{aligned}
\TypeACohomology{n} &= \bigoplus_{|\lambda|=n} \TypeACohomology{\lambda}
 &\text{ where } \quad
\TypeACohomology{\lambda} := \bigoplus_{\substack{X\in \lat(\A_{\symm{n}})\\ \sh{X} = \lambda}} \VG(\A_{\symm{n}})_{X},\\
\TypeBCohomology{n} &= \bigoplus_{|\mu| \leq n} (\TypeBCohomology{n})_\mu
 &\text{ where } \quad 
(\TypeBCohomology{n})_\mu := \bigoplus_{\substack{X\in \lat(\A_{\symmB{n}})\\ \sh{X} = \mu}} \VG(\A_{\symmB{n}})_{X}.
\end{aligned}
\end{equation}

\subsection{Diagonalizing the $\Z_2^n$-action on $H^* \TypeBConf{n}$} 
\label{eigengbasis-subsection}

 Our next step in analyzing 
 the ring $\TypeBCohomology{n}=H^* \TypeBConf{n}$ 
 is to make an invertible change-of-variables,
 diagonalizing the action of the normal subgroup $\Z_2^n=\langle \tau_1,\ldots,\tau_n\rangle$ within $\symmB{n}$. Here $\tau_i$ denotes reflection that performs a sign change in the 
 $i^{th}$ coordinate.
  It will be important for the remainder of this section to assume that {\bf the field $\kk$ does not have characteristic two}, that is, $2 \in \kk^\times$.   We will introduce a new basis for $\TypeBCohomology{n}$, a filtration using that basis, and a corresponding associated graded ring. Along the way, we will see several useful decompositions of $\TypeBCohomology{n}$. 

\begin{definition}\label{def:basismap}\rm
Define an isomorphism of graded $\kk$-algebras $\basismap$ by
   \begin{align*}
       \basismap: \kk[v_{ij}, w_{ij}, u_i] &\longrightarrow \kk[u_{ij}^+, u_{ij}^-, u_i]\\
       u_i & \longmapsto u_i\\
       v_{ij} &\longmapsto   u^+_{ij} + u_{ij}^-\\
      w_{ij} & \longmapsto u_{ij}^+ - u_{ij}^-
   \end{align*} 
   where $1 \leq i < j \leq n$, and 
whose inverse $\basismap^{-1}$ sends $\basismap^{-1}(u_{ij}^+) = \frac{1}{2}(v_{ij} + w_{ij}) $ and $\basismap^{-1}(u_{ij}^-) = \frac{1}{2}(v_{ij} - w_{ij})$.
\end{definition}

The point of introducing the generators $v_{ij}$ and $w_{ij}$ is that, after defining the action of $\symmB{n}$ on $\kk[u_i,v_{ij},w_{ij}]$ so as to make the isomorphism $\B$ equivariant, the reflections $\tau_i$ generating  $\Z_2^n$ inside $\symmB{n}$ act on the variables $v_{ij}, w_{ij}$ diagonally via $\pm 1$ as shown in Table \ref{table:generatoraction}, which extends the table from  Remark~\ref{rmk:choice-of-roots}. This will be crucial when considering the $\Z_2^n$-invariant subalgebra of $\TypeBCohomology{n}$ in the next Section \ref{sec:ringZn}. 

\begin{center}
\centering
\setlength{\tabcolsep}{10pt} 
\renewcommand{\arraystretch}{1.25} 
\begin{table}[!h]
\begin{tabular}{|l|l|l|}\hline
Variable  &Image       & Image    \\
generator &under $\tau_i$ & under $\tau_j$ 
\\\hline \hline
  $u_{i}$&$-u_{i}$&$u_{i}$
\\ \hline \hline
 $v_{ij}$ &$v_{ij}$ & $-v_{ij}$ 
 \\ \hline
 $w_{ij}$ & $-w_{ij}$ & $w_{ij}$ 
 \\ \hline \hline
  $u_{ij}^+$ & $u_{ij}^-$ & $-u_{ij}^-$
  \\ \hline
 $u_{ij}^-$ & $u_{ij}^+$ & $-u_{ij}^+$
 \\ \hline
\end{tabular}
\caption{Action of $\tau_i,\tau_j$ on
the diagonalized variables $u_i,v_{ij},w_{ij}$.}
\label{table:generatoraction}
\end{table}
\end{center}
\noindent

\begin{remark} \rm
\label{rmk:signed-perms-on-uvw}
It is again somewhat complicated to write down the action of an arbitrary $\sigma$ in $\symmB{n}$ on all variables $v_{ij},w_{ij}$.  But similar to Remark~\ref{rmk:choice-of-roots}, if
$\hat{\sigma}$ in $\symm{n}$ is the image of
$\sigma$ under $\varphi:\symmB{n} \twoheadrightarrow \symm{n}$, then
the variables $v_{ij}, w_{ij}$ 
will have $\sigma(v_{ij}), \sigma(w_{ij})$
taking one of the forms
$\pm v_{\hat{\sigma}(i)\hat{\sigma}(j)}$
or $\pm w_{\hat{\sigma}(i)\hat{\sigma}(j)}$.
\end{remark}

We wish to rewrite the presentation 
$\TypeBCohomology{n}:=
\kk[ u_{ij}^+, u_{ij}^-, u_i]/\J_{\symmB{n}}$
in terms of these new variables $v_{ij}, w_{ij}$, using a Gr\"obner basis argument.  Pick any monomial ordering $\prec$ on $\kk[u_i,v_{ij},w_{ij}]$ which orders the variables in a way that extends the following partial order: 
\begin{equation}
\label{uvw-variable-partial-order}
 \underbrace{\{ u_1 \}}_{V_1:=} 
 \prec  \underbrace{\{ u_2, \  v_{12}, \ w_{12} \}}_{V_2:=}
 \prec \cdots  
 \prec \underbrace{\{ u_n, \ v_{1n}, \ w_{1n}, \cdots, v_{(n-1)n}, \ w_{(n-1)n} \}}_{V_n:=}. 
\end{equation}

\begin{theorem}
\label{theorem:newpresentation}
The isomorphism $\basismap: \kk[v_{ij}, w_{ij}, u_i] \longrightarrow \kk[u_{ij}^+, u_{ij}^-, u_i]$
induces a graded
$\kk$-algebra isomorphism
\[ 
\kk [v_{ij},w_{ij}, u_{i} ] / \I 
\overset{\sim}{\longrightarrow}
\kk[ u_{ij}^+, u_{ij}^-, u_i]/\J_{\symmB{n}}  =:\TypeBCohomology{n},
\]
where $\I$ is generated by the relations $\GGG$ listed in Table \ref{table:Irelations}. 

Moreover, $\GGG$ gives a quadratic Gr\"obner basis 
for the ideal $\I$ with respect to $\prec$. Here the standard monomial $\kk$-basis for the quotient
$\kk [v_{ij},w_{ij}, u_{i} ] / \I$
is the set of  squarefree monomials $\vbasis$
obtained 
from taking products with at most one element from each of these sets $V_1,V_2,\ldots,V_n$
defined in \eqref{uvw-variable-partial-order}.
\end{theorem}

\noindent We call the presentation of $\TypeBCohomology{n}$ in Theorem~\ref{thm:AB_VGpresentation} its {\it $\ubasis$-presentation} and the presentation in Theorem~\ref{theorem:newpresentation} its {\it $\vbasis$-presentation}.

\begin{center}
\begin{table}[!h]
\centering
\setlength{\tabcolsep}{10pt} 
\renewcommand{\arraystretch}{1.5} 
\begin{tabular}{|l|}  \hline
\rm
\rm Generating relations $\GGG$ for $\I$  \\ \hline\hline 
    $\underline{u_{i}^2}  $\\
 $\underline{v_{ij}w_{ij}}$\\
  $u_{i}w_{ij} -\underline{u_{j} v_{ij}}$\\ 
  $ - 2u_{i} w_{ij} + \underline{v_{ij}^2}$\\
  $ + 2 u_{i} w_{ij} + \underline{w_{ij}^{2}}$\\
 $u_{i} v_{ij} - 2 u_{i} u_{j} - \underline{u_{j} w_{ij}}$\\
  $v_{ij}w_{jk} - w_{ij}w_{ik} - \underline{v_{ik}v_{jk}}$\\
  $w_{ij}w_{jk} - v_{ij}w_{ik} - \underline{w_{ik}w_{jk}}$\\
 $v_{ij}v_{jk} - v_{ij}v_{ik} - \underline{v_{ik}w_{jk}}$ \\
  $w_{ij}v_{jk} - w_{ij}v_{ik} - \underline{w_{ik}v_{jk}}$ \\
  \\
 \hline 
\end{tabular}
\caption{Generating relations $\GGG$ for $\I$,
with indices $i,j,k$ assumed to satisfy $i<j<k$.  Initial terms with respect to any monomial order $\prec$ extending \eqref{uvw-variable-partial-order} are underlined.}\label{table:Irelations}
\end{table}
\end{center}
\begin{proof}[Proof of Theorem~\ref{theorem:newpresentation}]
Consider the composite surjection
\[ \tilde{\basismap}: \kk[u_i,v_{ij}, w_{ij}] \longrightarrow \kk[u_i,u_{ij}^+, u_{ij}^-] \longrightarrow \kk[u_i,u_{ij}^+, u_{ij}^-]/\J_{\symmB{n}}\]
which applies the isomorphism $\basismap$ followed by the surjection $\kk[u_i,u_{ij}^+, u_{ij}^-]/\J_{\symmB{n}} \cong \TypeBCohomology{n}$. 
It is straightforward (though slightly tedious) to check that each generator $g$ of the ideal $\I$ listed in Table~\ref{table:Irelations} is in fact in $\ker(\tilde{\basismap})$, by checking that $\basismap(g)$ lies in $\J_{\symmB{n}}$. 
For example, here is the check for the generator  
$u_{i}w_{ij} - u_{j} v_{ij}$ of $\I$:
\begin{align*}
     \basismap( u_{i}w_{ij} - u_{j} v_{ij}) &= u_i (u_{ij}^+ - u_{ij}^-) -  u_j (u_{ij}^+ + u_{ij}^-)\\ 
     &= u_i u_{ij}^+ - u_j u_{ij}^+ - u_i u_{ij}^- - u_j u_{ij}^- \\
     &= u_i u_j - u_i u_j = 0 \quad ( \in \J_{\symmB{n}}).
\end{align*}
As another example, here is the check for the generator $v_{ij}w_{jk}- w_{ij}w_{ik} - v_{ik}v_{jk}$ of $\I$:
\begin{align}
\notag
    \basismap(v_{ij}w_{jk}- w_{ij}w_{ik} - v_{ik}v_{jk})=& \ (u_{ij}^+ + u_{ij}^-)(u_{jk}^+ - u_{jk}^-) - (u_{ij}^+ - u_{ij}^-)(u_{ik}-u_{ik}^-) - (u_{ik}^+ + u_{ik}^-)(u_{jk}^+ + u_{jk}^-)\\
    \notag
    =& \ u_{ij}^+ u_{jk}^+ + u_{ij}^- u_{jk}^+ - u_{ij}^+ u_{jk}^- - u_{ij}^- u_{jk}^- -u_{ij}^+ u_{ik}^+ + u_{ij}^- u_{ik}^+\\
    \notag
    &+ u_{ij}^+u_{ik}^- - u_{ij}^- u_{ik}^- - u_{ik}^+ u_{jk}^+ - u_{ik}^- u_{jk}^+ - u_{ik}^+ u_{jk}^- - u_{ik}^- u_{jk}^- \\
    \notag
    =& \ \big( u_{ij}^+ u_{jk}^+  -u_{ij}^+ u_{ik}^+ -  - u_{ik}^+ u_{jk}^+   \big) + \big( u_{ij}^- u_{jk}^+ - u_{ij}^- u_{ik}^- - u_{ik}^- u_{jk}^+ \big) \\
    \label{fourth-and-fith-lines}
    & + \big( -u_{ij}^+u_{jk}^- + u_{ij}^+ u_{ik}^- - u_{ik}^- u_{jk}^- \big) + \big( -u_{ij}^- u_{jk}^- + u_{ij}^- u_{ik}^+ - u_{ik}^+ u_{jk}^- \big)
\end{align}
which lies in $\J_{\symmB{n}}$ since the summands
on the right side of equality \eqref{fourth-and-fith-lines} all appear in $\J_{\symmB{n}}$. The remaining checks follow from very similar computations.   

Thus the ideal $\I$ generated by $\GGG$ has an induced $\kk$-algebra surjection 
$$
\kk [v_{ij},w_{ij}, u_{i} ] / \I 
\twoheadrightarrow
\kk[ u_{ij}^+, u_{ij}^-, u_i]/\J_{\symmB{n}}.
$$
Note that Theorem \ref{thm:Moseley} implies the target $\kk[ u_{ij}^+, u_{ij}^-, u_i]/\J_{\symmB{n}} \cong \VG(\A_{\symmB{n}})$ has $\kk$-dimension 
$$
|\symmB{n}|=2^n \cdot n! = 2 \cdot 4 \cdots (2n-2) \cdot 2n=|\V|.
$$
Thus by dimension-counting, it suffices to check that
the images of $\vbasis$ give a $\kk$-spanning set for the source 
$\kk [v_{ij},w_{ij}, u_{i} ] / \I$.
However, note from inspection of Table~\ref{table:Irelations} that 
$\vbasis$ is exactly the set of {\it $\prec$-standard monomials} with respect to
$\GGG$, meaning the monomials that
are divisible by none of the $\prec$-initial terms of any
of the elements of $\GGG$.  Therefore the {\it multivariate division algorithm} with respect to $\GGG$ using $\prec$ (see, e.g., Cox, Little and O'Shea \cite[\S 2.3]{CoxLittleOShea}) shows that the images of $\vbasis$ will $\kk$-span $\kk [v_{ij},w_{ij}, u_{i} ] / \I$.
This also shows that every $\prec$-initial term of an element of $I$ is divisible by some $\prec$-initial term of an element of $\GGG$, that is, $\GGG$ gives a Gr\"obner basis for $\I$ with respect to $\prec$.
\end{proof}

\section{Proof of Theorem~\ref{peak-idempotents-give-cohomology-reps}(iii): vanishing of cohomology and peak idempotents}
\label{sec:ringZn} 

We recall the statement of 
 Theorem~\ref{peak-idempotents-give-cohomology-reps} here.

\vskip.1in
\noindent
{\bf Theorem~\ref{peak-idempotents-give-cohomology-reps}.}
{\it
Let $H^* Z_n$ be the cohomology with coefficients
in a field $\kk$ of 
$
Z_n=\TypeBConf{n}/\Z_2^n=\Conf_n(\RP^2 \times \R) .
$
\begin{itemize}
\item[(i)] When $\fieldchar(\kk)>2$, the $\kk \symm{n}$-module on the total cohomology is
$
H^* \PeakConf{n} \cong \kk \symm{n}.
$
\item[(ii)] When $\fieldchar(\kk)>n$, for  $0 \leq k \leq n$ one has a $\kk \symm{n}$-module isomorphism 
$
H^{2k} \PeakConf{n} \cong \left( \kk \symm{n} \right) \PeakIdempotent{n}{n-k}.
$
\item[(iii)] If $\fieldchar(\kk)>2$ then  $H^i\PeakConf{n}=0$
unless $i \equiv 0 \bmod 4$.
If $\fieldchar(\kk)>n$ then $\PeakIdempotent{n}{k}=0$ 
unless $k \equiv n \bmod 2$. 
\end{itemize}
}
\vskip.1in

The proof strategy for (iii) is fairly simple.
Recall the proof of Theorem~\ref{peak-idempotents-give-cohomology-reps}(i),
in \eqref{isomorphism-strings},
identified 
$$
H^*\PeakConf{n} \cong \left( H^{*}\TypeBConf{n} \right)^{\Z_2^n} = (\TypeBCohomology{n})^{\Z_2^n}.
$$
We examine the action of $\Z_2^n=\langle \tau_1,\ldots,\tau_n\rangle$ in $\symmB{n}$
on $\TypeBCohomology{n}$, and
construct  $\Z_2^n$-invariants.
Recall from Theorem~\ref{theorem:newpresentation} the
$\kk$-basis $\vbasis$ for $\TypeBCohomology{n}$, the set of
squarefree products of at most one variable
from each of $V_1,\ldots,V_n$:

\begin{equation*}
 \underbrace{\{ u_1 \}}_{V_1:=} 
 \prec  \underbrace{\{ u_2, \  v_{12}, \ w_{12} \}}_{V_2:=}
 \prec \cdots  
 \prec \underbrace{\{ u_n, \ v_{1n}, \ w_{1n}, \cdots, v_{(n-1)n}, \ w_{(n-1)n} \}}_{V_n:=}. 
\end{equation*}
Table~\ref{table:generatoraction} 
showed how the generators $\tau_1,\ldots,\tau_n$
of $\Z_2^n$ scale all variables $u_i, v_{ij}, w_{ij}$ by $\pm 1$ signs: 
\begin{align*}
\tau_k(u_i)
&=\begin{cases} 
    -u_i&\text{ if }k = i\\
    u_i&\text{ in all other cases},\end{cases}\\
    \tau_k(v_{ij})&=\begin{cases} 
    -v_{ij}&\text{ if }k = j (>i)\\
    v_{ij}&\text{ in all other cases},\end{cases}\\
     \tau_k(w_{ij})&=\begin{cases} 
    -w_{ij}&\text{ if }k = i (<j)\\
    w_{ij}&\text{ in all other cases}.\end{cases}
\end{align*}
From this, one can check that
none of the variables are $\Z_2^n$-invariant.  However, certain squarefree {\it quadratic} monomials
in $\vbasis$ will be $\Z_2^n$-invariant, namely
\begin{equation}
\label{Z-generators}
\quadgens:= \{ u_{i} w_{ij} \}_{1 \leq i < j \leq n}
\quad \sqcup \quad 
\{ w_{ij}w_{ik} \}_{1 \leq i < j < k \leq n}
\quad \sqcup \quad
\{ v_{ij}w_{jk} \}_{1 \leq i < j < k \leq n}.
\end{equation}
This can be verified with the following easy checks (bearing in mind that $\tau_k$ fixes $u_i, v_{ij}, w_{ij}$ if $k \neq i,j$):
$$
\begin{array}{lcl}
        \tau_i ( u_{i} w_{ij} )= (-u_i) (-w_{ij}) = u_i w_{ij}& &\\
        \tau_j ( u_i w_{ij} )= u_i w_{ij} & &\\
        & &\\
        \tau_i ( w_{ij}w_{ik} ) = (-w_{ij})(-w_{ik}) = w_{ij}w_{ik}& &\tau_i ( v_{ij}w_{jk} )= v_{ij}w_{jk}\\
        \tau_j ( w_{ij}w_{ik} )= w_{ij}w_{ik} & &\tau_j ( v_{ij}w_{jk} )= (-v_{ij})(-w_{jk}= v_{ij})w_{jk})\\
        \tau_k (  w_{ij}w_{ik} ) =  w_{ij}w_{ik}& &\tau_k (  v_{ij}w_{jk} )= v_{ij}w_{jk}.\\
\end{array}
$$
Consequently, any product $q_1 q_2 \cdots q_k$ of quadratics $q_i$ in $\quadgens$
is $\Z_2^n$-invariant.
Letting $\prod \quadgens$ denote the set of
all such products $q_1 q_2 \cdots q_k$ from $\quadgens$,
we will show the following.

\begin{theorem}
\label{thm:tvbasis}
The monomials in $\tvbasis$ have their images
in $H^* \PeakConf{n}=(\TypeBCohomology{n})^{\Z_2^n}$ forming a $\kk$-basis.
\end{theorem}

Assuming Thorem \ref{thm:tvbasis} for the moment, we can easily prove  Theorem~\ref{peak-idempotents-give-cohomology-reps}(iii).
\begin{proof}[Proof of Theorem~\ref{peak-idempotents-give-cohomology-reps}(iii),
assuming Theorem~\ref{thm:tvbasis}]
Since $u_i, v_{ij}, w_{ij}$ lie in degree $1$ of $\TypeBCohomology{n}$, they lie in $H^2 \TypeBConf{n}$,
and the quadratics $\quadgens$ lie in 
$H^4 \PeakConf{n} \cong \left( H^4\TypeBConf{n} \right)^{\Z_2^n}$.
Theorem~\ref{thm:tvbasis} then forces
$H^i \PeakConf{n}=0$ unless $i \equiv 0 \bmod 4$.  

For the assertion on when $\PeakIdempotent{n}{k}$ vanishes,
replace $k$ by $n-k$ in  
Theorem~\ref{peak-idempotents-give-cohomology-reps}(ii) to give
$$
H^{2(n-k)} \PeakConf{n} \cong \left( \kk \symm{n} \right) \PeakIdempotent{n}{k}.
$$
This implies $\PeakIdempotent{n}{k}=0$ unless $2(n-k) \equiv 0 \bmod{4}$, that is,
unless $n \equiv k \bmod{2}$.
\end{proof}

We will now prove Theorem~\ref{thm:tvbasis}, which takes a bit more work.
\begin{proof}[Proof of Theorem~\ref{thm:tvbasis}]

Note the images of $\tvbasis$
are $\kk$-linearly independent in $H^* \PeakConf{n}=(\TypeBCohomology{n})^{\Z_2^n}$, being a subset of the $\kk$-basis $\vbasis$ for
the larger space $\TypeBCohomology{n}$. As Theorem~\ref{peak-idempotents-give-cohomology-reps}(i) implies 
$\dim_{\kk}  H^*(\PeakConf{n})=n!$, 
one concludes 
$$
| \tvbasis | \leq n!.
$$
By dimension-counting, it suffices to
prove the opposite inequality 
$
|\tvbasis| \geq n!,
$
which would then necessarily be an equality,
and also imply Theorem~\ref{thm:tvbasis}.
We deduce $|\tvbasis|\geq n!$ from something with more payoff in Section~\ref{sec:filtrations}:  
the Pairing Lemma~\ref{lem:pairing-bijection} below, giving an injection $\phi: \typeabasis \hookrightarrow \tvbasis$, where $\typeabasis$
was the $\kk$-basis of
cardinality $n!$ for $\TypeACohomology{n}=  \VG(\A_{\symm{n}}) \cong \kk [t_{ij}]/\J_{\symm{n}}$
from Theorem~\ref{thm:AB_VGpresentation}.
That is, $\typeabasis$ consists of squarefree products
of at most variable from each set $T_1,\ldots,T_{n-1}$ here:
$$
\underbrace{\{ t_{12} \}}_{T_1} 
  \prec  \underbrace{\{ t_{13}, \ t_{23}  \}}_{T_2} 
  \prec \underbrace{\{ t_{14}, \ t_{24} , \  t_{34} \}}_{T_3}
  \prec \cdots 
  \prec \underbrace{\{  t_{1n}, \ t_{2n}, \ \cdots, t_{(n-1)n} \}}_{T_{n-1}}
$$
This injection $\phi$ will also be well-behaved with respect to this map of $\kk$-algebras (and $\kk \symm{n}$-modules):
\begin{equation}
\label{def:gamma-surjection}
\begin{array}{rcl}
\kk[u_i,v_{ij},w_{ij}] &\overset{\gamma}{\longrightarrow}& \kk[t_{ij}]\\
u_i &\longmapsto&1,\\
v_{ij} & \longmapsto & t_{ij}\\
w_{ij} & \longmapsto & t_{ij}.
\end{array}
\end{equation}

\begin{lemma}
(Pairing Lemma) 
\label{lem:pairing-bijection}
There exists an injection (and hence bijection) 
    $$
    \phi: \typeabasis \hookrightarrow \tvbasis
    $$
    with the property that $\gamma(\phi(m))=m$ for all $m$ in $\typeabasis$, that is, $\gamma \circ \phi = 1_{\typeabasis}$.
\end{lemma}

\begin{example}\rm
\label{ex:pairing-lemma-examples-small-n}
Here is the bijection 
$\phi: \typeabasis \hookrightarrow \tvbasis$ for $n=2,3,4$. 
For $n=2$, little happens as $\phi(1)=1$.
For $n=3, 4$,
 here are $(m,\phi(m))$, segregated in boxes by $(\deg(m),\deg(\phi(m))$:
\begin{center}
$n=3$:\,\,
\begin{tabular}{|c|c|}\hline
    $m$ & $\phi(m)$ \\\hline\hline
    $1$ & $1$\\\hline
    $t_{12}$& $u_1 w_{12}$\\
    $t_{13}$& $u_1 w_{13}$\\
    $t_{23}$& $u_2 w_{23}$\\\hline
    $t_{12}t_{13}$& $w_{12} w_{13}$\\
    $t_{12}t_{23}$& $v_{12} w_{23}$\\\hline
\end{tabular}
\qquad \qquad
$n=4$:\,\,
\begin{tabular}{|c|c||c|c|}\hline
    $m$ & $\phi(m)$ &$m$ & $\phi(m)$  \\\hline\hline
    $1$ & $1$ & & \\\hline
    $t_{12}$ & $u_1 w_{12}$ & $t_{12}t_{13}$ &$w_{12}w_{13}$\\
    $t_{13}$ & $u_1 w_{13}$ &$t_{13}t_{14}$ & $w_{13}w_{14}$\\
    $t_{14}$ & $u_1 w_{14}$ &$t_{12}t_{14}$ & $w_{12}w_{14}$\\
    $t_{23}$ & $u_2 w_{23}$ &$t_{23}t_{24}$ & $w_{23}w_{24}$\\
    $t_{24}$ & $u_2 w_{24}$ &$t_{12}t_{23}$ & $v_{12}w_{23}$\\ 
    $t_{34}$ & $u_3 w_{34}$ &$t_{13}t_{14}$ & $v_{13}w_{34}$\\ 
     & &$t_{23}t_{34}$ & $v_{23}w_{34}$\\
     & &$t_{12}t_{24}$ & $v_{12}w_{24}$\\\hline
     $t_{12} t_{34}$ &  $(u_1 w_{12})(u_3 w_{34})$ & $t_{12} t_{13} t_{34}$ & $(u_1 w_{12})(w_{13}w_{34})$\\ 
     $t_{13} t_{24}$ &  $(u_1 w_{13})(u_2 w_{24})$ & $t_{12}t_{23} t_{34}$ & $(u_1 w_{12})(w_{23}w_{24})$\\ 
     $t_{23}t_{14}$ &  $(u_1 w_{14})(u_2 w_{23})$ & $t_{12}t_{13} t_{34}$ & $(u_1 w_{12})(v_{13}w_{34})$\\ 
      &  & $t_{12}t_{23}t_{34}$ & $(u_1 w_{12})(v_{23}w_{34})$\\ 
       &  & $t_{12}t_{13} t_{24}$ & $(u_1 w_{13})(v_{12}w_{24})$\\ 
        &  & $t_{12}t_{23} t_{14}$ & $(u_1 w_{14})(v_{12}w_{23})$\\ 
\hline
\end{tabular}
\end{center}
\end{example}

\begin{proof}[Proof of Pairing Lemma~\ref{lem:pairing-bijection}].
We recursively define for each monomial $m$ in $\typeabasis$ a
particular factorization $\phi(m)$ as a product of
quadratics in the set $\quadgens$
from \eqref{Z-generators}
$$
\quadgens:= \{ u_{i} w_{ij} \}_{1 \leq i < j \leq n}
\quad \sqcup \quad 
\{ w_{ij}w_{ik} \}_{1 \leq i < j < k \leq n}
\quad \sqcup \quad
\{ v_{ij}w_{jk} \}_{1 \leq i < j < k \leq n}
$$
so that the product also lies in $\vbasis$, that is, $\phi(m)$ lies in $\tvbasis$.
By construction, it will have $\gamma(\phi(m))=m$.

To define $\phi(m)$, induct on $\deg(m)$, with base case $m=1$ having $\phi(1):=1$.  In the inductive step,
define 
$$
t_{i_0,j_0}:=\max_{<}\{ t_{ij} \text{ divides }m\}
$$
where $<$ is this total order\footnote{{\bf Warning}: the total order $<$ is {\it not} a linear extension of the partial order $\prec$ from Theorem~\ref{thm:AB_VGpresentation}(ii). } on the variables
$\{t_{ij}: 1 \leq i < j \leq n\}:$
\begin{align*}
t_{12} < t_{13} < t_{14} < \cdots < t_{1,n-1} &< t_{1,n}\\ 
<t_{23} < t_{24} < \cdots < t_{2,n-1} &< t_{2,n}\\
\vdots & \\
<t_{n-2,n-1}&< t_{n-2,n}\\
&< t_{n-1,n}. 
\end{align*}
Now recursively define
the factorization $\phi(m)$ in three cases, based on
whether these sets are empty or not:
\begin{align}
\label{W-set-definition}
T_{i_0,*}&:=\{ t_{i_0,k} \text{ dividing }m/t_{i_0,j_0}\},\\
\label{V-set-definition}
T_{*,i_0}&:=\{ t_{h,i_0}  \text{ dividing }m/t_{i_0,j_0}\}.
\end{align}

\vskip.1in
\noindent
{\sf Case 1.} $T_{i_0,*} \neq \varnothing$.  In this case,
define $t_{i_0 ,k_0}:=\max_< (T_{i_0,*})$, and then define $\hat{m}$ and $\phi(m)$ by
\begin{align*}
m&= \hat{m} \cdot t_{i_0 ,k_0} t_{i_0,j_0},\\
\phi(m)&:= \phi(\hat{m}) \cdot 
w_{i_0 ,k_0} w_{i_0,j_0},
\end{align*}
where we note that $w_{i_0 ,k_0} w_{i_0,j_0}$ lies in $\quadgens$, since $k_0 < j_0$ from \eqref{W-set-definition} and the definition of $<$.

\vskip.1in
\noindent
{\sf Case 2.} $T_{i_0,*} = \varnothing$,
but $T_{*,i_0} \neq \varnothing$.  
In this case, define $t_{h_0,i_0}:=\max_< T_{*,i_0}$,
and then define $\hat{m}$ and $\phi(m)$ by
\begin{align*}
m&= \hat{m} \cdot t_{h_0,i_0} t_{i_0,j_0},\\
\phi(m)&:= \phi(\hat{m}) \cdot 
v_{h_0 ,i_0} w_{i_0,j_0},
\end{align*}
where we note that $v_{h_0 ,i_0} w_{i_0,j_0}$ lies in $\quadgens$, since $h_0 < i_0$ from \eqref{V-set-definition} and the definition of $<$.

\vskip.1in
\noindent
{\sf Case 3.} Both $T_{i_0,*} = \varnothing=T_{*,i_0}$.
In this case, define $\hat{m}$ and $\phi(m)$ by
\begin{align*}
m&= \hat{m} \cdot t_{i_0,j},\\
\phi(m)&:= \phi(\hat{m}) \cdot 
u_{i_0} w_{i_0,j}.
\end{align*}
where we note that $u_{i_0} w_{i_0,j_0}$ lies in $\quadgens$.

\begin{example} \rm
Taking $n=10$, we illustrate
the computation of $\phi(m)$ for this monomial $m$ in $\typeabasis$:
\begin{align*}
m &:=t_{12} t_{24} t_{35}t_{16} t_{5,10} t_{78} t_{89} t_{7,10} & \\
& &\\
\phi(m)
&=\phi(t_{12} t_{24} t_{35}t_{16} t_{5,10} t_{78} t_{89} t_{7,10})&
\\
&=\phi(t_{12} t_{24} t_{35}t_{16} t_{5,10} t_{7,10}) \cdot v_{78} w_{89}
&[t_{i_0,j_0}=t_{89},\,\,\text{\sf Case 2}]\\
&=\phi(t_{12} t_{24} t_{35}t_{16}) \cdot w_{5,10} w_{7,10} \cdot v_{78} w_{89}
&[t_{i_0,j_0}=t_{7,10},\,\,\text{\sf Case 1}]\\
&=\phi(t_{12} t_{24} t_{16})\cdot u_3 w_{35} \cdot w_{5,10} w_{7,10} \cdot v_{78} w_{89}
&[t_{i_0,j_0}=t_{35},\,\,\text{\sf Case 3}]\\
&=\phi(t_{16})\cdot v_{12} w_{24} \cdot u_3 w_{35}  \cdot w_{5,10} w_{7,10} \cdot v_{78} w_{89}
&[t_{i_0,j_0}=t_{24},\,\,\text{\sf Case 2}]\\
&=u_1 w_{16}\cdot v_{12} w_{24} \cdot u_3 w_{35}  \cdot w_{5,10} w_{7,10} \cdot v_{78} w_{89}
&[t_{i_0,j_0}=t_{16},\,\,\text{\sf Case 3}]\\
\end{align*}
\end{example}

In each {\sf Case 1,2,3}, induction shows $\gamma(\phi(m))=m$.  Thus $\phi$ is injective.  The construction also shows that $\phi(m)$ lies in $\prod \quadgens$.  It only remains to show $\phi(m)$ lies in $\vbasis$, and hence in $\tvbasis$, as claimed.  That is, we must show that
for each $k=1,2,\ldots,n$, the monomial $\phi(m)$
is not divisible by any pair from this set:
$$
V_k:=\{u_k, \,\, v_{1k},v_{2k}, \ldots,v_{k-1,k},
\,\, w_{1k},w_{2k}, \ldots,w_{k-1,k}\}.
$$
Since $\gamma(\phi(m))=m$, if a pair from $V_k$ of any of the forms $v_{ik} v_{jk}$ or $v_{ik} w_{jk}$
or $w_{ik} w_{jk}$ with $1 \leq i,j \leq k-1$
divides $\phi(m)$, then $t_{ik} t_{jk}$
from the set $T_k=\{t_{1k}, t_{2k}, \ldots t_{k-1,k}\}$ would divide $m$, contradicting $m \in \typeabasis$.

It thus remains to reason a contradiction
if any pairs from $V_k$ of the form $u_i w_{ij}$ or $u_i v_{ik}$ divide $\phi(m)$. 

If $u_i w_{ij}$ or $v_i w_{ij}$ divides $\phi(m)$, then at some
point in the factorization process, $u_i$ was introduced during an instance of {\sf Case 3}.  This means there was a monomial $m'$ dividing $m$, and the  $<$-maximum variable $t_{i_0,j_0}$ dividing $m'$ had the form $t_{k,j_0}$, with
$T_{k,*}=T_{*,k}=\varnothing$ for $m'$.  Note that
$t_{ik}$ divides $m$, since either $v_{ik}$ or $w_{ik}$ divides $\phi(m)$.  However, since $m'$ had  $T_{*,k}=\varnothing$, it must be that
$t_{ik}$ had been factored out of some $m''$ earlier than $m'$.
Since $t_{k,j_0} > t_{ik}$, and $t_{k,j_0}$ divides $m'$ and hence also $m''$, it cannot be that $t_{ik}$ was
factored out as the $<$-maximum variable in $m''$.
Rather $t_{ik}$ it must have been the $<$-smaller element factored out in a {\sf Case 1} or {\sf Case 2} pair, depending upon whether we assumed $u_i w_{ij}$ or $u_i v_{ij}$
divided $\phi(m)$.

If it was a {\sf Case 1} pair, say  
$t_{ik} t_{ip}$ with $p >k$, this would mean
$t_{ip}$ was the $<$-maximum variable dividing $m''$.  But this contradicts the fact that $t_{k,j_0} > t_{ik}$ and $t_{k,j_0}$ also divides $m''$.

If it was a {\sf Case 2} pair, say  
$t_{ik} t_{kp}$ with $p >k$, then $m''$ being in {\sf Case 2} 
implies $T_{k,*}=\varnothing$, contradicting
the fact that $t_{k,j_0}$ divides $m''$.
\end{proof}

Having proven Lemma~\ref{lem:pairing-bijection}, this completes the proof of
Theorem~\ref{thm:tvbasis} and 
the proof of Theorem~\ref{peak-idempotents-give-cohomology-reps}(iii).
\end{proof}

\section{Filtration, bigrading, and further decompositions}
\label{sec:filtration-and-gradings}

The goal in this section is to prove several $\kk$-vector space and $\kk \symmB{n}$-module decompositions on $\TypeBCohomology{n}=H^* \TypeBConf{n}$ that refine its cohomological grading.
Some of these arise from a natural filtration and associated graded ring with a bigrading,
derived from the two $\symmB{n}$-orbits of hyperplanes in $$
\A_{\symmB{n}}=\{x_i=0\}_{i=1,2,\ldots,n} \sqcup \{x_i=\pm x_j\}_{1 \leq i<j\leq n}.
$$
The inherited decompositions and bigrading on the $\Z_2^n$-fixed space $H^* \PeakConf{n}=(\TypeBCohomology{n})^{\Z_2^n}$ turn out to be crucial for analyzing its $\kk\symm{n}$-module structure, and for the proof of Theorem~\ref{decomposition-of-peak-idempotent-reps-theorem} in Section~\ref{sec:proofofpeakreps}.

\subsection{A filtration of $\TypeBCohomology{n}$} 
\label{sec:filtrations}

Using the $\Z_2^n$-diagonalized  $\vbasis$-presentation
$$
\TypeBCohomology{n} := H^* \TypeBConf{n} \cong 
\kk [v_{ij},w_{ij}, u_{i} ] /\I,
$$
it will be useful to consider a certain  filtration on $\TypeBCohomology{n}$. As noted in \eqref{signed-perms-on-u-variables}, a
signed permutation $\sigma$ has $\sigma(u_i)=\pm u_{\hat{\sigma}(i)}$.
Consequently, the ideal of $\TypeBCohomology{n}$
$$
\uu:=(u_1,u_2,\ldots,u_n)
$$
generated by the images of $\{u_1,\ldots,u_n\}$ is $\symmB{n}$-stable.
One then has a $\symmB{n}$-stable
{\it $\uu$-adic filtration}
$$
\TypeBCohomology{n}=\uu^0 \supseteq \uu^1 \supseteq \uu^2 \supseteq \cdots \supseteq \uu^n \supseteq \uu^{n+1}=0,
$$
and its {\it associated graded ring}
$$
\gr(\TypeBCohomology{n}):=
\TypeBCohomology{n}/\uu
\,\, \oplus \,\, \uu/\uu^2 
\,\,\oplus \,\, \uu^2/\uu^3  
\,\, \oplus \,\, \cdots
=\bigoplus_{i=0}^n \uu^i/\uu^{i+1}.
$$

\begin{prop}
\label{prop:ufiltration}
The ring $\gr(\TypeBCohomology{n})$ has 
a quadratic Gr\"obner basis presentation
$
\kk [v_{ij},w_{ij}, u_{i} ] / \I_\gr
$
for any monomial ordering $\prec$ as in Theorem~\ref{theorem:newpresentation} that respects
the partial order as in \eqref{uvw-variable-partial-order}, with generators as in Table~\ref{table:Igrrelations} below. Their
$\prec$-initial terms are underlined, and coincide with the $\prec$-initial terms for corresponding generator of $\I$ shown in the same row.  Consequently,
$\gr(\TypeBCohomology{n})$ has a standard monomial $\kk$-basis which is the image of the squarefree monomials $\vbasis$.
\end{prop}

\begin{center}
\begin{table}[!h]
\centering
\setlength{\tabcolsep}{10pt} 
\renewcommand{\arraystretch}{1.5} 
\begin{tabular}{|l|l|}  \hline
\rm
\rm Generator of $\I$ &  Corresponding generator of $\I_\gr$ \\ \hline\hline 
    $\underline{u_{i}^2}  $ &    $\underline{u_{i}^2}  $\\
  $- 2u_{i} w_{ij}+\underline{v_{ij}^2}$ &  $\underline{v_{ij}^2}$\\
  $2 u_{i} w_{ij} + \underline{w_{ij}^{2}}$ &  $\underline{w_{ij}^{2}}$\\
   $\underline{v_{ij}w_{ij}}$ &  $\underline{v_{ij}w_{ij}}$\\
    $u_{i}w_{ij} -\underline{u_{j} v_{ij}}$ &  $u_{i}w_{ij} -\underline{u_{j} v_{ij}}$\\ 
 $u_{i} v_{ij} - 2 u_{i} u_{j} - \underline{u_{j} w_{ij}}$
 & $u_{i} v_{ij}  - \underline{u_{j} w_{ij}}$\\
  $v_{ij}w_{jk} - w_{ij}w_{ik} - \underline{v_{ik}v_{jk}}$
  & $v_{ij}w_{jk} - w_{ij}w_{ik} - \underline{v_{ik}v_{jk}}$\\
  $w_{ij}w_{jk} - v_{ij}w_{ik} - \underline{w_{ik}w_{jk}}$
  &   $w_{ij}w_{jk} - v_{ij}w_{ik} - \underline{w_{ik}w_{jk}}$\\
 $v_{ij}v_{jk} - v_{ij}v_{ik} - \underline{v_{ik}w_{jk}}$ 
 &  $v_{ij}v_{jk} - v_{ij}v_{ik} - \underline{v_{ik}w_{jk}}$ \\
  $w_{ij}v_{jk} - w_{ij}v_{ik} - \underline{w_{ik}v_{jk}}$ 
  &  $w_{ij}v_{jk} - w_{ij}v_{ik} - \underline{w_{ik}v_{jk}}$ \\
  & \\
 \hline 
\end{tabular}
\caption{Corresponding generators for $\I$ and $\I_\gr$,
with indices $i,j,k$ satisfying $i<j<k$.  Initial terms with respect to any monomial order $\prec$ extending \eqref{uvw-variable-partial-order} are underlined.}
\label{table:Igrrelations}
\end{table}
\end{center}

\begin{proof}
Consider the map of $\kk$-algebras
$
f: \kk [u_i, v_{ij},w_{ij} ] \longrightarrow \gr(\TypeBCohomology{n})
$
defined by sending $\{ v_{ij},w_{ij}\}$ to their images
in $\gr(\TypeBCohomology{n})_0=\TypeBCohomology{n}/\uu$
and sending $\{ u_i\}$ to their images in $\gr(\TypeBCohomology{n})_1=\uu/\uu^2$.  This map surjects since both
$\TypeBCohomology{n}$ and $\gr(\TypeBCohomology{n})$ are generated
as $\kk$-algebras by the images of the variables. 

We claim that $f$ descends to a (surjective) map of 
the quotient ring 
\begin{equation}
\label{Igr-descended-quotient-mapping}
\kk [v_{ij},w_{ij}, u_{i}] /\I_\gr \twoheadrightarrow \gr(\TypeBCohomology{n}),
\end{equation}
where $\I_\gr$ is the ideal generated by the quadratic polynomials $q$ in the right column of Table~\ref{table:Igrrelations}.
To see that each such $q$ from the table lies in $\ker(f)$,
note that in each case, there is an index $j \in \{0,1,2\}$ such that 
every monomial of $q$ has degree exactly $j$ in the variables $\{ u_i\}$, but its corresponding generator $\hat{q}$ of $\I$
in the left column has every monomial of $\hat{q}-q$ of degree $j+1$
in the variables $\{u_i\}$.  This implies that $f(q)=f(\hat{q})$ in $\gr(\TypeBCohomology{n})_j=\uu^j/\uu^{j+1}$. But
$f(\hat{q})$ vanishes in $\uu^j$ since $\hat{q}$ lies in $\I$ and hence
maps to $0$ already in $\TypeBCohomology{n}$.

To see that the quotient mapping \eqref{Igr-descended-quotient-mapping}
is an isomorphism, since 
$$
\dim_{\kk} \gr(\TypeBCohomology{n}) = 
\dim_{\kk} \TypeBCohomology{n} = |\V| \quad (=2^n \cdot n!),
$$
by dimension-counting, it suffices to check that the monomials in
$\vbasis$ give a $\kk$-spanning set for the quotient 
$\kk [v_{ij},w_{ij}, u_{i}]/\I_\gr$.  However, this is argued
just as in the proof of Theorem~\ref{theorem:newpresentation},
as the generators for $\I$ and $\I_\gr$ have the same $\prec$-initial
terms.  Similarly, this also shows the generators of $\I_\gr$ are
a Gr\"obner basis.
\end{proof}

\begin{remark} \rm
    Since the isomorphism $\B: \kk [v_{ij},w_{ij}, u_{i}]
    \rightarrow \kk[u_{ij}^+, u_{ij}^-, u_i]$ of Definition~\ref{def:basismap} (and the induced isomorphism $\tilde{\B}$
    in Theorem~\ref{theorem:newpresentation}) map $u_i \longmapsto u_i$, we could have defined the {\it same} $\uu$-adic filtration on $\TypeBCohomology{n}$
    using its $\ubasis$-presentation rather than its $\vbasis$-presentation.
\end{remark}

\begin{definition} \rm (Bigrading on $\gr(\TypeBCohomology{n})$) 
\label{def:bigrading-on-grY}
Note that the associated graded ring $\gr(\TypeBCohomology{n})$ is {\it bigraded} or {\it $\mathbb{N}^2$-graded}, since it inherits the cohomological grading from $\TypeBCohomology{n}=H^*(\TypeBConf{n})$, as
well as having its $\uu$-adic grading.  We define
$
\gr(\TypeBCohomology{n})_{k,\ell}
$
to be the intersection of $\gr(\TypeBCohomology{n})_k=\uu^k/\uu^{k+1}$ with the $k+\ell$ cohomological graded component,  
that is, 
$
\gr(\TypeBCohomology{n})_{k,\ell}
$ is the
$\kk$-span of the images of monomials in $\kk [v_{ij},w_{ij}, u_{i}]$ 
\begin{itemize} 
\item of total degree $k$, with
\item degree $\ell$ in the variables $\{v_{ij},w_{ij}\}$.
\end{itemize}
\end{definition}

Since the $\uu$-adic filtration is $\symmB{n}$-stable, one has a $\kk\symmB{n}$-module decomposition
\begin{equation}
\label{bigrading-on-gr(Y)}
\gr(\TypeBCohomology{n})=
\bigoplus_{k=0}^n \bigoplus_{\ell=0}^k \gr(\TypeBCohomology{n})_{k,\ell}.
\end{equation}

Although $\TypeBCohomology{n} \not\cong \gr(\TypeBCohomology{n})$ as {\it $\kk$-algebras}, when $\fieldchar(\kk) > n$ the semisimplicity 
of $\kk\symmB{n}$ implies that they are isomorphic as $\kk\symmB{n}$-modules, as discussed in Section~\ref{sec:semisimplicity}. Hence one has 
$\kk\symmB{n}$-module isomorphisms
\begin{align*}
H^* \TypeBConf{n}&=\TypeBCohomology{n} \cong \gr(\TypeBCohomology{n})=
\bigoplus_{k=0}^n \bigoplus_{\ell=0}^k \gr(\TypeBCohomology{n})_{k,\ell}\\
H^{2k}(\TypeBConf{n}) &=(\TypeBCohomology{n})_k \cong \gr(\TypeBCohomology{n})_{k,*}:=\bigoplus_{\ell=0}^k \gr(\TypeBCohomology{n})_{k,\ell}.
\end{align*}
Note that Proposition~\ref{prop:ufiltration} says that the images of the monomials in $\vbasis$ form a $\kk$-basis of $\gr(\TypeBCohomology{n})$, and each of these monomials is homogeneous with respect to the bigrading.
The description of $\vbasis$ as squarefree products of at most one variable from each of $V_1,\ldots,V_n$
leads immediately to this {\it bigraded Hilbert series}
\begin{align*}
\bihilb(\gr(\TypeBCohomology{n}), t,q)
&:=\sum_{k,\ell} t^k q^\ell \dim_{\kk} \gr(\TypeBCohomology{n})_{k,\ell}\\
&= (1+t)\,\,
(1+t(1+2q))\,\,
(1+t(1+4q))\,\,
(1+t(1+6q)) \cdots (1+t(1+2(n-1)q)). 
\end{align*}

Moreover, one can take $(\Z_2^n)$-fixed spaces in
the bigrading and $\kk\symmB{n}$-module decomposition \eqref{bigrading-on-gr(Y)}.
As $(\Z_2)^n$ is normal in $\symmB{n}$, with quotient $\symm{n}=\symmB{n}/\Z_2^n$, one has an induced bigrading and $\kk\symm{n}$-module decomposition 
\begin{equation}
\label{bigrading-on-gr(Z)}
 \gr(\TypeBCohomology{n})^{\Z_2^n} = 
  \bigoplus_{k=0}^n \gr(\TypeBCohomology{n})_k^{\Z_2^n}
  = \bigoplus_{k=0}^n \bigoplus_{\ell=0}^k \gr(\TypeBCohomology{n})^{\Z_2^n}_{k,\ell}.
\end{equation}

\begin{remark} \rm
As noted in Section~\ref{sec:semisimplicity}, when $\fieldchar(\kk)>2$ so that $\kk \Z_2^n$ is semisimple, one can 
consider either 
the $\Z_2^n$-fixed subspace $\gr(\TypeBCohomology{n})^{\Z_2^n} \subset \gr(\TypeBCohomology{n})$, or
the associated grading ring for the $\uu$-adic
filtration on $(\TypeBCohomology{n})^{\Z_2^n}$:
$$
\gr( (\TypeBCohomology{n})^{\Z_2^n}):=
\bigoplus_{k=0}^n \,\,
\uu^k \cap (\TypeBCohomology{n})^{\Z_2^n} \, / \, \uu^{k+1} \cap (\TypeBCohomology{n})^{\Z_2^n}.
$$
One has a $\kk\symm{n}$-module isomorphism
$
\gr(\TypeBCohomology{n})^{\Z_2^n} 
\cong
\gr( (\TypeBCohomology{n})^{\Z_2^n})
$
between the two.

Furthermore, as long as $\fieldchar(\kk)>n$, one has these accompanying $\kk \symm{n}$-module isomorphisms

\begin{align*}
H^* \PeakConf{n} =(H^* \TypeBConf{n})^{\Z_2^n}
&=(\TypeBCohomology{n})^{\Z_2^n} 
\cong \gr(\TypeBCohomology{n})^{\Z_2^n}
=\bigoplus_{k=0}^n \bigoplus_{\ell=0}^k \gr(\TypeBCohomology{n})^{\Z_2^n}_{k,\ell}\\
H^{2k}(\TypeBConf{n})^{\Z_2^n} 
&=(\TypeBCohomology{n})^{\Z_2^n}_k 
\cong \gr(\TypeBCohomology{n})^{\Z_2^n}_{k,*}
:=\bigoplus_{\ell=0}^k \gr(\TypeBCohomology{n})^{\Z_2^n}_{k,\ell}.
\end{align*}

\end{remark}

Just as it did for $\TypeBCohomology{n}$, the set $\vbasis$ again provides a (bihomogeneous) $\kk$-basis for $\gr(\TypeBCohomology{n})$.

\begin{prop}
\label{prop:bihomogeneous-gr(Z)-basis}
The images of $\tvbasis$ in $\gr(\TypeBCohomology{n})^{\Z_2^n}$ 
form a $\kk$-basis. They are
bihomogeneous, with a monomial
$q$ having bidegree $(k,\ell)=(\deg(q), \deg(\gamma(q)))$, where $\gamma$ from \eqref{def:gamma-surjection} maps $u_i \mapsto 1$ and
$v_{ij},w_{ij} \mapsto u_{ij}$.
\end{prop}
\begin{proof}
The proof is similar to that of Theorem~\ref{thm:tvbasis}.  The images of 
$\tvbasis$ in $\gr(\TypeBCohomology{n})^{\Z_2^n}$ 
are $\kk$-linearly independent, as a subset of the $\kk$-basis
$\vbasis$ from Proposition~\ref{prop:ufiltration}.
They also have the correct cardinality 
$$
n!=\dim_{\kk} (\TypeBCohomology{n})^{\Z_2^n}=\dim_{\kk} \gr(\TypeBCohomology{n})^{\Z_2^n}
$$
from the Pairing Lemma~\ref{lem:pairing-bijection}.  Thus by dimension-counting, they are
a basis.  Their bihomogeneity follows from the definition of the
$(k,\ell)$-bidegree, where $k$ is total degree,
and $\ell$ is the degree in variables $\{v_{ij},w_{ij}\}$.
\end{proof}

As a consequence of Proposition~\ref{prop:bihomogeneous-gr(Z)-basis}, one can reinterpret the tables in
Example~\ref{ex:pairing-lemma-examples-small-n}
as grouping elements $q$ in $\tvbasis$ according
to their bidegree within $\gr(\TypeBCohomology{n})^{\Z_2^n}$:  one has $q=\phi(m)$ if and only if $\gamma(m)=q$, hence
$$
(\deg(m),\deg(\phi(m))=(\deg(\gamma(m)),\deg(q))=(\ell,k).
$$
These let one tabulate the bigraded Hilbert series for $\gr(\TypeBCohomology{n})^{\Z_2^n}$
defined by
$$
\bihilb(\gr(\TypeBCohomology{n})^{\Z_2^n}, t,q)
=\sum_{k,\ell} t^k q^\ell \dim_{\kk} \gr(\TypeBCohomology{n})^{\Z_2^n}_{k,\ell}
$$
when $n=2,3,4$ as follows:
\begin{equation}
\label{first-few-bigraded-hilbs}
\begin{tabular}{|c|l|}\hline
$n$ & $\bihilb(\gr(\TypeBCohomology{n})^{\Z_2^n}, t,q)$\\ \hline\hline
$2$ & $1$\\ \hline
$3$ & $1+t^2(3q+2q^2)$\\ \hline
$4$ & $1+t^2(6q+8q^2)+t^4(3q^2+6q^3)$ \\ \hline
\end{tabular}
\end{equation}
These Hilbert series will be discussed further and refined in Section~\ref{sec:recursions}.

\subsection{A grading on $\gr(\TypeBCohomology{n})$ by double partitions}\label{sec:signedpartitionsdecomposition}

Our goal in this section will be to decompose $\gr(\TypeBCohomology{n})$ as a $\kk$-vector space indexed by {\it double partitions} $\slambda$ of  $n$ (written $\slambda \vdash n)$,
which are simply ordered pairs of number partitions having total weight $|\lambda^+|+|\lambda^-|=n$.

\begin{definition} 
\label{def:multigraphs-and-double-partitions}
\rm
  
To each monomial $q$ in $\kk [u_i,v_{ij},w_{ij}]$ or in 
$\kk [u_i,u^+_{ij},u^-_{ij}]$, expressed uniquely 
as
\begin{align*}
q&=\prod_{i=1}^n u_i^{a_i} 
\prod_{1 \leq i < j \leq n} v_{ij}^{b_{ij}} w_{ij}^{c_{ij}}
\quad \text{ in }\kk[u_i,v_{ij},w_{ij}] \\
\text{ or }q&=\prod_{i=1}^n u_i^{a_i} 
\prod_{1 \leq i < j \leq n} (u^+_{ij})^{b_{ij}} (u^-_{ij})^{c_{ij}} 
\quad \text{ in }
\kk [u_i,u^+_{ij},u^-_{ij}],\
\end{align*}
we will associate two pieces of data:
\begin{itemize}
    \item A {\it multigraph} (graph with parallel edges, self-loops) denoted $G(q)$ on the vertex set 
$[n]=\{1,2,\ldots,n\}$, having
$b_{ij}+c_{ij}$ parallel copies of
the edge $\{i,j\}$, and $a_i$ copies of
the self-loop on vertex $i$.
    \item A double partition $\slambda$, where $\lambda^+$ (resp. $\lambda^-$) lists in weakly decreasing order the sizes of connected components of $G(q)$
with {\it no} loop vertex (resp.~{\it at least one} loop vertex).
\end{itemize}
\end{definition}

\begin{example} 
\label{ex:multigraph}
\rm
For $n=9$, either of the monomials 
\begin{align*}
q&= v_{12}^2 \cdot w_{12} \cdot w_{13}^2 \cdot v_{23}^2 \cdot w_{24} \cdot u_{5} \cdot w_{56}^3 \cdot u_7^{2} \qquad \in \kk[u_i,v_{ij},w_{ij}]\\
q'&= (u^+_{12})^2 \cdot u^-_{12} \cdot (u^-_{13})^2 \cdot (u^+_{23})^2 \cdot u^-_{24} \cdot u_{5} \cdot (u^-_{56})^3 \cdot u_7^{2} \qquad \in \kk[u_i,u^+_{ij},u^-_{ij}]
\end{align*}
has the same multigraph $G(q)=G(q')$ shown below, and
double partition $\slambda=((4,1,1),(2,1))$:
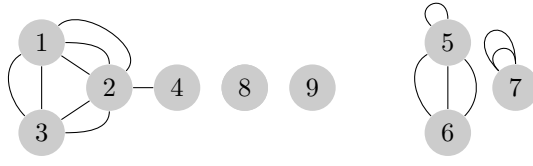
\begin{center}
\begin{minipage}{.2\textwidth}
 \hspace*{-0.6\linewidth}
\begin{tikzpicture}
  [scale=.30,auto=left,every node/.style={circle,fill=black!20}]
 \node (v1) at (-10,14) {$1$};
  \node (v3) at (-10,10) {$3$};
  \node (v2) at (-7,12) {$2$};
  \node (v4) at (-4,12) {$4$};
  \node (v8) at (-1,12) {$8$};
  \node (v9) at (2,12) {$9$};
  \node (v5) at (8,14) {$5$};
  \node (v6) at (8,10) {$6$};
    \node (v7) at (11,12) {$7$};
  \node (v8) at (-1,12) {$8$};
 \draw (v1) to [out=210,in=135,looseness=1] (v3);
  \draw (v1) to [out=0,in=90,looseness=1] (v2);
   \draw (v1) to [out=45,in=45,looseness=1] (v2);
     \draw (v2) to [out=-90,in=0,looseness=1] (v3);
    \draw (v5) to [out=210,in=135,looseness=1] (v6);
    \draw (v5) to [out=-45,in=45,looseness=1] (v6);
   \draw (v5) to [out=135,in=90,looseness=4] (v5);
   \draw (v7) to [out=135,in=90,looseness=4] (v7);
      \draw (v7) to [out=135,in=90,looseness=8] (v7);
  \foreach \from/\to in {v1/v2,v1/v3,v2/v3,v2/v4,v5/v6}
    \draw (\from) -- (\to);
\end{tikzpicture}
\end{minipage}
\end{center}
\end{example}

Define a $\kk$-vector space decomposition 
\begin{equation}
\label{double-partition-decomposition-of-V-poluynomials}
\kk[u_i,v_{ij},w_{ij}]=
\bigoplus_{\slambda \vdash n}
\kk[u_i,v_{ij},w_{ij}]_{\slambda}
\end{equation}
where 
$\kk[u_i,v_{ij},w_{ij}]_{\slambda}$ is the $\kk$-span of monomials $q$ having $\slambda$ as their associated double partition.

\begin{prop}
\label{double-partition-decomposition-of-grY}
The decomposition \eqref{double-partition-decomposition-of-V-poluynomials}
of $\kk$-vector spaces
\begin{itemize}
\item[(i)] is also a $\kk\symmB{n}$-module decomposition, and 
\item[(ii)] 
descends to a $\kk \symmB{n}$-module
decomposition
$$
\gr(\TypeBCohomology{n})=
\bigoplus_{\slambda \vdash n}
\gr(\TypeBCohomology{n})_{\slambda}.
$$
\end{itemize}
\end{prop}

\begin{proof}
For (i), recall from Remark~\ref{rmk:signed-perms-on-uvw} that for a signed permutation $\sigma$ in $\symmB{n}$, if $\hat{\sigma}:=\phi(\sigma)$ is its underlying unsigned permutation in $\symm{n}$,
then $\sigma(u_i) = \pm u_{\hat{\sigma}(i)}$,
and variables $v_{ij}, w_{ij}$ 
will have $\sigma(v_{ij}), \sigma(w_{ij})$
taking one of the forms
$\pm v_{\hat{\sigma}(i)\hat{\sigma}(j)}$
or $\pm w_{\hat{\sigma}(i)\hat{\sigma}(j)}$.
Hence for any monomial $q$, one has that
$\sigma(q)=\pm q'$ for some monomial $q'$ where the graphs $G(q), G(q')$ are isomorphic
via the permutation $\hat{\sigma}$. Thus $q,q'$ have the same 
 double partition $\slambda$,
showing $\kk[u_i,v_{ij},w_{ij}]_{\slambda}$
forms a $\kk\symmB{n}$-submodule.

For (ii), since $\gr(\TypeBCohomology{n}) \cong \kk[u_i,v_{ij},w_{ij}]/\I_{\gr}$, it 
suffices to check that defining 
\begin{equation}
\label{Igr-deocomposes-by-double-partitions}
(\I_\gr)_{\slambda}:=
\I_\gr \cap \kk[u_i,v_{ij},w_{ij}]_{\slambda}
\end{equation}
gives a $\kk$-vector space (and hence also a $\kk\symmB{n}$-module) decomposition
$$
\I_\gr=
\bigoplus_{\slambda \vdash n}
(\I_\gr)_{\slambda}. 
$$
This holds because each generator $g$ of $\I_\gr$ in the right column of Table~\ref{table:Igrrelations} has
the following property: given any two monomials
$q, q'$ appearing within $g$, their two graphs $G(q), G(q')$, although not isomorphic, have the same connected components, and the same subset of
loopless versus looped components,
and hence the same double partition $\slambda$.
 For example, if 
 $$
 g=u_i w_{ij} - v_j v_{ij}=q-q'
 $$
 one sees that both $G(q), G(q')$ have the unique edge $\{i,j\}$, and only differ in the location of the loop vertex $i$ or $j$ within that edge component.  Similarly, if 
 $$
 g=v_{ij}w_{jk} - w_{ij}w_{ik} - v_{ik}v_{jk}=q-q'-q'',
 $$
 then each of $G(q),G(q'),G(q'')$ contains two out of three edges in the triangle $\{i,j,k\}$.
This same property remains true when multiplying any of those generators $g$ by a monomial $q$ giving $q\cdot g$. Since $\I_\gr$ is the $\kk$-span of all such products $q \cdot g$, the decomposition
\eqref{Igr-deocomposes-by-double-partitions} follows.
\end{proof}

Note that Theorems~\ref{thm:AB_VGpresentation}, \ref{theorem:newpresentation}
provide explicit (finite, squarefree) monomial 
$\kk$-bases $\ubasis, \vbasis$ for $\TypeBCohomology{n}, \gr(\TypeBCohomology{n})$. However, neither of these
bases is stable under the action of $\symmB{n}$.  For some arguments, 
it is helpful to 
use certain redundant $\kk$-spanning sets
that are stable under the $\symmB{n}$-action, 
described next, in terms of the graphs $G(q)$.

\begin{definition} \rm
Say that a multigraph $G$ on vertex set $[n]$
is a {\it lightly-looped forest} if its non-loop multigraph edges form a forest, and each tree in the forest contains at most one loop. Here is an example with $n=9$:

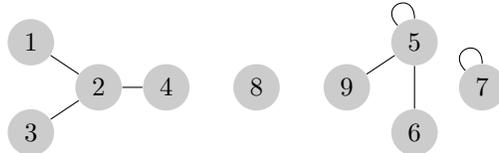
\begin{center}
\begin{minipage}{.2\textwidth}
 \hspace*{-0.6\linewidth}
\begin{tikzpicture}
  [scale=.30,auto=left,every node/.style={circle,fill=black!20}]
 \node (v1) at (-10,14) {$1$};
  \node (v3) at (-10,10) {$3$};
  \node (v2) at (-7,12) {$2$};
  \node (v4) at (-4,12) {$4$};
  \node (v9) at (4,12) {$9$};
  \node (v5) at (7,14) {$5$};
  \node (v6) at (7,10) {$6$};
    \node (v7) at (10,12) {$7$};
  \node (v8) at (0,12) {$8$};
   \draw (v5) to [out=135,in=90,looseness=4] (v5);
   \draw (v7) to [out=135,in=90,looseness=4] (v7);
  \foreach \from/\to in {v1/v2,v2/v3,v2/v4,v5/v6,v5/v9}
    \draw (\from) -- (\to);
\end{tikzpicture}
\end{minipage}
\end{center}
\end{definition}

\begin{prop}
\label{prop:spanned-by-lightly-looped-forests}
The images $\bar{q}$ within $\TypeBCohomology{n}$ of the following subsets of monomials $q$ in $\kk [u_i,u^+_{ij},v^+_{ij}]$  give
$\symmB{n}$-stable $\kk$-spanning sets for these subspaces:
\begin{itemize}
\item[(i)] $\TypeBCohomology{n}=\spn_{\kk}
    \{\bar{q}: G(q)\text{ is a lightly-looped forest }\}.$
\item[(ii)] $\TypeBCohomology{n} \cap \uu^d =\spn_{\kk}
    \{\bar{q}: G(q)\text{ is a lightly-looped forest with }\slambda\text{ having }\ell(\lambda^-)\geq d\}.$
\item[(iii)] $(\TypeBCohomology{n})_\mu=\spn_{\kk}
    \{\bar{q}: G(q)\text{ is a lightly-looped forest  with }\slambda\text{ having }\lambda^+=\mu\}$ for all $|\mu| \leq n$.
\item[(iv)] $(\TypeBCohomology{n})_\mu \cap \uu^d =\spn_{\kk}
    \{\bar{q}: G(q)\text{ is a lightly-looped forest with }\slambda\text{ having }\lambda^+=\mu\text{ and }\ell(\lambda^-)\geq d\}.$
\end{itemize}

\noindent
Likewise, the images $\bar{q}$ within $\gr(\TypeBCohomology{n})$ of the following subsets of monomials $q$ in $\kk [u_i,v_{ij},w_{ij}]$  give
$\symmB{n}$-stable $\kk$-spanning sets for these subspaces:
\begin{itemize}
\item[(v)] $\gr(\TypeBCohomology{n})=\spn_{\kk}
    \{\bar{q}: G(q)\text{ is a lightly-looped forest }\}.$
\item[(vi)] $\gr(\TypeBCohomology{n})_{\slambda}=\spn_{\kk}
    \{\bar{q}: G(q)\text{ is a lightly-looped forest with }G(q)=\slambda\}.$
\end{itemize} 
\end{prop}

\begin{proof}
Since the $\symmB{n}$-action on monomials $q$ leaves $G(q)$ unchanged up to graph-isomorphism, also preserving the 
double partition $\slambda$, all 
of the above sets are $\symmB{n}$-stable.

For assertions (i),(v) about the $\kk$-spanning
sets of $\TypeBCohomology{n}$ and $\gr(\TypeBCohomology{n})$, it suffices to check that
they contain the $\kk$-basis $\ubasis$
for $\TypeBCohomology{n}$ and 
the $\kk$-basis $\vbasis$ for $\gr(\TypeBCohomology{n})$.  In
other words, one must check that monomials $q$ in $\ubasis$ or $\vbasis$ always have $G(q)$ a lightly-looped forest.  This is easily proven by induction on $\deg(q)$ from the description of $q$ as a squarefree monomial which is a product of at most one variable from each set $U_1,\ldots,U_n$ in
Theorem~\ref{thm:AB_VGpresentation}(ii), or from each set $V_1,\ldots,V_n$ in \eqref{uvw-variable-partial-order}.

For assertion (ii), note that inside $\kk[u_i,u^+_{ij},u^-_{ij}]$, the ideal $\uu^d$ is the $\kk$-span of all monomials
$q$ having degree at least $d$ in the variables $\{u_i\}$.  The division algorithm in $\kk[u_i,u^+_{ij},u^-_{ij}]$
using the monomial order $\prec$ from
\eqref{uvw-variable-partial-order} lets
one iteratively rewrite
$q$ modulo $\I_{\symmB{n}}$ as a sum of standard monomials $q'$ in $\ubasis$.  At each iteration, the division algorithm makes a replacement $f \rightarrow f'$ using a generator $g$ from the right column of Table~\ref{table:AB_VG} to substitute
the underlined leading term $\init_\prec(g)$
for the sum $g-\init_\prec(g)$ of its non-leading terms.  Inspection of Table~\ref{table:Irelations} shows that
the monomials within $g-\init_\prec(g)$
always have $\{u_i\}$-degree bounded below by the $\{u_i\}$-degree of $\init_\prec(g)$.  Hence by induction on the number of division algorithm steps, since $q$ has $\{u_i\}$-degree at least $d$, it will be rewritten as a sum of monomials $q'$ in $\U$ all with $\{u_i\}$-degree at least $d$.  On the other hand, when $G(q')$ is a
lightly-looped forest with double partition $\slambda$, the $\{u_i\}$-degree of $q'$ is exactly $\ell(\lambda^-)$ since
there is one $u_i$ variable for each (lightly) rooted tree.  Assertion (ii) follows.

For assertions (iii), (iv), consider a monomial $q$ in
$\kk [u_i,u^+_{ij},u^-_{ij}]$ having $G(q)$ a lightly-looped forest and associated double partition 
$\slambda$.  It is not hard to see that the type $B$ 
intersection subspace $X=\cap_j H_j$ in $\lat(\A_{\symmB{n}})$
corresponding to $q$, as defined
in Section~\ref{subsec:typebarrangement},
will be the signed partition $\pi_X$ whose nonzero blocks
correspond to the unlooped trees of $G(q)$, and whose
zero block is the union of all vertices in the looped trees
of $G(q)$.  Consequently, the $\symmB{n}$-orbit $[X]$ in $\lat(\A_{\symmB{n}})/\symmB{n}$ will be indexed by $\mu=\lambda^+$.
Bearing this in mind, assertion (iii) follows from assertion (i), and assertion (iv) follows from assertion (ii).

Assertion (vi) follows from assertion (i) and the definition of
$\gr(\TypeBCohomology{n})_{\slambda}$
implicit in Proposition~\ref{double-partition-decomposition-of-grY}. 
\end{proof}

One can now see how double partitions refine the bigrading on $\gr(\TypeBCohomology{n})$ from
Definition~\ref{def:bigrading-on-grY}.

\begin{prop}
\label{prop:double-partition-to-bigrading}
The decomposition in Proposition~\ref{double-partition-decomposition-of-grY} refines
the bigrading on
$\gr(\TypeBCohomology{n})$ in the sense that
$$
\gr(\TypeBCohomology{n})_{k,\ell}
=\bigoplus_{\substack{\slambda\vdash n\\
\ell(\lambda^+)=n-k\\
\ell(\lambda^+)+\ell(\lambda^-)=n-\ell}}
\gr(\TypeBCohomology{n})_{\slambda}.
$$
\end{prop}
\begin{proof}
By Proposition~\ref{prop:spanned-by-lightly-looped-forests}, it suffices to check this:
a monomial $q$ in $\kk[u_i,v_{ij},w_{ij}]$
with $G(q)$ a lightly-looped forest,
of total degree $k$ and degree $\ell$ in the
$\{v_{ij},w_{ij}\}$, has its
associated double partition $\slambda$ satisfying
\begin{align}
\label{desired-bigrading-relation-1}
\ell(\lambda^+)&=n-k \\
\label{desired-bigrading-relation-2}
\ell(\lambda^+)+\ell(\lambda^-)&=n-\ell.
\end{align}
These two relations \eqref{desired-bigrading-relation-1}, \eqref{desired-bigrading-relation-2} follow from the fact that any tree on $m$ vertices has $m-1$ edges: 
\begin{align*}
\ell= \#\{\text{nonloop edges in }G(q)\}
&=\sum_{i=1}^{\ell(\lambda^+)} (\lambda^+_i-1)
+ \sum_{i=1}^{\ell(\lambda^-)} (\lambda^-_i-1)\\
&= (|\lambda^+| - \ell(\lambda^+)) +(|\lambda^-| - \ell(\lambda^-))
=n- (\ell(\lambda^+)+\ell(\lambda^-))\\
& \\
k= \#\{\text{edges in }G(q)\}
&=\sum_{i=1}^{\ell(\lambda^+)} (\lambda^+_i-1)
+ \sum_{i=1}^{\ell(\lambda^-)} \lambda^-_i\\
&= (|\lambda^+| - \ell(\lambda^+)) +|\lambda^-| 
=n- \ell(\lambda^+).\qedhere
\end{align*}
\end{proof}

We next examine how double partitions 
interact with the flat-orbit decomposition 
\eqref{eq:typeABflatorbitdecomposition}
of $\TypeBCohomology{n}$. It helps to be able to pass between the  $\ubasis$- and $\vbasis$-presentations
for $\TypeBCohomology{n}$. This passage is governed by the 
isomorphism $\basismap: \kk[v_{ij}, w_{ij}, u_i] \longrightarrow \kk[u_{ij}^+, u_{ij}^-, u_i]$
and its inverse $\basismap^{-1}$, from Definition~\ref{def:basismap}:
$$
\begin{array}{lcccl}
\basismap(u_i):=u_i& & & &\basismap^{-1}(u_i)=u_i\\
    \basismap(v_{ij}):= u^+_{ij} + u_{ij}^{-}& & & &
   \basismap^{-1}(u_{ij}^+):= 
    (v_{ij} + w_{ij})/2\\ 
    \basismap(w_{ij}):= u_{ij}^+ - u_{ij}^-& & & &\basismap^{-1}(u_{ij}^-):= (v_{ij} - w_{ij})/2
\end{array}
$$
\noindent
These $\basismap, \basismap^{-1}$ 
do not send monomials to monomials,
but {\it do} respect the multigraphs $G(q)$ 
in a certain sense.

\begin{prop}
    \label{prop:isos-preserve-multigraphs}
    Both $\basismap, \basismap^{-1}$ send a monomial $q$ to a $\kk$-linear sum of monomials $q'$
    with $G(q')=G(q)$.
\end{prop}
\begin{proof}
We give the argument for $\basismap$; the argument for $\basismap^{-1}$ is similar.  A typical monomial
in $\kk[u_i,v_{ij},w_{ij}]$
\begin{align*}
q&=\prod_{i=1}^n u_i^{a_i} 
\prod_{1 \leq i < j \leq n} v_{ij}^{b_{ij}} w_{ij}^{c_{ij}}\\
\text{ has }\basismap(q)&=\prod_{i=1}^n u_i^{a_i} 
\prod_{1 \leq i < j \leq n} (u^+_{ij}+u^-_{ij})^{b_{ij}} (u^+_{ij}-u^-_{ij})^{c_{ij}}\\
&=\prod_{i=1}^n u_i^{a_i} 
\prod_{1 \leq i < j \leq n} 
\sum_{k,\ell} \binom{b_{ij}}{k}  \binom{c_{ij}}{\ell}
(-1)^{c_{ij}-\ell} \cdot
(u^+_{ij})^k (u^-_{ij})^{b_{ij}-k} 
(u^+_{ij})^\ell (u^-_{ij})^{c_{ij}-\ell}.
\end{align*}
But every monomial $q'$ appearing in this expansion 
has $G(q')=G(q)$, since it has the form:
$$
q'=\prod_{i=1}^n u_i^{a_i} 
\prod_{1 \leq i < j \leq n} (u^+_{ij})^k (u^-_{ij})^{b_{ij}-k} 
(u^+_{ij})^\ell (u^-_{ij})^{c_{ij}-\ell}. \qedhere
$$
\end{proof}

\begin{prop}
\label{prop:compatiblefiltrations}
When $\fieldchar(\kk)>n$, one has a $\kk \symmB{n}$-module isomorphism
\begin{equation}
\label{mu-filtration-isorphism}
(\TypeBCohomology{n})_\mu 
\cong 
\bigoplus_{\substack{\slambda \vdash n:\\ \lambda^{+} = \mu}} 
\gr(\TypeBCohomology{n})_{\slambda}.
\end{equation}
\end{prop}
\begin{proof}
Consider two subsets of monomials $U_\mu \subset \kk[u_i,u^+_{ij},u^-_{ij}]$ and $V_\mu \subset \kk[u_i,v_{ij},w_{ij}]$, both defined by the
condition that they are the monomials $q$ for which 
$G(q)$ is a lightly-rooted forest and the associated double partition $\slambda$ has $\lambda^+=\mu$.
By Proposition~\ref{prop:isos-preserve-multigraphs},
the isomorphism $\basismap: \kk[u_i,v_{ij},w_{ij}]  
\overset{\sim}{\rightarrow}  
\kk[u_i,u^+_{ij},u^-_{ij}]$ maps
the $\kk$-span of $V_\mu$ isomorphically to the $\kk$-span
of $U_\mu$.  

We now use $U_\mu, V_\mu$ to re-interpret the two sides of \eqref{mu-filtration-isorphism}. Proposition~\ref{prop:spanned-by-lightly-looped-forests} (vi) implies
the right side of 
\eqref{mu-filtration-isorphism} is
the image of $\spn_{\kk} V_\mu$ under  $\kk[u_i,v_{ij},w_{ij}] \twoheadrightarrow \gr(\TypeBCohomology{n})$.

For the left side of \eqref{mu-filtration-isorphism},  consider the $\uu$-adic filtration within $\TypeBCohomology{n}$ of the subspace $(\TypeBCohomology{n})_\mu$:
$$
(\TypeBCohomology{n})_\mu
\,\, \supset \,\, \uu \cap (\TypeBCohomology{n})_\mu
\,\, \supset \,\, \uu^2 \cap (\TypeBCohomology{n})_\mu
\,\, \supset \,\, \cdots 
\,\, \supset \,\, \uu^n \cap (\TypeBCohomology{n})_\mu
\,\, \supset  0.
$$
By semisimplicity, this gives rise to a $\kk\symmB{n}$-module isomorphism
with the direct sum of its filtration factors:
\begin{equation}
\label{easy-filtration-iso}
(\TypeBCohomology{n})_\mu
\cong 
\bigoplus_{d=0}^n \quad
\uu^d \cap (\TypeBCohomology{n})_\mu \,\, /\,\,
\uu^{d+1} \cap (\TypeBCohomology{n})_\mu.
\end{equation}

Proposition~\ref{prop:spanned-by-lightly-looped-forests} assertions (iii) and (iv)
imply that the set $U_\mu$ contains
for each $d$ a subset whose images in $\TypeBCohomology{n}$ span $(\TypeBCohomology{n})_\mu \cap \uu^d$.
Hence the right side of \eqref{easy-filtration-iso} is the image of the $\kk$-span of $U_\mu$ under the surjection
$\kk[u_i,u^+_{ij},u^-_{ij}] \twoheadrightarrow \gr(\TypeBCohomology{n})$, considering
$\gr(\TypeBCohomology{n})$ as a quotient of 
$\kk[u_i,u^+_{ij},u^-_{ij}]$.  However, as noted above,
the map $\basismap^{-1}$ sends $\spn_{\kk} U_\mu$ isomorphically onto $\spn_{\kk}V_\mu$ in $\kk[u_i,v_{ij},w_{ij}]$.  Hence the right side of \eqref{easy-filtration-iso} is also the image of the $\kk$-span of $V_\mu$
under $\kk[u_i,v_{ij},w_{ij}] \twoheadrightarrow \gr(\TypeBCohomology{n})$, when considering $\gr(\TypeBCohomology{n})$ as a quotient of 
$\kk[u_i,v_{ij},w_{ij}]$.  That is, it is 
exactly the right side of \eqref{mu-filtration-isorphism}. 

In summary, one has these isomorphisms and
equalities proving \eqref{mu-filtration-isorphism}:
\begin{align*}
(\TypeBCohomology{n})_\mu 
&\cong
\bigoplus_{d=0}^n \quad
\uu^d \cap (\TypeBCohomology{n})_\mu \,\, /\,\,
\uu^{d+1} \cap (\TypeBCohomology{n})_\mu 
& \text{ by eqn. }~\eqref{easy-filtration-iso}
\\
&=\spn_{\kk}\{\bar{q} \in \gr(\TypeBCohomology{n}): q \in U_\mu\}
&\text{ by Proposition~\ref{prop:spanned-by-lightly-looped-forests}(iii),(iv)}
\\
&= \spn_{\kk}\{\bar{q} \in \gr(\TypeBCohomology{n}): q \in V_\mu\}
&\text{ by the isomorphisms }\basismap, \basismap^{-1}
\\
&=\bigoplus_{\substack{\slambda \vdash n:\\ \lambda^{+} = \mu}} 
\gr(\TypeBCohomology{n})_{\slambda}
&\text{ by Proposition~\ref{prop:spanned-by-lightly-looped-forests}(vi)}.\quad \qedhere
\end{align*}
\end{proof}

\subsection{A vector space presentation of $\gr(\TypeBCohomology{n})$}
\label{sec:VSdescription}
Beyond its bigrading, there are other 
advantages to using $\gr(\TypeBCohomology{n})$
versus $\TypeBCohomology{n}$.  When considering $\symm{n}$-actions in Section~\ref{sec:proofofpeakreps}, it turns out to be useful to develop a different presentation of 
$\gr(\TypeBCohomology{n})$ as a quotient $\kk$-vector space, rather than as a ring.

This presentation reflects a different
$\kk$-vector space decomposition
of $\kk[u_i,v_{ij},w_{ij}]$,
according to the following data associated
to the multigraphs $G(q)$ for its monomials $q$.

\begin{definition} \rm
    A multigraph $G$ on vertex set $[n]$
    has {\it cyclomatic-loop (CL-)data}
    the triple $(\pi,\beta,\lambda)$ where
    \begin{itemize}
        \item  $\pi=\{B_1,\ldots,B_k\}$ is the vertex set partition $[n]=B_1 \sqcup \cdots \sqcup B_k$ of $G$ into connected components $B_i$,
        \item $\beta: \pi \rightarrow \{0,1,2,\ldots\}$ 
        sends $B_i$ to
        the {\it cyclomatic number}
        $
        \#\{\text{edges of }G|_{B_i}\} - \#\{\text{vertices of }G|_{B_i}\} + 1,
        $
        \item $\lambda: \pi \rightarrow \{0,1,2,\ldots\}$ sends $B_i$ to
        its {\it number of self-loops} in $G|_{B_i}$.
    \end{itemize}
\end{definition}

\begin{example} \rm
The multigraph $G$ from Example~\ref{ex:multigraph}
\begin{center}
\begin{minipage}{.2\textwidth}
 \hspace*{-0.6\linewidth}
\begin{tikzpicture}
  [scale=.30,auto=left,every node/.style={circle,fill=black!20}]
 \node (v1) at (-10,14) {$1$};
  \node (v3) at (-10,10) {$3$};
  \node (v2) at (-7,12) {$2$};
  \node (v4) at (-4,12) {$4$};
  \node (v8) at (-1,12) {$8$};
  \node (v9) at (2,12) {$9$};
  \node (v5) at (8,14) {$5$};
  \node (v6) at (8,10) {$6$};
    \node (v7) at (11,12) {$7$};
  \node (v8) at (-1,12) {$8$};
 \draw (v1) to [out=210,in=135,looseness=1] (v3);
  \draw (v1) to [out=0,in=90,looseness=1] (v2);
   \draw (v1) to [out=45,in=45,looseness=1] (v2);
     \draw (v2) to [out=-90,in=0,looseness=1] (v3);
    \draw (v5) to [out=210,in=135,looseness=1] (v6);
    \draw (v5) to [out=-45,in=45,looseness=1] (v6);
   \draw (v5) to [out=135,in=90,looseness=4] (v5);
   \draw (v7) to [out=135,in=90,looseness=4] (v7);
      \draw (v7) to [out=135,in=90,looseness=8] (v7);
  \foreach \from/\to in {v1/v2,v1/v3,v2/v3,v2/v4,v5/v6}
    \draw (\from) -- (\to);
\end{tikzpicture}
\end{minipage}
\end{center}
has the following CL-data $(\pi,\beta,\lambda)$:
\begin{align*}
\pi&=\{
\underbrace{\{1,2,3,4\}}_{B_1}, 
\underbrace{\{8\}}_{B_2}, 
\underbrace{\{9\}}_{B_3}, 
\underbrace{\{5,6,7\}}_{B_4}
\underbrace{\{7\}}_{B_5}
\}\\
\beta(B_1)&=5,\\
\beta(B_2)&=\beta(B_3)=0,\\
\beta(B_4)&=3,\\
\beta(B_5)&=2,\\
&\\
\lambda(B_1)&=\lambda(B_2)=\beta(B_3)=0,\\
\lambda(B_4)&=1,\\
\lambda(B_5)&=2.
\end{align*}
\end{example}

Here are a few simple properties of CL-data.

\begin{prop}
\label{prop:CL-data-properties}
Let $G$ be a multigraph with CL-data $(\pi,\beta,\lambda)$.
\begin{itemize}
\item[(i)] One has $\beta(B_i) \geq \lambda(B_i)$ for all $i$.
\item[(ii)] Lightly-looped forests $G$
are characterized by this: 
for all $ i, \,\, \beta(B_i) \leq 1$
and $\beta(B_i)=1$ implies $ \lambda(B_i)=1$.
\item[(iii)] When $G^+$ is obtained from $G$ by adding one more copy of the edge $\{i,j\}$, 
the CL-data $(\pi^+,\beta^+,\lambda^+)$ for $G^+$
can be computed from  $(\pi,\beta,\lambda)$,
without knowledge of $G$ itself.  
\item[(iv)] The same holds if $G^+$ is obtained from $G$ by adding one more copy of a loop on vertex $i$.
\end{itemize}
\end{prop}
\begin{proof}
Assertions (i),(ii) are straightforward.

For (iii), assume $G^+$ adds a copy of the edge $\{i,j\}$ to $G$.  Consider two cases:  either the
two vertices $i,j$ lie in the same component $G|_{B_k}$,
or they lie in different components, say $i \in B_k, j \in B_\ell$.  In the former case, $\pi^+=\pi$ and $\lambda^+=\lambda$, and $\beta^+(B_\ell)=\beta(B_\ell)$ for $\ell \neq k$, but $\beta^+(B_k)=\beta(B_k)+1$. 
In the latter case, $\pi^+$ contains a merged block $B_{k\ell}:=B_k \sqcup B_\ell$ for which 
\begin{align*}
 \beta^+(B_{k\ell})&=\beta(B_k)+\beta(B_\ell),\\
 \lambda^+(B_{k\ell})&=\lambda(B_k)+\lambda(B_\ell),
\end{align*}
with the rest of $(\pi^+,\beta^+,\lambda^+)$
the same as in $(\pi,\beta,\lambda)$.

For (iv), if $G^+$ adds a copy of the loop on vertex $i$ to $G$, then $\pi^+=\pi$. If $i$ lies in the block $B_k$, one has $\beta^+(B_\ell)=\beta(B_\ell)$
and $\lambda^+(B_\ell)=\lambda(B_\ell)$ for 
$\ell \neq k$, with 
\begin{align*}
\beta^+(B_k)&=\beta(B_k)+1,\\
\lambda^+(B_k)&=\lambda(B_k)+1.\qedhere
\end{align*}
\end{proof}

\begin{definition} \rm

Define a $\kk$-vector space decomposition
\begin{equation}
\label{def:CL-data-grading}
\kk[u_i,v_{ij},w_{ij}]
= \bigoplus_{(\pi,\beta,\lambda)}
\kk[u_i,v_{ij},w_{ij}]_{(\pi,\beta,\lambda)}
\end{equation}
where $\kk[u_i,v_{ij},w_{ij}]_{(\pi,\beta,\lambda)}$
is the $\kk$-span of all monomials $q$
whose multigraph $G(q)$ has CL-data $(\pi,\beta,\lambda)$.

Say that a polynomial $f$ is {\it CL-homogeneous} if it lies in some $\kk[u_i,v_{ij},w_{ij}]_{(\pi,\beta,\lambda)}$.
Call a $\kk$-linear subspace $V \subset \kk[u_i,v_{ij},w_{ij}]$ {\it CL-homogeneous}
if every $f$ in $V$ has some expression
$f=\sum_{(\pi,\beta,\lambda)} f_{(\pi,\beta,\lambda)}$
where each $f_{(\pi,\beta,\lambda)}$ is CL-homogeneous.  The direct sum decomposition \eqref{def:CL-data-grading} makes such  expressions unique, so that
\begin{equation}
\label{direct-sum-expression-for-homogeneity}
V=\bigoplus_{(\pi,\beta,\lambda)} V \cap  \kk[u_i,v_{ij},w_{ij}]_{(\pi,\beta,\lambda)}.
\end{equation}
Equivalently,  $V$
is CL-homogeneous if and only if it has a $\kk$-spanning set consisting of homogeneous elements.
\end{definition}

\begin{prop}
\label{prop:monomial-times-CL-homog-is-CL-homog}
Products $q \cdot f$ of monomials $q$ and
CL-homogeneous $f$  in $\kk[u_i,v_{ij},w_{ij}]$ are CL-homogeneous.
\end{prop}
\begin{proof}
Induct on $\deg(q)$, writing $q$ as a product of variables $u_i, v_{ij}, w_{ij}$, and use Proposition~\ref{prop:CL-data-properties}(iii),(iv).
\end{proof}

This leads to our desired presentation of $\gr(\TypeBCohomology{n})$ as a quotient $\kk$-vector space.

\begin{definition}\rm
\label{def:linear-presentation-subspaces}
Define these two nested $\kk$-linear spaces of $I_n \subset F_n$ inside $\kk[u_i,v_{ij},w_{ij}]$:
\begin{itemize}
\item $F_n$ is the $\kk$-span within $\kk[u_i,v_{ij},w_{ij}]$ of monomials $q$
having $G(q)$ a lightly-looped forest,
\item $I_n$ is the subspace
spanned by products $q \cdot g$
in $F_n$, with $q$ a monomial and $g$ in Table~\ref{table:vector-space-relations} below:
\begin{center}
\begin{table}[!h]
\centering
\setlength{\tabcolsep}{10pt} 
\renewcommand{\arraystretch}{1.5} 
\begin{tabular}{|l|l|}  \hline  
 $u_{i}w_{ij} -u_{j} v_{ij}$\\ 
 $u_{i} v_{ij}  - u_{j} w_{ij}$\\
 $v_{ij}w_{jk} - w_{ij}w_{ik} - v_{ik}v_{jk}$\\
 $w_{ij}w_{jk} - v_{ij}w_{ik} - w_{ik}w_{jk}$\\
 $v_{ij}v_{jk} - v_{ij}v_{ik} - v_{ik}w_{jk}$ \\
 $w_{ij}v_{jk} - w_{ij}v_{ik} - w_{ik}v_{jk}$ \\
 \hline 
\end{tabular}
\caption{The generators of $\I_\gr$ from Table~\ref{table:Igrrelations} that also lie in $F_n$.}
\label{table:vector-space-relations}
\end{table}
\end{center}
\end{itemize}
\end{definition}

\begin{prop}
\label{prop:vectorspaceV}
One has an isomorphism of $\kk$-vector spaces
and $\kk\symmB{n}$-modules, 
\[ 
\gr(\TypeBCohomology{n}) \cong F_n/I_n.
\]
\end{prop}
\noindent
{\bf Note:} Though $\gr(\TypeBCohomology{n})$ is an algebra,
the quotient
$F_n/I_n$ is not;
Proposition~\ref{prop:vectorspaceV} is only
about vector spaces.
\begin{proof}[Proof of Proposition~\ref{prop:vectorspaceV}]
Recall the $\vbasis$-presentation, 
$\gr(\TypeBCohomology{n}) \cong \kk[u_i,v_{ij},w_{ij}]/\I_\gr$.
Proposition~\ref{prop:spanned-by-lightly-looped-forests}(v) then tells us that the composite of these two maps is a surjection:
$$
F_n \hookrightarrow \kk[u_i,v_{ij},w_{ij}] \twoheadrightarrow 
\gr(\TypeBCohomology{n}).
$$
The kernel of the composite is $F_n \cap \I_\gr$, so it induces a $\kk$-vector space
and $\kk \symmB{n}$-module isomorphism
$$
F_n/(F_n \cap \I_\gr) \cong
\gr(\TypeBCohomology{n}).
$$
It therefore only remains to show that $I_n=F_n \cap \I_\gr$.
The inclusion $I_n \subseteq F_n \cap \I_\gr$ is easy:  any
$g$ in Table~\ref{table:vector-space-relations} is also a generator of $\I_\gr$ from Table~\ref{table:Igrrelations}, so any of the spanning elements $q \cdot g$ of $I_n$ also lies in $\I_\gr$.

For the reverse inclusion $F_n \cap \I_\gr \subseteq I_n$,
first check that all the generators $g$ for $\I_\gr$ given in Table~\ref{table:Igrrelations} 
are also CL-homogeneous.  Hence for any monomial $q$, the products
$q \cdot g$ are also CL-homogeneous, by Proposition~\ref{prop:monomial-times-CL-homog-is-CL-homog}. Since $\I_\gr$ is $\kk$-spanned by all such products, $\I_\gr$ is a CL-homogeneous subspace of $\kk[u_i,v_{ij},w_{ij}]$, and as in \eqref{direct-sum-expression-for-homogeneity}, one has 
$$
I_\gr =
\bigoplus_{(\pi,\beta,\lambda)} 
I_\gr \cap  \kk[u_i,v_{ij},w_{ij}]_{(\pi,\beta,\lambda)}.
$$
In particular,
$I_\gr \cap  \kk[u_i,v_{ij},w_{ij}]_{(\pi,\beta,\lambda)}$
is spanned by the products $q \cdot g$ in $\kk[u_i,v_{ij},w_{ij}]_{(\pi,\beta,\lambda)}$ where $q$ is a monomial,
and $g$ is a generator of $\I_\gr$ from Table~\ref{table:Igrrelations}.
On the other hand, Proposition~\ref{prop:CL-data-properties}(ii) implies that $F_n$ is the direct sum of certain of the components
$\kk[u_i,v_{ij},w_{ij}]_{(\pi,\beta,\lambda)}$,
namely those for which $\beta(B_i) \leq 1$ and $\beta(B_i)=1$ implies
$\lambda(B_i)=1$.  Consequently, $F_n \cap I_\gr$ is
$\kk$-spanned by those products $q \cdot g$ in $F_n$ where $q$ is a monomial, and
$g$ is a generator of from Table~\ref{table:Igrrelations}.  However, this forces $g$ itself to lie in $F_n$, that is, $g$ comes from the smaller Table~\ref{table:vector-space-relations}, and hence $q \cdot g$ lies in $I_n$. 
Consequently, $F_n \cap \I_\gr \subseteq I_n$.
\end{proof}


\section{Proof of Theorem \ref{decomposition-of-peak-idempotent-reps-theorem}: Peak representations as sums of higher Lie characters}
\label{sec:proofofpeakreps}

Recall the statement, which used $\odd(\lambda)$ to denote the number of odd parts in a partition $\lambda$.

\vskip.1in
\noindent
{\bf Theorem~\ref{decomposition-of-peak-idempotent-reps-theorem}.}
{\it 
When $\fieldchar(\kk) >n$, for  $0 \leq k \leq n$ with $k$ even, one has a $\kk \symm{n}$-module  isomorphism 
$$
\left( \kk \symm{n} \right) \PeakIdempotent{n}{n-k} 
\cong \bigoplus_{\substack{|\lambda|=n:\\ \odd(\lambda)=n-k}} \Lie{\lambda}  \cong
H^{2k} \PeakConf{n}.
$$    
Hence the Betti number $\dim_{\kk} H^{2k} \PeakConf{n}$ counts permutations in $\symm{n}$ with $n-k$ odd cycles.

Furthermore, the associated graded ring
for $H^* \PeakConf{n}$ has the following $\kk \symm{n}$-module description.  Consider 
the bigraded component $\gr(\TypeBCohomology{n})_{k,\ell}^{\Z_2^n}$   of $H^{2k} \PeakConf{n}$ corresponding to filtration degree $\ell$ in variables indexed by the $x_i=\pm x_j$ hyperplanes.  Then this component is $\kk \symm{n}$-module isomorphic to
$$
\bigoplus_{\substack{|\lambda|=n:\\ \ell(\lambda)=n-\ell\\ \odd(\lambda)=n-k }}
\Lie{\lambda}.
$$
}
\vskip.1in

The proof comes from an 
unexpected $\kk \symm{n}$-module isomorphism, mentioned as \eqref{singly-graded-isomorphism}
in the Introduction
$$
(H^*\PeakConf{n} \cong ) \quad  
\gr(\TypeBCohomology{n})^{\Z_2^n} \cong \TypeACohomology{n} \quad (= H^*\TypeAConf{n}),
$$
that will be another application of 
the Pairing Lemma~\ref{lem:pairing-bijection}.
Recall the
$\kk$-linear subspace $F_n \subset
\kk[u_i,v_{ij},w_{ij}]$
from Definition~\ref{def:linear-presentation-subspaces}, which gave rise to the $\kk \symmB{n}$-module presentation
$$
\gr(\TypeBCohomology{n}) \cong F_n/I_n
$$
from Proposition~\ref{prop:vectorspaceV}.
We wish to consider the composite of
these three maps:
\begin{equation}
\label{three-linear-map-composite}
F_n \hookrightarrow 
\kk[u_i,v_{ij},w_{ij}]
\overset{\gamma}{\longrightarrow} \kk[t_{ij}]
\twoheadrightarrow \kk[t_{ij}]/\I_{\symm{n}}
\cong \TypeACohomology{n}.
\end{equation}

\begin{cor}
    \label{Sn-equivariant-surjection-and-iso}
The composite in \eqref{three-linear-map-composite}
descends to a surjection of $\kk \symm{n}$-modules (but not of $\kk$-algebras) 
$$
\gr(\TypeBCohomology{n}) \cong  F_n/I_n
\twoheadrightarrow 
\TypeACohomology{n}.
$$
which then further restricts to an isomorphism of 
$\kk\symm{n}$-modules (but not of $\kk$-algebras)
$$
\gr(\TypeBCohomology{n})^{\Z_2^n}
\overset{\sim}{\longrightarrow} 
\TypeACohomology{n}.
$$
In particular, this last isomorphism gives a $\kk\symm{n}$-module isomorphism
$$
H^* \PeakConf{n} \cong H^* \TypeAConf{n},
$$
but one which does {\bf not} respect the cohomological grading, that is, 
$H^i \PeakConf{n} \not\cong H^i \TypeAConf{n}$.

\end{cor}
\begin{proof}
Note all three
maps composed in \eqref{three-linear-map-composite} are maps of $\kk\symm{n}$-modules,
and hence so is the composite.

To check that the composite descends to the
quotient $F_n/I_n$, one must show $\gamma(I_n) \subset \I_{\symm{n}}$.
Recall that $I_n$ is the subspace
spanned by products $q \cdot g$
in $F_n$ with $q$ a monomial and $g$ listed in Table~\ref{table:vector-space-relations}.
Since in this situation, one has $\gamma(q\cdot g)=\gamma(q) \cdot \gamma(g)$, it will suffice to check that each $g$ in the table has $\gamma(g)$ lying in the ideal $\I_{\symm{n}}$.  Note that
the first two entries in Table~\ref{table:vector-space-relations} have vanishing image under $\gamma$
$$
\gamma(u_{i}w_{ij} -u_{j} v_{ij})=
\gamma(u_{i} v_{ij}  - u_{j} w_{ij})=
1 \cdot t_{ij}-1 \cdot t_{ij}=0
$$
while the last four entries in the table
\begin{align*}
 &v_{ij}w_{jk} - w_{ij}w_{ik} - v_{ik}v_{jk}\\
 &w_{ij}w_{jk} - v_{ij}w_{ik} - w_{ik}w_{jk}\\
 &v_{ij}v_{jk} - v_{ij}v_{ik} - v_{ik}w_{jk} \\
 &w_{ij}v_{jk} - w_{ij}v_{ik} - w_{ik}v_{jk} 
 \end{align*}
 all have the same image $t_{ij}t_{jk} - t_{ij}t_{ik} - t_{ik}t_{jk}$ under $\gamma$,
which is a generator of $\I_{\symm{n}}$.

Hence $\gamma(I_n) \subset \I_{\symm{n}}$,
and the composite maps induces a map
on the quotient $F_n/I_n \rightarrow \kk[t_{ij}]/I_{\symm{n}} =\TypeACohomology{n}$.
To see that it is surjective,
note that Lemma~\ref{lem:pairing-bijection} lets one re-express the $\kk$-basis $\typeabasis$ for
  $\TypeACohomology{n}$ 
as follows: 
$$
\typeabasis = \left\{ \, \gamma(\phi(m))\,\, : m \in \typeabasis \, \right\} = \gamma\left( \tvbasis \,\,\right).
$$
Note the images of $\tvbasis\subset F_n$ give a $\kk$-basis for $\gr(\TypeBCohomology{n})^{\Z_2^n} \subset  \gr(\TypeBCohomology{n}) \cong F_n/I_n.$
Thus this map $F_n/I_n \rightarrow \TypeACohomology{n}$ surjects, 
and restricts to
an isomorphism
$\gr(\TypeBCohomology{n})^{\Z_n} \cong \TypeACohomology{n}$,
sending the $\kk$-basis $\tvbasis$ to the $\kk$-basis $\typeabasis$.
\end{proof}

This isomorphism  
$\gr(\TypeBCohomology{n})^{\Z_2^n}
\overset{\sim}{\longrightarrow} 
\TypeACohomology{n}$
in Corollary~\ref{Sn-equivariant-surjection-and-iso}  
will allow us to align the $\kk\symm{n}$-module decomposition 
$$
\TypeACohomology{n} \cong \bigoplus_{\lambda \vdash n} \TypeACohomology{\lambda}
$$
from \eqref{eq:typeABflatorbitdecomposition}
with the $\kk\symm{n}$-module decomposition 
$$
\gr(\TypeBCohomology{n})^{\Z_2^n}
=\bigoplus_{\slambda\vdash n}
\gr(\TypeBCohomology{n})^{\Z_2^n}_{\slambda}
$$
that comes from taking ${\Z_2^n}$-fixed points in the decomposition 
in Proposition~\ref{double-partition-decomposition-of-grY}
$$
\gr(\TypeBCohomology{n})=
\bigoplus_{\slambda \vdash n}
\gr(\TypeBCohomology{n})_{\slambda}.
$$
Given $\lambda$, define $\oddparts(\lambda)$ 
(resp. $\evenparts(\lambda)$) to be the partition obtained by selecting only the odd (resp. even) parts of $\lambda$.
Also define their  lengths
\begin{align*}
    \odd(\lambda) &:= \ell(\oddparts(\lambda))\\
\even(\lambda)&:=\ell(\evenparts(\lambda)).
\end{align*}
Say $\lambda$ is an \emph{odd  partition} (resp., \emph{even partition}) if all of its parts are odd (resp., even).
 
\begin{cor}
\label{cor:representation-alignment}
To each partition $\lambda \vdash n$
associate the double partition
$$\slambda:=(\oddparts(\lambda),\evenparts(\lambda)) \vdash n.
$$

Then the $\kk\symm{n}$-module isomorphism  
$\gr(\TypeBCohomology{n})^{\Z_2^n}
\overset{\sim}{\longrightarrow} 
\TypeACohomology{n}$
restricts to a $\kk\symm{n}$-module isomorphism
$$
\gr(\TypeBCohomology{n})^{\Z_2^n}_{\slambda}
\overset{\sim}{\longrightarrow} 
\TypeACohomology{\lambda}.
$$
In particular, $\gr(\TypeBCohomology{n})^{\Z_2^n}_{\slambda}=0$ unless one has both that $\lambda^+$ is odd and
that $\lambda^-$ is even.
\end{cor}

\begin{example} \rm
For example, when $n=4$, Corollary \ref{cor:representation-alignment} gives the following isomorphisms:
\begin{align*}
\gr(\TypeBCohomology{4})^{\Z_2^4}_{(\emptyset, (4))}
&\overset{\sim}{\longrightarrow} 
\TypeACohomology{(4)}\\
\gr(\TypeBCohomology{4})^{\Z_2^4}_{((3,1),\emptyset)}
&\overset{\sim}{\longrightarrow} 
\TypeACohomology{(3,1)}\\
\gr(\TypeBCohomology{4})^{\Z_2^4}_{(\emptyset, (2,2))}
&\overset{\sim}{\longrightarrow} 
\TypeACohomology{(2,2)}\\
\gr(\TypeBCohomology{4})^{\Z_2^4}_{((1,1), (2))}
&\overset{\sim}{\longrightarrow} 
\TypeACohomology{(2,1,1)}\\
\gr(\TypeBCohomology{4})^{\Z_2^4}_{((1,1,1,1), \emptyset)}
&\overset{\sim}{\longrightarrow} 
\TypeACohomology{(1,1,1,1)}.
\end{align*}
\end{example}

\begin{proof}[Proof of Corollary~\ref{cor:representation-alignment}]
Proposition~\ref{prop:bihomogeneous-gr(Z)-basis} showed that the images of
 $\tvbasis$ in $\gr(\TypeBCohomology{n})^{\Z_2^n}$ 
form a $\kk$-basis.   Hence it suffices to check for each monomial $q$ in $\tvbasis$, if its multigraph $G(p)$
has double partition $\slambda$,
then these hold: 
\begin{itemize}
    \item[(i)] $\lambda^+$ is odd, $\lambda^-$ is even, and 
    \item[(ii)] the image monomial $\gamma(q)=\prod_m t_{i_m,j_m}$
in $\kk[t_{ij}]$ has the $\symm{n}$-orbit $[X]$ of $X=\bigcap_m \{x_{i_m}=x_{j_m}\}$ corresponding to the partition $\lambda$ uniquely defined by $\lambda^+=\oddparts(\lambda)$ and $ \lambda^-=\evenparts(\lambda)$.
\end{itemize}

Recall from the proof of Proposition~\ref{prop:spanned-by-lightly-looped-forests} that $G(q)$ is a lightly-looped forest because $q$ lies in $\vbasis$.  Because it also lies in $\prod \quadgens$, inspection of the quadratic monomials $\quadgens$ in \eqref{Z-generators} shows that each of the unlooped trees in $G(q)$ must have evenly many edges, coming from factors all
of the form $v_{ij} w_{jk}$ or  $w_{ij}w_{ik}$.  Thus its unlooped trees
all have an odd number of vertices,
so that $\lambda^+$ is odd. Similarly, each of the looped trees in $G(q)$ must have an odd number of nonloop edges, due to exactly one factor of the form $u_i w_{ij}$. Thus the looped trees in $G(q)$ have an even number of vertices, so that $\lambda^-$ is even.

Since $\gamma$ sends $u_i \mapsto 1$ and $v_{ij},w_{ij} \mapsto t_{ij}$, the image monomial $\gamma(q)=\prod_m t_{i_m,j_m}$
in $\kk[u_{ij}]$ has corresponding type $A$ intersection subspace
$X=\bigcap_m \{x_{i_m}=x_{j_m}\}$ corresponding to the set partition $\pi_X$
of $[n]$ whose blocks are the connected components of $G(q)$.
Hence its $\symm{n}$-orbit $[X]$ in 
$\lat(\A_{\symm{n}})/\symm{n}$ is indexed by the partition
$\lambda$ that merges all the parts in $\lambda^+$ and $\lambda^-$.  Therefore $\oddparts(\lambda)=\lambda^+$ and $\evenparts(\lambda)=\lambda^-$.
\end{proof}

\begin{proof}[Proof of Theorem~\ref{decomposition-of-peak-idempotent-reps-theorem}]

It suffices to prove the second (``Furthermore") assertion, which implies the rest, and describes the $\kk\symm{n}$-module structure of the bigraded component $\gr(\TypeBCohomology{n})^{\Z_2^n}_{k,\ell}$.

Start with $\kk\symm{n}$-module decompositions and isomorphisms discussed in \eqref{bigrading-on-gr(Z)}
$$
\gr(\TypeBCohomology{n})^{\Z_2^n}
=\bigoplus_{k=0}^n \gr(\TypeBCohomology{n})^{\Z_2^n}_k
=\bigoplus_{k=0}^n \,\, 
\bigoplus_{\ell=0}^k \,\, \gr(\TypeBCohomology{n})^{\Z_2^n}_{k,\ell},
$$
which can be refined further by taking $\Z_2^n$-fixed spaces in the decomposition of 
Proposition~\ref{prop:double-partition-to-bigrading}
$$
\gr(\TypeBCohomology{n})^{\Z_2^n}_{k,\ell}\
=
\bigoplus_{\substack{\slambda \vdash n\\ \ell(\lambda^+)=n-k\\ \ell(\lambda^+)+\ell(\lambda^-)=n-\ell}} \,\,
 \gr(\TypeBCohomology{n})^{\Z_2^n}_{\slambda}.
$$
This can be rewritten using Corollary~\ref{cor:representation-alignment} 
and the isomorphism $\TypeACohomology{\lambda} \cong \Lie{\lambda}$ from \eqref{Lie-lambda-isos}
as follows:
\begin{align*}
\gr(\TypeBCohomology{n})^{\Z_2^n}_{k,\ell}
&\cong 
\bigoplus_{\substack{\lambda \vdash n\\ \odd(\lambda)=n-k\\ \ell(\lambda)=n-\ell}}
\TypeACohomology{\lambda}
\cong 
\bigoplus_{\substack{\lambda \vdash n\\ \odd(\lambda)=n-k\\ \ell(\lambda)=n-\ell}}
\Lie{\lambda} \qedhere.
\end{align*}
\end{proof}

\noindent
Note that the above formulas show that
$\gr(\TypeBCohomology{n})^{\Z_2^n}_{k,\ell}=0$
unless $k$ is even, since $n=|\lambda| \equiv \odd(\lambda) \bmod{2}$.  This is consistent
with Theorem~\ref{peak-idempotents-give-cohomology-reps}(iii),
which asserts that
$
\gr(\TypeBCohomology{n})_{k}^{\Z_2^n} 
\cong (\TypeBCohomology{n})^{\Z_2^n}
= H^{2k} \PeakConf{n}=0
$
unless $2k \equiv 0 \bmod{4}$.

\section{Primitive peak representations}
\label{sec:connectiontopeakreps}

Corollary~\ref{cor:representation-alignment} has further consequences for
refinements of the peak idempotents 
$\{\PeakIdempotent{n}{k}\}$.  Recall
that the latter were defined using the
surjective $\kk$-algebra map 
$
\varphi: \kk \symmB{n} \twoheadrightarrow \kk \symm{n}
$
coming from $\symm{n} \cong \symmB{n}/\Z_2^n$, via
$$
\PeakIdempotent{n}{k}:=
\varphi(\TypeBEulerianIdempotent{n}{k}).
$$
On the other hand, recall that the type $B$ Eulerian idempotents $\{ \TypeBEulerianIdempotent{n}{k}\}$ can be written as sums of the more refined complete family of orthogonal BBHT idempotents $\{ \TypeBEulerianIdempotent{n}{\mu} \}_{|\mu| \leq n}$ for $W=\symmB{n}$ from \cite{bergeronbergeron,BBHT}. Importantly, the idempotents $\TypeBEulerianIdempotent{n}{\mu}$ are \emph{primitive} idempotents of $\sol(\symmB{n})$, meaning that they cannot be written as a non-trivial sum of other idempotents in $\sol(\symmB{n})$.

The general
relation \eqref{general-BBHT-to-Eulerian-refinement}
for $W=\symmB{n}$ becomes
\begin{equation}
\label{typeB-Eulerians-as-sums-of-BBHTs}
\TypeBEulerianIdempotent{n}{k}
=\sum_{\substack{|\mu| \leq n\\\ell(\mu)=k}}
\TypeBEulerianIdempotent{n}{\mu}.
\end{equation}
This motivates the definition of peak idempotents 
$\{\PeakIdempotent{n}{\mu} \}_{|\mu \leq n}$ within $\PeakAlgebra{n}$:
$$
\PeakIdempotent{n}{\mu}:= \varphi(\TypeBEulerianIdempotent{n}{\mu}).
$$
Since the 
$\{ \TypeBEulerianIdempotent{n}{\mu} \}_{|\mu| \leq n}$ form a complete family of orthogonal idempotents in $\kk \symmB{n}$, the finer peak idempotents
$\{ \PeakIdempotent{n}{\mu} \}_{|\mu| \leq n}$ form a complete family of orthognal idempotents in $\kk \symm{n}$, and one has a $\kk\symm{n}$-module decomposition
$$
\kk \symm{n} = \bigoplus_{|\mu| \leq n}
(\kk \symm{n})\PeakIdempotent{n}{\mu}.
$$
It will turn out that the non-vanishing elements $\PeakIdempotent{n}{\mu}$ are primitive idempotents of $\PeakAlgebra{n}$.

Here is the description of these finer peak idempotent representations, and their vanishing.

\begin{cor}
\label{cor:peakidempotentrep_primitive}
For each partition $\mu$ with $|\mu| \leq n$, one has a $\kk\symm{n}$-module isomorphism
$$
(\kk \symm{n})\PeakIdempotent{n}{\mu}
\cong
(\TypeBCohomology{n})^{\Z_2^n}_\mu
= \bigoplus_{\substack{\lambda \vdash n\\\oddparts(\lambda)=\mu}} \lie_{\lambda}.
$$
In particular, $\PeakIdempotent{n}{\mu}=0$ unless $\mu$ is an odd partition and $n-|\mu|$ is even.
\end{cor}

\begin{proof}
Calculate as follows:
\begin{align*}
(\kk \symm{n})\PeakIdempotent{n}{\mu}
&\cong
\left(
(\kk \symmB{n}) \TypeBEulerianIdempotent{n}{\mu}
\right)^{\Z_2^n}& \text{ by  Lemma~\ref{normal-subgroup-idempotent-lemma}},\\
&=(\TypeBCohomology{n})^{\Z_2^n}_\mu
& \text{ by Theorem~\ref{thm:eulerian-reps-as-cohomology}},\\
&\cong \bigoplus_{\substack{\slambda \vdash n:\\ \lambda^{+} = \mu}} 
\gr(\TypeBCohomology{n})_{\slambda}^{\Z_2^n}
&\text{ by Proposition~\ref{prop:compatiblefiltrations}},\\
&= \bigoplus_{\substack{\lambda \vdash n\\\oddparts(\lambda)=\mu}} \lie{\lambda}
& \text{ by Corollary~\ref{cor:representation-alignment}}.
\end{align*}
The vanishing assertion follows,
since if $\oddparts(\lambda)=\mu$,
then $n-|\mu|= |\evenparts(\lambda)|$,
an even number.
\end{proof}

We will use Corollary \ref{cor:peakidempotentrep_primitive} to show that the non-vanishing $\PeakIdempotent{n}{\mu}$ are in fact primitive idempotents of $\PeakAlgebra{n}$.

\begin{cor}
    The collection $\{ \PeakIdempotent{n}{\mu} \}_{\mu}$ for $\mu$ an odd partition and $n- |\mu|$ even gives a complete system of primitive, orthogonal idempotents of $\PeakAlgebra{n}$.
\end{cor}
\begin{proof}
Any finite dimensional algebra $A$ has a complete system of primitive, orthogonal idempotents $e_1, \cdots, e_\ell$ where $\ell $ is the dimension of $A$ modulo its radical $\rad(A)$; see \cite[Appendix D]{aguiar2017topics}. 

By Corollary 4.3 of \cite{aguiar2006peak}, the dimension of $\PeakAlgebra{n}/\rad(\PeakAlgebra{n})$ is the number of \emph{almost odd partitions}, which are partitions $\lambda$ of $n$ with exactly one even part (possibly of size zero). By Corollary \ref{cor:peakidempotentrep_primitive} each non-zero $\PeakIdempotent{n}{\mu}$ can be indexed by the almost odd partition $\mu \cup \{ (n-|\mu|) \}$, and so 
\[ |\{ \PeakIdempotent{n}{\mu}: |\mu|\leq n, \ \PeakIdempotent{n}{\mu} \neq 0 \}| = \dim(\PeakAlgebra{n}/\rad(\PeakAlgebra{n})).\]

Suppose some $\PeakIdempotent{n}{\mu}$ were not primitive; then the complete system of  idempotents $\{ \PeakIdempotent{n}{\mu} \}$ could be refined to a primitive system of idempotents. But this would give more idempotents than the number of almost odd partitions, a contradiction. 
 \end{proof}

  \begin{example} \rm The primitive peak representations 
  $(\kk \symm{n})\PeakIdempotent{n}{\mu}$ for $n=4$ have the following descriptions: 
\begin{align*}
    (\kk \symm{4})\PeakIdempotent{4}{\emptyset} &\cong \lie_{(2,2)} \oplus \lie_{(4)} \\  
    (\kk \symm{4})\PeakIdempotent{4}{(1,1)} &\cong \lie_{(2,1,1)}\\
    (\kk \symm{4})\PeakIdempotent{4}{(3,1)} &\cong \lie_{(3,1)} \\
    (\kk \symm{4})\PeakIdempotent{4}{(1,1,1,1)} &\cong \lie_{(1,1,1,1)} .
\end{align*}
    \end{example}

We remark that since $\lie_\lambda \cong (\kk \symm{n})\TypeAEulerianIdempotent{n}{\lambda}$,
the isomorphism in Corollary~\ref{cor:peakidempotentrep_primitive}
suggests a stronger conjecture.

\begin{conj}
\label{conj:peak-idempotent-as-sum-of-Lies}
For any odd partition $\mu$ with $|\mu|\leq n$ one has
\[ 
\PeakIdempotent{n}{\mu} 
= \sum_{\substack{\lambda \vdash n:\\ \oddparts(\lambda) = \mu}} \TypeAEulerianIdempotent{n}{\lambda}.
\]
\end{conj}

\section{Generating functions, recursions and branching rules}
\label{sec:recursions}

We explore here what the previous results say
about the Hilbert series of 
$H^* \PeakConf{n}=(H^* \TypeBConf{n})^{\Z_2^n}=(\TypeBCohomology{n})^{\Z_2^n}$, as well
as its bigraded counterpart $\gr(\TypeBCohomology{n})^{\Z_2^n}$.  We then lift these to equivariant versions that keep track of the
$\kk\symm{n}$-module structure when $\kk$ has
characteristic zero.

\subsection{Ordinary Hilbert series and recursions}
\label{sec:hilbertseries}

Recall from Section~\ref{sec:proofofpeakreps}
that $\gr(\TypeBCohomology{n})_{k,\ell}^{\Z_2^n}$ vanishes unless $k$ is even, and define the following bigraded Hilbert series
$$
\bihilb_n:=\bihilb_n(t,q):=\sum_{k,\ell} 
t^k q^\ell \dim_{\kk} \gr(\TypeBCohomology{n})^{\Z_2^n}_{2k,\ell}
$$
whose $q=1$ specialization is (a rescaling of) the
{\it Poincar\'e polynomial} for $H^* \PeakConf{n}$:
\begin{align}
\bihilb_n(t,1)
=\sum_{k} 
t^k \dim_{\kk} \gr(\TypeBCohomology{n})^{\Z_2^n}_{2k,*}
=\sum_{k} 
t^k \dim_{\kk}  (\TypeBCohomology{n})^{\Z_2^n}_{2k}
=\sum_{k} 
t^k \dim_{\kk} H^{4k} \PeakConf{n}.
\end{align}
The rescalings $\bihilb_n(t^2,q)$ for $n=2,3,4$ appeared in
\eqref{first-few-bigraded-hilbs}, and reappear in the second column of the table in Example~\ref{ex:equivariant-data}
below. Theorem~\ref{decomposition-of-peak-idempotent-reps-theorem}
lets one evaluate $\bihilb_n$, via a known expression
for the $\kk\symm{n}$-module $\Lie{\lambda}$: it is induced from a  degree one character of the centralizer for a permutation of cycle type $\lambda$; see, e.g., Reutenauer \cite[Thm. 8.24]{reutenauer}.  Thus, as was mentioned in the Introduction, its dimension is the index of this centralizer subgroup, or the size of the conjugacy class:
\begin{equation}
\label{Lie-character-dimension}
\dim_{\kk} \Lie{\lambda} = |\{ \text{permutations }\sigma\text{ in }\symm{n}\text{ of cycle type }\lambda\}|.
\end{equation}

\begin{cor}
\label{cor:nonequivariant-bihilb-facts}
One has these expressions, recursions, and generating functions for  $\bihilb_n=\bihilb_n(t,q)$:
\begin{align}
\label{nonequivariant-bihilb-expression}
\bihilb_n(t,q)&=\sum_{\sigma \in \symm{n}} t^{\frac{n-\odd(\sigma)}{2}} q^{n-\cyc(\sigma)},\\
\label{nonequivariant-bihilb-recursion}
\bihilb_{n}(t,q) &= \bihilb_{n-1}(t,q) + tq (n-1) \cdot 
\big( 1 + q (n-2) \big) \cdot \bihilb_{n-2}(t,q),\\
\label{nonequivariant-bihilb-gf}
\sum_{n=0}^\infty \frac{x^n}{n!} 
\bihilb_n(t,q)&=
\left[ 
\left( 1-z \right)^{-\frac{a+b}{2}} \cdot \left( 1+z \right)^{\frac{a-b}{2}}
\right]_{\substack{a=q^{-1}\\ b=t^{-1/2}q^{-1} \\z=t^{1/2}qx}}
\end{align}  
where $\cyc(\sigma), \odd(\sigma)$ are the number of cycles and the number of odd-sized cycles in $\sigma$.
\end{cor}

\begin{proof}
For \eqref{nonequivariant-bihilb-expression}, Theorem~\ref{decomposition-of-peak-idempotent-reps-theorem}    implies
\begin{align*}
\bihilb_n(t^2,q)
&=\sum_{\lambda \vdash n} \dim_{\kk} \Lie{\lambda} 
\cdot t^{n-\odd(\lambda)} q^{n-\ell(\lambda)}\\
&=\sum_{\lambda \vdash n}
|\{\sigma \in \symm{n}\text{ of cycle type }\lambda\}| \cdot 
t^{n-\odd(\lambda)} q^{n-\ell(\lambda)}
=\sum_{\sigma \in \symm{n}} t^{n-\odd(\sigma)} q^{n-\cyc(\sigma)}.
\end{align*}
For \eqref{nonequivariant-bihilb-recursion}, segregate 
summands of \eqref{nonequivariant-bihilb-expression} by whether $n$ lies in a cycle of $\sigma$ of size $1, 2$ or at least $3$:
\begin{align*}
\bihilb_n(t,q)
&=\sum_{\sigma \in \symm{n}} t^{\frac{n-\odd(\sigma)}{2}} q^{n-\cyc(\sigma)}\\
&=\sum_{\substack{\sigma \in \symm{n}:\\ 
\sigma(n)=n}}
t^{\frac{n-\odd(\sigma)}{2}} q^{n-\cyc(\sigma)}
+
\sum_{i=1}^{n-1} 
\sum_{ \substack{ \sigma \in \symm{n}:\\ \sigma(n)=i\\ \sigma(i)=n}}
t^{\frac{n-\odd(\sigma)}{2}} q^{n-\cyc(\sigma)}
+
\sum_{\substack{(i,j):\\1 \leq i,j \leq n-1\\i \neq j}} 
\sum_{ \substack{\sigma \in \symm{n}:
\\ \sigma(n)=i \\ \sigma(i)=j }} 
t^{\frac{n-\odd(\sigma)}{2}} q^{n-\cyc(\sigma)}\\
&=\bihilb_{n-1}(t,q) + 
(n-1) \cdot t^1 q^1 \bihilb_{n-2}(t,q)
+ (n-1)(n-2) \cdot t^1 q^2 \bihilb_{n-2}(t,q)\\
&= \bihilb_{n-1}(t,q) + tq (n-1)  \cdot 
\big( 1 + q(n-2) \big) \cdot \bihilb_{n-2}(t,q).
\end{align*}

For \eqref{nonequivariant-bihilb-gf}, note that
 \eqref{nonequivariant-bihilb-expression} shows 
$$
\sum_{n=0}^\infty \frac{x^n}{n!} 
\bihilb_n(t,q)=
\left[ 
\sum_{n=0}^\infty
\frac{z^n}{n!} 
\sum_{w \in S_n} a^{\odd(w)} b^{\even(w)}
\right]_{\substack{a=q^{-1}\\ b=t^{-1/2}q^{-1} \\z=t^{1/2}qx}}
$$

The sum in the brackets is a specialization of
Touchard's {\it cycle indicator} generating function
\cite{Touchard}
$$
\sum_{n=0}^\infty
\frac{z^n}{n!} 
\sum_{w \in S_n} a_1^{m_1(w)} a_2^{m_2(w)} a_3^{m_3(w)} \cdots 
=\exp
\left( \sum_{m=1}^\infty \frac{a_m z^m}{m}
\right)
$$
with $m_j(w)$ the number of $j$-cycles in $w$. Setting $a_m=a$ for $m$ odd and 
$a_m=b$ for $m$ even gives
$$
\begin{aligned}
\sum_{n=0}^\infty
\frac{z^n}{n!} 
\sum_{w \in S_n} a^{\odd(w)} b^{\even(w)} 
&=\exp
\left(
a \sum_{\substack{m \geq 1\\m\text{ odd}}} \frac{z^m}{m}
+
b \sum_{\substack{m \geq 2\\m\text{ even}}} \frac{z^m}{m}
\right)\\
&=\exp
\left( 
a \frac{-\log(1-z)+
\log(1+z)}{2}
+
b \frac{-\log(1-z)-
\log(1+z)}{2}
\right)\\
&= \left( 1-z \right)^{-\frac{a+b}{2}} \ \left( 1+z \right)^{\frac{a-b}{2}}. \qedhere
\end{aligned} 
$$
\end{proof}

\begin{remark}\rm
It is worth noting how the assertions of Corollary~\ref{cor:nonequivariant-bihilb-facts}
specialize at $t=1$ and $q=1$.  

The isomorphism $H^*\PeakConf{n} \cong 
\TypeACohomology{n} = H^* \TypeAConf{n}$
in Corollary~\ref{Sn-equivariant-surjection-and-iso} shows  $\bihilb_n(1,q)$ is the Poincar\'e series
for $H^* \TypeAConf{n}$:
$$
\bihilb_n(1,q)=\sum_{\sigma \in \symm{n}}
q^{n-\cyc(\sigma)}
=(1+q)(1+2q)(1+3q)\cdots(1+(n-1)q)
$$
that appeared in Theorem~\ref{thm:AB_VGpresentation}(iv).
It is a generating function for
{\it Stirling numbers of the 1st kind}, counting
permutations by their number of cycles.
Note that it implies an obvious recursion \eqref{obvious-recursion} below, whose second iteration \eqref{recursion-iteration}
gives the $t=1$ specialization of \eqref{nonequivariant-bihilb-recursion}:
\begin{align}
\label{obvious-recursion}
\bihilb_n(1,q)
&=(1+q(n-1)) \cdot \bihilb_{n-1}(1,q) \\
\label{recursion-iteration}
&=\bihilb_{n-1}(1,q) + q(n-1) \cdot (1+q(n-2)) \cdot \bihilb_{n-2}(1,q)).
\end{align}
At $t=1$, \eqref{nonequivariant-bihilb-gf} 
is equivalent to another well-known generating function
for Stirling numbers of the 1st kind:
$$
\sum_{n=0}^\infty \frac{x^n}{n!} 
\bihilb_n(1,q)=
\left[ 
\left( 1-z \right)^{-a} 
\right]_{\substack{a=q^{-1}\\ z=t^{1/2}qx}}.
$$

On the other hand, $\bihilb_n(t,1)$
is the Poincar\'e series for
$H^* \PeakConf{n}$.  The $q=1$
specializations of \eqref{nonequivariant-bihilb-expression}, \eqref{nonequivariant-bihilb-recursion}
\begin{align}
\bihilb_n(t,1)
&=\sum_{\sigma \in\symm{n}} t^{\frac{n-\odd(\sigma)}{2}}\\
\label{Sheffer-recursion}
\bihilb_n(t,1)
&=\bihilb_{n-1}(t,1) + (n-1)^2 \cdot \bihilb_{n-2}(t,1),
\end{align}
show that $\{ \bihilb_n(t,1) \}_{n=0,1,2,\ldots}$ are a rescaled version of
a family of orthogonal polynomials $\{ p_n(x) \}_{n=0,1,2,\ldots}$
with three-term recurrence
$p_n= xp_{n-1} + 
(n-1)^2 p_{n-2}$
and initial conditions
$p_{-1}=0,  p_0=1$
known as \emph{Sheffer polynomials}; see \url{https://oeis.org/A060524}.
\end{remark}

\subsection{Equivariant Hilbert series and branching rule}

To track the structures of $\kk\symm{n}$-modules $V$,
with $\kk$ of characteristic zero,
we will use the {\it Frobenius characteristic
map} (see Macdonald \cite[\S I.7]{Macdonald}, Stanley \cite[\S 7.18]{Stanley-EC2}) sending $V \mapsto \fchar(V)$, an element of the ring
of symmetric functions $\Lambda_{\kk}=\kk[p_1,p_2,\ldots]$,
defined by
$$
\fchar(V):= \frac{1}{n!} \sum_{\sigma \in \symm{n}}
\mathrm{trace}(\sigma|_V) \cdot p_{\lambda(\sigma)}
$$
with $p_{\lambda(\sigma)}:=p_1^{m_1(\sigma)} \cdots p_n^{m_n(\sigma)}$
if $\sigma$ has $m_i(\sigma)$ cycles of size $i$.  See Macdonald \cite[Ch. 1]{Macdonald} and Stanley \cite[Ch. 7]{Stanley-EC2}
for background on the ring of symmetric functions, including various $\kk$-bases.  In particular,  the
{\it Schur functions} $\{s_\lambda\}$ are the images
under $\fchar$ of the irreducible $\kk\symm{n}$-modules $\{\chi^\lambda\}$ indexed by partitions $\lambda \vdash n$.

Define the {\it equivariant} bigraded Hilbert series
$$
\symmbihilb_n:=
\symmbihilb_n(t,q):=\sum_{k,\ell} 
t^k q^\ell 
\fchar\left( \gr(\TypeBCohomology{n})^{\Z_2^n}_{2k,\ell} 
\right) \qquad \left( \in \Lambda[t,q] \right).
$$
Theorem~\ref{decomposition-of-peak-idempotent-reps-theorem} then immediately implies
$$
\symmbihilb_n(t^2,q)
=\sum_{\lambda \vdash n} L_\lambda 
\cdot t^{n-\odd(\lambda)} q^{n-\ell(\lambda)}
$$
where $L_\lambda:=\fchar(\Lie{\lambda})$.
It is a well-known problem of Thrall \cite{thrall} (see Schocker \cite{Schocker-higher-lie}) to decompose
$\Lie{\lambda}$ into $\kk\symm{n}$-irreducible modules, or equivalently, express $L_\lambda$ as a sum of Schur functions $s_\mu$.  However, in practice one can compute this for fairly large $n$, as the expression for $\Lie{\lambda}$ as an induced
character mentioned earlier \cite[Thm. 8.24]{reutenauer} leads to symmetric function
expressions for $L_\lambda$ involving power sums and/or plethysms; see, e.g.,
Hersh and Reiner \cite[eqns. (25),(32)]{HershR},
and also \eqref{plethysm-formula-for-Lie-character} below.

\begin{example} \rm
\label{ex:equivariant-data}
The equivariant refinement of \eqref{first-few-bigraded-hilbs} for $n=2,3,4$
is the following:
\begin{equation*}
\begin{tabular}{|c|l|l|}\hline
$n$ & $\bihilb_n(t,q)$ & $\symmbihilb_n(t,q)$ \\ \hline\hline
$2$ & $1$ & $L_{(1,1)}$\\
   &  & $=s_{(2)}$ \\
      &  & \\\hline
$3$ & $1+3tq+2tq^2$ & $L_{(1,1,1)}+tq L_{(2,1)}+tq^2 L_{(3)}$ \\
   &  & $=s_{(3)} + tq (s_{(2,1)} + s_{(1,1,1)}) + tq^2 s_{(2,1)}$\\
        &  & \\\hline
$4$ & $1+6tq+8tq^2+3t^2q^2+6t^2q^3$ & 
$L_{(1,1,1,1)}+tq L_{(2,1,1)}+tq^2 L_{(3,1)}+ 
  t^2q^2 L_{(2,2)}+t^2q^3 L_{(4)}$\\
 &  & $=s_{(4)} + tq (s_{(3,1)}+s_{(2,1,1)}) + tq^2 (s_{(3,1)}+s_{(2,2)}+s_{(2,1,1)})$ \\
 & & $ \qquad +t^2q^2 (s_{(2,2)}+s_{(1,1,1,1)}) 
 +t^2q^3 (s_{(3,1)}+s_{(2,1,1)})$ \\
       &  & \\\hline
\end{tabular}
\end{equation*}
\end{example}

Our next goal is to generalize recursion \eqref{nonequivariant-bihilb-recursion} by lifting it to a 
{\it branching rule}.  Theorem~\ref{restriction-rep-recursion} below describes how
the bigraded $\kk\symm{n}$-character
$\symmbihilb_n(t,q)$ behaves upon restriction from $\kk\symm{n}$ to $\kk\symm{n-1}$.  We will abbreviate the operations of
restriction and induction between $\kk\symm{n}$-modules and $\kk\symm{n-1}$-modules as $V \mapsto V\res$ and $U \mapsto U\ind$, omitting the relevant $n$;  we hope that it is always clear from context.

The Frobenius characteristic map $\fchar$ translates induction and restriction into operations
on symmetric functions expressed as
polynomials $f=f(p_1,p_2,\ldots)$ in the power sums $p_1,p_2,\ldots$, as follows:
\begin{itemize}
\item {\it induction} $(-)\ind$ multiplies dimension by $n$, and corresponds
to $f \mapsto p_1 f$,
\item {\it restriction} $(-)\res$
leaves dimension unchanged, and corresponds to $f \mapsto \frac{\partial}{\partial p_1} f$.
\end{itemize}

\begin{theorem}
\label{restriction-rep-recursion}
In $\Lambda_{\kk}[q,t]$, one has the following branching rule for $\symmbihilb_n:=\symmbihilb_n(t,q)$,
which is an equivariant lift of \eqref{nonequivariant-bihilb-recursion}:
\begin{equation}
\begin{aligned}
\frac{\partial \symmbihilb_n}{\partial p_1} 
  &=\symmbihilb_{n-1} + tq p_1 
     \left(
       1 + q p_1 \frac{\partial}{\partial p_1}
    \right) \symmbihilb_{n-2}.
\end{aligned}
\end{equation}
\end{theorem}
The proof of Theorem~\ref{restriction-rep-recursion} is a technical symmetric function calculation, which we defer to Section~\ref{sec:branching-rule-proof} of the Appendix.  Here we collect a few remarks on consequences of the theorem.

First, note it is
equivalent to this branching rule for the bigraded components $G^n_{k,\ell}:=\gr(\TypeBCohomology{n})^{\Z_2^n}_{k,\ell}$:
\begin{equation}
\label{branching-rephrased-on-bigraded-components}
G^n_{2k,\ell}
\res
= G^{n-1}_{2k,\ell} \,\, \oplus \,\, 
G^{n-2}_{2(k-1),\ell-1}\ind \,\, \oplus \,\, 
 G^{n-2}_{2(k-1),\ell-2} \res \ind \ind.
\end{equation}

Second, Theorem~\ref{restriction-rep-recursion} at $t=1$ lifts \eqref{recursion-iteration} for the Poincar\'e series $\bihilb_n(1,q)$ of $\TypeACohomology{n}=H^* \TypeAConf{n}$
$$
\frac{\partial \symmbihilb_n(1,q)}{\partial p_1} 
=\symmbihilb_{n-1}(1,q) + 
 q p_1 \left( 1 + q p_1 \frac{\partial}{\partial p_1} \right) \symmbihilb_{n-2}(1,q),
$$
agreeing with the iteration of a
branching rule of Sundaram \cite[Thm.~4.10(1)]{Sundaram} that
lifts \eqref{obvious-recursion}:
\begin{align*}
\frac{\partial \symmbihilb_n(1,q)}{\partial p_1} 
=\left( 1+ q p_1 \frac{\partial}{\partial p_1} \right) \symmbihilb_{n-1}(1,q) 
&=\symmbihilb_{n-1}(1,q) +  q p_1 \frac{\partial}{\partial p_1} \symmbihilb_{n-1}(1,q) \\
&=\symmbihilb_{n-1}(1,q) +  q p_1 \left(
 \symmbihilb_{n-2}(1,q) +  q p_1 \frac{\partial}{\partial p_1} \symmbihilb_{n-2}(1,q)
\right).
\end{align*}

Third, Theorem~\ref{restriction-rep-recursion} at $q=1$ simplifies,
giving an equivariant lift of the Sheffer polynomial recursion \eqref{Sheffer-recursion}:
\begin{align*}
\frac{\partial \symmbihilb_n(t,1)}{\partial p_1} 
&=\symmbihilb_{n-1}(t,1) + 
 t p_1 \left( 1 +  p_1 \frac{\partial}{\partial p_1} \right) \symmbihilb_{n-2}(t,1)\\
&=\symmbihilb_{n-1}(t,1) + 
 t p_1 \frac{\partial}{\partial p_1} \left( p_1 \symmbihilb_{n-2}(t,1) \right).
\end{align*}

\section{Simple Jordan elements}
\label{sec:jordanalgebra}
We discuss here an unexpected connection\footnote{The authors thank Sheila Sundaram for bringing this to their attention.} between the bigraded components $G_{k,\ell}:=\gr(\TypeBCohomology{n})^{\Z_2^n}_{k,\ell}$
and the space of {\it simple Jordan elements within the free associative algebra} on $n$ generators studied by Robbins \cite{robbins1971jordan} and by Calderbank, Hanlon and Sundaram \cite{calderbank1994representations}.
Recall from Section~\ref{Thrall-section} that if
$V=\C^n$ with $\C$-basis $x_1,\ldots,x_n$, then one can consider the tensor powers $T^d(V):=\C^{\otimes d}$ and the {\it tensor algebra}
 $$
 T(V)=\bigoplus_{d=0}^\infty T^d(V)
 = \C \oplus V \oplus (V \otimes V) \oplus (V\otimes V \otimes V) \oplus \cdots
 $$
 as a {\it free associative $\C$-algebra} on $x_1,\ldots,x_n$. 
 Consider the deformation of the Lie bracket by $\alpha \in \C$:
\[ [x,y]_{\alpha}:= xy - \alpha yx\]
and the $\C$-linear subspace of $T(V)$ spanned by
($\alpha$-deformed) left-bracketings
$$
 G_\alpha(V) :=V \oplus [V,V]_\alpha \oplus [[V,V]_\alpha,V]_\alpha \oplus \cdots
 $$
which consists of spans of elements such as these, where $v_i \in V$:
 $$
 v_1, \,\, 
 [v_1,v_2]_\alpha, \,\, 
 [[[v_1,v_2]_\alpha,v_3]_\alpha, \,\, 
 [[[[v_1,v_2]_\alpha,v_3]_\alpha,v_4]_\alpha,\ldots
 $$

 When $\alpha=1$, this $G_\alpha(V)$ is the {\it free Lie algebra} $\cLfree(V)$ of Section~\ref{Thrall-section}.  When $\alpha=-1$, the bracket $[-,-]_{-1}$ is the
 {\it Jordan bracket} on $T(V)$, and $G_\alpha(V)$ was called the {\it  space of simple Jordan elements}  by Robbins \cite{robbins1971jordan}. 
 
 \begin{remark} \rm
     We point out two subtleties here.  First, the Jordan bracket is non-associative and does not satisfy the Jacobi identity. As a result, $G_{-1}(V)$ ends up being a proper subset of what Robbins called the {\it space of Jordan elements}, which is closed under all Jordan brackets. Second, even the larger space of Jordan elements differs from the {\it free Jordan algebra} as studied by Kashuba and Mathieu \cite{KashubaMathieu}.
\end{remark}
 
 Following \cite{calderbank1994representations}, define $V_n(\alpha) \subset G_\alpha(V)$ to be the subspace spanned by {\it multilinear} left-bracketings $[ - ,- ]_{\alpha}$ of homogeneous degree $n$ on the generators $x_1, \ldots, x_n$, i.e. the span of all left-bracketings
\[ [[\cdots [ x_{\sigma(1)}, x_{\sigma(2)}], x_{\sigma(3)}], \cdots] x_{\sigma(n)}]\]
for $\sigma \in S_n$. Thus $V_n(1) \cong \lie_{(n)}$. When $\alpha = -1$, the following was proved by Robbins in \cite[\S 6, Thm. 7]{robbins1971jordan} and later in \cite[Thm 2.1]{calderbank1994representations} by Calderbank--Hanlon--Sundaram:

\begin{theorem}
\label{thm:jordanoddlie} 
There is a $\C\symm{n}$-module isomorphism
$
V_n(-1) \cong \bigoplus_{\substack{\lambda \vdash n \\ \lambda \  \odd }} \lie_{\lambda}.
$
\end{theorem}

Compare this with the assertion of  Theorem~\ref{decomposition-of-peak-idempotent-reps-theorem}, 
$$
G^n_{k,\ell}:=\gr(\TypeBCohomology{n})^{\Z_2^n}_{k,\ell} \cong 
\bigoplus_{\substack{\lambda \vdash n:\\ \ell(\lambda)=n-\ell\\ \odd(\lambda)=n-k }}
\Lie{\lambda}.
$$
This suggests a re-grouping of the
bigraded components of $\gr(\TypeBCohomology{n})^{\Z_2^n}$. Define the subspace $P_m^{(n)} \subset \gr(\TypeBCohomology{n})^{\Z_2^n}$
for $0 \leq m \leq \frac{n}{2}$  corresponding to degree $m$
in the $\{u_i\}$-variables:
  \[ 
  P^{(n)}_m :=
  \bigoplus_{\substack{(k,\ell):\\k-\ell=m}} G^{n}_{k,j} 
  \cong 
  \bigoplus_{\substack{|\lambda|=n \\ \even(\lambda) = \ell}} \lie_{\lambda}. 
  \]
Then Theorem~\ref{thm:jordanoddlie}
asserts that 
$$
V_n(-1)\cong P^{(n)}_0,
$$
which is the subspace involving no $\{u_i\}$,
only the variables $\{u^+_{ij},u^-_{ij}\}$.

Results of Calderbank, Hanlon and Sundaram \cite{calderbank1994representations} specialized to the case $p=2, \alpha=-1$ assert the following.
    \begin{prop}
    \label{CHS-parity-relations-for-Jordan-reps}
    The representations $V_n(-1)$ and $V_{n-1}(-1)$ are related as follows.
    \begin{itemize}
        \item [(a)]\cite[Prop 1.2]{calderbank1994representations}
    For odd $n$, one has 
     $ V_{n}(-1) \cong V_{n-1}(-1) \ind.$
     \item[(b)]\cite[Prop 1.3]{calderbank1994representations} For even $n$, one has  $ V_{n}(-1) \res
    \cong V_{n-1}(-1) \res \ind. $
     \end{itemize}
    \end{prop}
 Here is a generalization of both parts of Proposition~\ref{CHS-parity-relations-for-Jordan-reps}.
   
    \begin{prop}
    \label{thm:generalization-of-CHS-prop}
    Interpreting $P^{(n)}_m:=0$ if $m<0$,
    the representations $\{P^{(n)}_m \}$ satisfy
    the following:
    \begin{itemize}
        \item[(a)]
    For odd $n$, one has
    $
    P_m^{(n)} \cong P_m^{(n-1)} \ind. 
    $
 \item[(b)] For even $n$,  one has
    $
    P_m^{(n)} \res \cong \left( P_m^{(n-1)} \res \ 
    \oplus \  P_{m-1}^{(n-2)} \right) \ind.
    $
   \end{itemize}
\end{prop} 
The proof, involving symmetric function manipulations, is deferred to Section~\ref{sec:CHS_generalization-proof} of the Appendix.
 
\bibliographystyle{abbrv}
\bibliography{bibliography}
\appendix
\section{Symmetric function proofs of Theorem~\ref{restriction-rep-recursion}
and Proposition~\ref{thm:generalization-of-CHS-prop}} \label{appendix}

\subsection{Proof of Theorem~\ref{restriction-rep-recursion}}
\label{sec:branching-rule-proof}

Recall that Theorem \ref{restriction-rep-recursion} concerns this
{\it equivariant} bigraded Hilbert series
$$
\symmbihilb_n:=
\symmbihilb_n(t,q):=\sum_{k,\ell} 
t^k q^\ell 
\fchar\left( \gr(\TypeBCohomology{n})^{\Z_2^n}_{2k,\ell} 
\right)
=\sum_{\lambda \vdash n} L_\lambda 
\cdot t^{\frac{n-\odd(\lambda)}{2}} q^{n-\ell(\lambda)},
$$
which has been re-expressed on the far right above in terms of $L_\lambda:=\fchar(\Lie{\lambda})$, using
Theorem~\ref{decomposition-of-peak-idempotent-reps-theorem}.

\vskip.1in
\noindent
{\bf Theorem~\ref{restriction-rep-recursion}.}
{\it 
In $\Lambda_{\kk}[q,t]$, one has this branching rule for $\symmbihilb_n:=\symmbihilb_n(t,q)$
equivariantly lifting \eqref{nonequivariant-bihilb-recursion}:
\begin{equation*}
\frac{\partial \symmbihilb_n}{\partial p_1} 
=\symmbihilb_{n-1} + tq p_1 
     \left(
       1 + q p_1 \frac{\partial}{\partial p_1}
    \right) \symmbihilb_{n-2}.
\end{equation*}
}

\begin{proof}
Introduce this generating function  in $\Lambda[t,q][[x]]$, compiling the $\{ \symmbihilb_n(t,q)\}_{n=0,1,2,\ldots}$:
$$
\symmbihilb(x):=\sum_{n=0}^\infty x^n \symmbihilb_n(t,q)
= \sum_{\lambda} L_\lambda 
\cdot t^{\frac{n-\odd(\lambda)}{2}} q^{n-\ell(\lambda)} x^{|\lambda|}.
$$
Then one can rephrase the theorem as asserting that  
\begin{equation}
    \label{desired-partial-derivative-recurrence-on-gf}
\frac{\partial}{\partial p_1} \symmbihilb(x)
= x \symmbihilb(x) + 
  x^2 t q p_1 \left( 1 + 
               q p_1 \frac{\partial}{\partial p_1} 
            \right) \symmbihilb(x).
\end{equation}
Instead of working with $\symmbihilb(x)$, we will work with a rescaled version that is more convenient:
\begin{equation}
\label{rescaling-of-gf}
\begin{aligned}
\N(a,b,c)&:=\sum_{\lambda} L_\lambda \cdot a^{\even(\lambda)} b^{\odd(\lambda)} c^{|\lambda|},\\
\text{ so that }\quad\symmbihilb(x)&=\left[ \N(a,b,c)\right]_{\substack{a=q^{-1}\\ b=t^{-1/2}q^{-1}\\ c=t^{1/2}qx}}.
\end{aligned}
\end{equation}
We first work on deriving an expression for $\N(a,b,c)$ 
in terms of power sum symmetric functions $\{p_n\}$.
The expression of $\Lie{\lambda}$ as an induced character mentioned earlier \cite[Thm.~8.24]{reutenauer} gives this plethystic expression for $L_\lambda$ (see, e.g., \cite[eqns. (25)]{HershR}): if $\lambda=1^{m_1} 2^{m_2} \cdots$ with $m_i$ parts of size $i$, then
\begin{equation}
\label{plethysm-formula-for-Lie-character}
L_\lambda:=\fchar(\Lie{\lambda})
= h_{m_1}[\Lie{1}] \cdot h_{m_2}[\Lie{2}]\cdot h_{m_3}[\Lie{3}] \cdots
\end{equation}
where $\Lie{n}:=\Lie{(n)}$, with $\{ h_n\}$ the
{\it complete homogeneous} symmetric functions,
and $f \mapsto f[g]$ the {\it plethystic composition} \cite[\S I.8]{Macdonald} of symmetric functions.  Then one can compute as follows:
$$
\begin{aligned}
\N(a,b,c)
&=\sum_{\lambda} 
L_\lambda \cdot a^{\even(\lambda)} b^{\odd(\lambda)} c^{|\lambda|}\\
&=\sum_{\lambda=1^{m_1} 2^{m_2} \cdots} 
\left(
\prod_{\ell \geq 2 \text{ even}} 
  h_{m_\ell}[L_{\ell}] \cdot a^{m_\ell} c^{\ell m_\ell}
\prod_{\ell \geq 1 \text{ odd}} 
  h_{m_\ell}[L_{\ell}] \cdot b^{m_\ell} c^{\ell m_\ell}
\right)\\
&=(1+h_1+h_2+h_3+ \cdots)[bc L_1+ ac^2 L_2+bc^3 L_3 + ac^4 L_4 + \cdots]
\end{aligned}
$$
where the last equality used the identity
$$
h_n\left[\sum_{i \geq 1} f_i\right]=\sum_{m_1+m_2+\cdots=n} h_{m_1}[f_1] \cdot h_{m_2}[f_2] \cdots .
$$
Using the identity $1+h_1+h_2+\cdots = \exp\left(\sum_{m=1}^\infty \frac{p_m}{m}\right)$ one can continue the rewriting as
$$
\N(a,b,c)=\exp\left(\sum_{m=1}^\infty \frac{p_m}{m}\right)
\left[ b \sum_{\ell \geq 1 \text{ odd}} c^\ell L_\ell + a \sum_{\ell \geq 2 \text{ even}} c^\ell L_\ell \right].
$$
Taking $\log(-)$
of both sides gives
\begin{equation}
\label{first-log-of-N-manipulation}
\begin{aligned}
\log \N(a,b,c) &= 
\left(\sum_{m=1}^\infty \frac{p_m}{m}\right)
\left[ b \sum_{\ell \geq 1 \text{ odd}} c^\ell L_\ell + a \sum_{\ell \geq 2 \text{ even}} c^\ell L_\ell \right]\\
&=b \sum_{\ell \geq 1 \text{ odd}} \quad
\sum_{m=1}^\infty \frac{p_m[c^\ell L_\ell]}{m}
\quad + \quad 
a \sum_{\ell \geq 2 \text{ even}} \quad
\sum_{m=1}^\infty \frac{p_m[c^\ell L_\ell]}{m}.
\end{aligned}
\end{equation}
Recalling this expression \cite[Thm. 8.3]{reutenauer} for $L_\ell$
$$
L_\ell=\frac{1}{\ell} \sum_{d: d | \ell} \mu(d) \,\, p_d^{\ell/d},
$$
one can see that each inner sum on $m$ above looks as follows:
$$
\sum_{m=1}^\infty \frac{p_m[c^\ell L_\ell]}{m}
=\sum_{m=1}^\infty 
\frac{c^{\ell m}}{m \ell} \sum_{d: d | \ell} \mu(d) \,\, p_m\left[ p_d^{\ell/d}\right] 
=\sum_{m=1}^\infty 
\frac{c^{\ell m}}{m \ell} \sum_{d: d | \ell} \mu(d) \,\,  p_{dm}^{\ell/d} 
=\frac{c^\ell p_1^\ell}{\ell} + f(p_2,p_3,\ldots) .
$$
Here we have isolated the $d=m=1$ summand as
the only one involving $p_1$; all other summands with either $d$ or $m \geq 2$ sum to give a function $f(p_2,p_3,\ldots)$ purely in the power sums $p_2,p_3,\ldots$.
Consequently, 
$$
\begin{aligned}
\log \N(a,b,c) 
&=
b \sum_{\ell \geq 1 \text{ odd}} \frac{c^\ell p_1^\ell}{\ell}
\quad + \quad 
a \sum_{\ell \geq 2 \text{ even}} \frac{c^\ell p_1^\ell}{\ell}
\quad + g(p_2,p_3,\ldots)\\
&=
\frac{b}{2} \left( -\log(1-cp_1)+\log(1+cp_1) \right) +
\frac{a}{2} \left( -\log(1-cp_1)-\log(1+cp_1) \right) 
+ g(p_2,p_3,\ldots).\\
\end{aligned}
$$
Hence upon exponentiation one obtains an expression of this form:
\begin{equation}
\label{N-gf-with_p1s-separated-out}
\N(a,b,c) =
\left( 1+cp_1\right)^{\frac{b-a}{2}}
\left( 1-cp_1\right)^{\frac{-(b+a)}{2}}
G(p_2,p_3,\ldots).
\end{equation}
Now one can take a partial derivative with respect to $p_1$, and use product rule:
$$
\begin{aligned}
\frac{\partial}{\partial p_1}\N(a,b,c) 
&=
\frac{c(b-a)}{2}
\left( 1+cp_1\right)^{\frac{b-a}{2}-1}
\left( 1-cp_1\right)^{\frac{-(b+a)}{2}}
G(p_2,p_3,\ldots)\\
&\qquad +
\frac{c(b+a)}{2}
\left( 1+cp_1\right)^{\frac{b-a}{2}}
\left( 1-cp_1\right)^{\frac{-(b+a)}{2}-1}
G(p_2,p_3,\ldots)\\
&=\frac{c}{2}\left( \frac{b-a}{1+cp_1}
+ \frac{b+a}{1-cp_1}\right) \N(a,b,c)\\
&=\frac{c(b+acp_1)}{1-c^2p_1^2} \N(a,b,c).
\end{aligned}
$$
Translating back to $\symmbihilb(x)$ via  \eqref{rescaling-of-gf}, this gives
$$
\begin{aligned}
\frac{\partial}{\partial p_1} \symmbihilb(x)
&= \frac{x(1+t q p_1 x)}{1-t q^2 p_1^2 x^2} \symmbihilb(x)\\
(1-t q^2 p_1^2 x^2) \frac{\partial}{\partial p_1} \symmbihilb(x)
&= x(1+t q p_1 x) \symmbihilb(x)\\
\end{aligned}
$$
which after some algebra steps to isolate the
$x^0$ term on left and $x^1, x^2$ terms on right, leads to
$$
\frac{\partial}{\partial p_1} \symmbihilb(x)
= x \symmbihilb(x) + t q p_1 x^2 
\left(
 1+ q p_1 \frac{\partial}{\partial p_1}
\right) \symmbihilb(x),
$$
matching the desired recurrence \eqref{desired-partial-derivative-recurrence-on-gf}.
\end{proof}

\subsection{Proof of Proposition~\ref{thm:generalization-of-CHS-prop}}
\label{sec:CHS_generalization-proof}

Recall the definition of the $\kk \symm{n}$-modules $P^{(n)}_m$ for $0 \leq m \leq \frac{n}{2}$:
 \[ 
  P^{(n)}_m =
  \bigoplus_{\substack{(k,\ell):\\k-\ell=m}} G^{n}_{k,j} 
  \cong 
  \bigoplus_{\substack{|\lambda|=n \\ \even(\lambda) = \ell}} \lie_{\lambda}. 
  \]
\vskip.1in
\noindent
{\bf Proposition~\ref{thm:generalization-of-CHS-prop}.}
{\it
The representations $\{P^{(n)}_m \}$ satisfy
    the following.
    \begin{itemize}
        \item[(a)]
    For odd $n$, one has
    $
    P_m^n \cong P_m^{n-1} \ind. 
    $
 \item[(b)] For even $n$,  one has
    $
    P_m^{n} \res \cong \left( P_m^{n-1} \res \ 
    \oplus \  P_{m-1}^{n-2} \right) \ind.
    $
   \end{itemize}
}
\vskip.1in

\begin{proof}
For the sake of symmetric function manipulations, we abuse notation, and identify
$P^{(n)}_m$ with its Frobenius characteristic
image $\fchar(P^{(n)}_m)$ in $\Lambda_{\C}$.  The statement of the proposition then becomes
  \begin{itemize}
        \item[(a)]
    For odd $n$, one has
    $
    P_m^n = p_1 \cdot P_m^{n-1}. 
    $
 \item[(b)] For even $n$,  one has
    $
    \frac{\partial}{\partial p_1}P_m^{n} = p_1 \left( \frac{\partial}{\partial p_1} P_m^{n-1} +
     P_{m-1}^{n-2} \right).
    $
\end{itemize}

    First we show that (a) implies (b),   
    using \eqref{branching-rephrased-on-bigraded-components}:
$$
G^n_{2k,\ell}
\res
= G^{n-1}_{2k,\ell} \,\, \oplus \,\,  
G^{n-2}_{2(k-1),\ell-1}\ind \,\, \oplus \,\, 
 G^{n-2}_{2(k-1),\ell-2} \res \ind \ind.
$$
Taking the direct sum over all pairs $(k,\ell)$
with $2k-\ell=m$ gives this $\kk\symm{n}$-module isomorphism:
$$
P^{(n)}_m \res 
\cong
P^{(n-1)}_m \,\, \oplus \,\,
P^{(n-2)}_{m-1} \ind \,\, \oplus \,\,
P^{(n-2)}_m \res\ind\ind .
$$
Rephrasing this in terms of symmetric functions
gives the starting point for this derivation of (b):
\begin{align*}
\frac{\partial}{\partial p_1} P^{(n)}_m 
&= P^{(n-1)}_m + 
p_1 P^{(n-2)}_{m-1} + 
p_1 \cdot p_1 \frac{\partial}{\partial p_1} P^{(n-2)}_m & \\
&= P^{(n-1)}_m + 
p_1 P^{(n-2)}_{m-1} + 
p_1 \left( \frac{\partial}{\partial p_1}( p_1 P^{(n-2)}_m) -P^{(n-2)}_m\right) & \text{ using } 
p_1 \frac{\partial f}{\partial p_1}  = \frac{\partial}{\partial p_1} p_1 f-f\\
&= P^{(n-1)}_m + 
p_1 P^{(n-2)}_{m-1} + 
p_1 \frac{\partial}{\partial p_1}( p_1 P^{(n-2)}_m) -p_1 P^{(n-2)}_m  
& \\
&= P^{(n-1)}_m + 
p_1 P^{(n-2)}_{m-1} + 
p_1 \frac{\partial}{\partial p_1} P^{(n-1)}_m - P^{(n-1)}_m  
& \text{ using part }(a)\text{ twice}\\
&= 
p_1 \left( \frac{\partial}{\partial p_1} P^{(n-1)}_m + P^{(n-2)}_{m-1} \right).
&\\
\end{align*}
 
  For (a), to show $n$ odd implies $P^{(n)}_m=p_1 \cdot P^{(n-1)}_m$,
  we rephrase it using this generating function
  $$
\PPP(a,c):=\left[\N(a,b,c)\right]_{b=1}=\sum_\lambda L_\lambda \cdot a^{\even(\lambda)} c^{|\lambda|}
  =\sum_{n=0}^\infty c^n (P^{(n)}_0 + a P^{(n)}_1+ a^2 P^{(n)}_2+\cdots).
  $$
 Then the assertion of (a) is equivalent to showing that one has a factorization
 $$
 \PPP(a,c)=(1+cp_1) \sum_{\substack{n \geq 0\\\text{even}}} c^n (P^{(n)}_0 + a P^{(n)}_1+ a^2 P^{(n)}_2+\cdots),
 $$
 or equivalently, within the ring $\Lambda[[a,c]]$, the quotient $\frac{\PPP(a,c)}{(1+cp_1)}$ lies
 in $\Lambda[[a,c^2]]$, meaning that it only contains {\it even powers of $c$}.  By the result of Calderbank, Hanlon, Sundaram stated as Proposition~\ref{CHS-parity-relations-for-Jordan-reps}(a) above, we know this holds upon setting $a=0$, since their result shows 
 $P^{(n)}_0=p_1 \cdot P^{(n-1)}_0$ for $n$ even, and hence
 $$
 \PPP(0,c)=\sum_{n=0}^\infty c^n P^{(n)}_0 
 =
 (1+cp_1) \sum_{\substack{n \geq 0\\\text{even}}}^\infty c^{n} P^{(n)}_0.
 $$
 On the other hand, exponentiating
 \eqref{first-log-of-N-manipulation} shows that
 \begin{equation}
 \label{exponential-with-separated-a-b}
 \begin{aligned}
 \N(a,b,c)&=
 \exp\left(
 b \sum_{\substack{\ell \geq 1\\ \text{odd}}} \quad
\sum_{m=1}^\infty \frac{p_m[c^\ell L_\ell]}{m}
 \right)
\exp\left(
 a \sum_{\substack{\ell \geq 2\\\text{even}}} \quad
\sum_{m=1}^\infty \frac{p_m[c^\ell L_\ell]}{m}
 \right)\\
 &=
 \N(0,b,c) \cdot \N(a,0,c)
 \end{aligned}
\end{equation}
where we note that $\N(a,0,c)$ lies in $\Lambda[[a,c^2]]$.  Now setting
setting $b=1$ in \eqref{exponential-with-separated-a-b} gives
$$
\begin{aligned}
\PPP(a,c)&=\PPP(0,c) \cdot \N(a,0,c)\\
&= (1+cp_1) \cdot \left( \sum_{m=0}^\infty  c^{2m}  P^{(2m)}_0 \right) \cdot \N(a,0,c)
\end{aligned}
$$ 
which achieves the goal of exhibiting 
$\frac{\PPP(a,c)}{(1+cp_1)}$ as an element of
 in $\Lambda[[a,c^2]]$.
\end{proof}


\end{document}